\documentclass[leqno]{siamart1116}
\usepackage{amssymb}
\usepackage{graphicx} 
\usepackage{xspace}
\usepackage{multirow}
\usepackage{bbm}
\usepackage{algpseudocode}
\usepackage[numbers]{natbib}
\let\cite=\citet
\usepackage{enumerate}

\newcommand\FIGS{.}
\usepackage{mydef}

\begin{document}
\newcommand\footnotemarkfromtitle[1]{%
\renewcommand{\thefootnote}{\fnsymbol{footnote}}%
\footnotemark[#1]%
\renewcommand{\thefootnote}{\arabic{footnote}}}

\title{Invariant domain preserving  \\
 discretization-independent schemes \\ and convex limiting for hyperbolic systems\footnotemark[1]}

\author{Jean-Luc Guermond\footnotemark[2]
\and Bojan Popov\footnotemark[2], \and Ignacio Tomas\footnotemark[2]}
\date{Draft version \today}

\maketitle

\renewcommand{\thefootnote}{\fnsymbol{footnote}} \footnotetext[1]{
  This material is based upon work supported in part by the National
  Science Foundation grants DMS-1217262, by the Air
  Force Office of Scientific Research, USAF, under grant/contract
  number FA99550-12-0358, and by the Army Research Office under grant/contract
  number W911NF-15-1-0517.  Draft
  version, \today} \footnotetext[2]{Department of Mathematics, Texas
  A\&M University 3368 TAMU, College Station, TX 77843, USA.}
\renewcommand{\thefootnote}{\arabic{footnote}}

\begin{abstract} 
  We introduce an approximation technique for nonlinear hyperbolic
  systems with sources that is invariant domain preserving. The method
  is discretization-independent provided elementary symmetry and
  skew-symmetry properties are satisfied by the scheme. The method is
  formally first-order accurate in space. A series of higher-order
  methods is also introduced. When these methods violate the invariant
  domain properties, they are corrected by a limiting technique that
  we call \emph{convex limiting}. After limiting, the resulting
  methods satisfy all the invariant domain properties that are imposed
  by the user (see Theorem~\ref{Thm:limiting_and_invariant_domain}). A
  key novelty is that the bounds that are enforced on the solution at
  each time step are necessarily satisfied by the low-order
  approximation.
\end{abstract}



\pagestyle{myheadings} \thispagestyle{plain} \markboth{J.-L. GUERMOND,
  B. POPOV, I. TOMAS}{Invariant domain preserving
approximation of hyperbolic systems}

\section{Introduction} \label{Sec:introduction} The present paper is
concerned with the approximation of hyperbolic systems in conservation
form with a source term:
\begin{equation}
 \label{def:hyperbolic_system} 
 \begin{cases} \partial_t \bu + \DIV \polf(\bu)=\bS(\bu), 
\quad \mbox{for}\, (\bx,t)\in \Dom\CROSS\Real_+,\\
\bu(\bx,0) = \bu_0(\bx), \quad \mbox{for}\, \bx\in \Real^d.
\end{cases}
\end{equation}
The space dimension $d$ is arbitrary. The dependent variable
$\bu$ takes values in $\Real^m$ and the flux $\polf$ takes values in
$(\Real^m)^d$.  In this paper $\bu$ is considered as a column vector
$\bu=(u_1,\ldots,u_m)\tr$. The flux is a matrix with entries
$\polf_{ij}(\bu)$, $1\le i\le m$, $1\le j\le d$ and $\DIV\polf$ is a
column vector with entries
$(\DIV\polf)_i= \sum_{1\le j\le d}\partial_{x_j} \polf_{ij}$. For any
$\bn=(n_1\ldots,n_d)\tr\in \Real^d$, we denote $\polf(\bu)\bn$ the
column vector with entries $\sum_{1\le l\le d} \polf_{il}(\bu) n_l$,
where $i\in\intset{1}{m}$.  To simplify questions regarding boundary
conditions, we assume that either periodic boundary conditions are
enforced, or the initial data is compactly supported or constant
outside a compact set.  In both cases we denote by $\Dom \subseteq \Real^d$
the spatial domain where the approximation is constructed. The domain $\Dom$ is
the $d$-torus in the case of periodic boundary conditions. In the case
of the Cauchy problem, $\Dom$ is a compact, polygonal portion of
$\Real^d$ large enough so that the domain of influence of $\bu_0$ is
always included in $\Dom$ over the entire duration of the simulation.

The objective of the paper is to generalize the techniques that was
introduced in \cite{GuerNazPopTom2017} for the approximation of the
compressible Euler equations using continuous finite elements. We want
to present an approximation technique that is almost discretization
independent and works with any hyperbolic system with source term,
under some mild assumptions on the source.  The formalism encompasses
finite volumes, continuous finite elements and discontinuous finite
elements.  The method is formally second-order or higher-order in
space and can be made (at least) fourth-order accurate in time by
using explicit Runge Kutta SSP methods.  The key ingredients of the
method are as follows: (i) A low-order invariant domain preserving
approximation technique using a graph viscosity. (The viscosity is
based on the connectivity graph of the degrees of freedom of the
method. One viscosity coefficient is computed on every edge of the
graph.) (ii) A high-order approximation technique. (The method may not
be fully entropy consistent and may step out of the local invariant
domain); (iii) A \emph{convex limiting} technique with guaranteed
bounds. (The bounds in question are obtained by computing auxiliary
states on every edge of the connectivity graph. The convex limiting
method works for any quasiconcave functional, \ie it is possible to
limit any quasiconcave functional of the approximate solution.)

The paper is organized as follows.  We recall elementary properties of
the hyperbolic system~\eqref{def:hyperbolic_system} in
\S\ref{Sec:Preliminaries}. The theory for the low-order method is
explained in \S\ref{sec:firstorder}. The main result of this section
is Theorem~\ref{Thm:UL_is_invariant}. The auxiliary states, which play
a key role in the convex limiting technique are defined
in~\eqref{def_barstates}. The method is illustrated in the context of
finite volumes, continuous finite elements, and discontinuous finite
elements in \S\ref{Sec:examples_discretization}. A brief overview of
explicit Runge Kutta Strong Stability methods is made in
\S\ref{Sec:SSP_overview}. The key result of this section is a
reformulation of the Shu-Osher Theorem~\ref{Th:Shu_Osher} which does
not involve any norm. We show therein that only convexity matters.  It
seems that the result, as reformulated, is not well known in the
literature. We show in \S\ref{Sec:High_order_methods} how higher-order
schemes can be constructed. These methods are not necessarily
invariant domain preserving. In passing we revisit an idea initially
proposed by \cite[Eq. (12)]{Jameson_aiaa_1981} which consists of
constructing a second-order graph viscosity by using a smoothness
indicator.  In Theorem~\ref{Thm:max_principle_smoothness} we prove
that a high-order scheme based on the smoothness indicator of a
conserved scalar component of the system does indeed preserve the
bounds (for that component) naturally satisfied by the first-order
method. In Theorem~\ref{Thm:Greedy_viscosity} we present another
invariant domain preserving result for one scalar component of the
conserved variables, but in this case the graph viscosity is computed
by using a gap estimate (see Lemma~\ref{Lem:lower_upper_bounds}) 
instead of a smoothness indicator. To the best of our knowledge, it seems
that both results are original in the context of hyperbolic
systems. The convex limiting technique is presented in
\S\ref{Sec:convex_limiting}, the key results of this section are
Lemma~\ref{Lem:NaturalBounds}, Lemma~\ref{Lem:quasiconcave} and
Theorem~\ref{Thm:compute_lij}.  All these results are
recapitulated into Theorem~\ref{Thm:limiting_and_invariant_domain},
which in some sense summarizes the content of the present paper.
The idea of using the auxiliary states~\eqref{def_barstates} and
convex limiting has originally been proposed in
\cite{GuerNazPopTom2017} for the Euler equations. The proposed
generalization to general hyperbolic systems with source term for generic
discretizations seems to be new.

Computations illustrating the performance of the abstract
results stated in the paper can be found in
\cite{GuePo2016,GuerNazPopTom2017} for the compressible Euler
equations, and in
\cite{Azerad_Guermond_Popov_2017,Guermond_Quezada_Popov_Kees_Farthing_2018}
for the shallow water equations.

\section{Preliminaries} \label{Sec:Preliminaries}
We recall in this section key properties about the
system~\eqref{def:hyperbolic_system} that will be used repeatedly in the
paper. The reader who is familiar with hyperbolic systems with source
terms, Riemann problems, and invariant sets is invited to jump to
\S\ref{sec:firstorder}.

\subsection{Riemann problem space average and maximum wave speed}
\label{Sec:Riemann_pb_average_max_wavespeed}
We consider \eqref{def:hyperbolic_system} without source term in this
subsection, \ie $\bS(\bu) = 0$. Instead of trying to give a precise meaning to
the solutions of~\eqref{def:hyperbolic_system}, which is either a very technical 
task or a completely open problem, we instead assume that there
is a clear notion of solution for the Riemann problem. 
That is to say
we assume that there exists an nonempty admissible set
$\calA \subset\Real^{m} $ such that for any pair of states
$(\bu_L,\bu_R)\in \calA\CROSS \calA$ and any unit vector $\bn$ in
$\Real^d$, the following one-dimensional Riemann problem
\begin{equation}
 \label{def:Riemann_problem} 
  \partial_t \bv + \partial_x (\polf(\bv)\bn)=0, 
\quad  (x,t)\in \Real\CROSS\Real_+,\qquad 
\bv(x,0) = \begin{cases} \bv_L, & \text{if $x<0$} \\ \bv_R,  & \text{if $x>0$}, \end{cases}
\end{equation}
has a unique (entropy satisfying) self-similar solution denoted by
$\bv(\bn,\bv_L,\bv_R, \xi)$, where $\xi = \frac{x}{t}$ is the
self-similarity parameter, see for instance \cite{Lax_1957_II,Toro_2009}. The key
result that we are going to use in this paper is that there exists a
maximum wave speed henceforth denoted
$\lambda_{\max}(\bn,\bv_L,\bv_R)$ such that
$\bv(\bn,\bv_L,\bv_R, \xi) = \bv_L$ if
$\xi\le -\lambda_{\max}(\bn,\bv_L,\bv_R)$ and
$\bv(\bn,\bv_L,\bv_R, \xi) = \bv_R$ if
$\xi\ge \lambda_{\max}(\bn,\bv_L,\bv_R)$. We assume that
$\lambda_{\max}(\bn,\bv_L,\bv_R)$ can be estimated from above
efficiently; for instance, we refer the reader to
\cite{Guermond_Popov_Fast_Riemann_2016} where guaranteed upper bounds
on the maximum wave speed are given for the Euler equations with the
co-volume equation of state.  The following elementary result, which
we are going to invoke repeatedly, is an important
consequence of the finite speed of propagation assumption:

\begin{lemma}[Average over the Riemann
  fan] \label{Lem:elemetary_Riemann_pb} Let $(\eta,\bq)$ be an entropy
  pair for the system~\eqref{def:hyperbolic_system}.  Let
  $ \overline\bv(t,\bn,\bv_L,\bv_R) :=\int_{-\frac12}^{\frac12}
  \bv(\bn,\bv_L,\bv_R,\xi) \diff x $
  be the average of the Riemann solution over the Riemann fan at time
  $t$. Assume that $t\lambda_{\max}(\bn,\bv_L,\bv_R) \le \frac12$,
  then the following holds true:
\begin{align}
  \overline\bv(t,\bn,\bv_L,\bv_R)  &= \frac{1}{2}(\bv_L+\bv_R) 
  - t\big(\polf(\bv_R)\bn - \polf(\bv_L)\bn\big).
\label{elemetary_Riemann_pb}\\
\label{entropy_elemetary_Riemann_pb}
\eta(\overline\bv(t,\bn,\bv_L,\bv_R) )
&\le \tfrac{1}{2}(\eta(\bv_L)+\eta(\bv_R)) 
  - t(\bq(\bv_R)\SCAL\bn - \bq(\bv_L)\SCAL\bn).
\end{align}
\end{lemma}

\subsection{Invariant sets and invariant domains}\label{sec:invarSets}
We introduce in this section the notions of invariant sets and
invariant domains. Our definitions are slightly different from those
in
\cite{Chueh_Conley_Smoller_1977,Hoff_1985,Smoller1994,Frid_2001}. We
associate invariant sets with solutions of Riemann problems
and define invariant domains only for an approximation process; our definition 
has some similarities with Eq.~(2.14) in \cite{Zhang_JCP_2011}.

\begin{definition}[Invariant set] \label{Def:invariant_set} We say
  that a set $\calB\subset \calA\subset \Real^m$ is invariant for
  \eqref{def:hyperbolic_system} if $\calB$ is convex and for any pair
  $(\bu_L,\bu_R)\in \calB\CROSS \calB$, any unit vector $\bn\in \Real^d$, and
  any $t>0$ such that $t\lambda_{\max}(\bn,\bv_L,\bv_R) \le \frac12$,
  the average of the entropy solution of the Riemann problem
  \eqref{def:Riemann_problem} over the Riemann fan, say
  $\overline\bv(t,\bn,\bv_L,\bv_R)$, remains in $\calB$ and if there
  exists $\tau_0>0$ such that for any $\bsfU\in \calB$ and any
  $\tau\le \tau_0$ the quantity $\bsfU + \tau \bS(\bsfU)$ is in $\calB$.
\end{definition}

We now introduce the notion of invariant domain for an approximation
process.  Let $\Nglob$ be a positive natural number and let
$\bR_h : (\Real^m)^\Nglob \to (\Real^m)^\Nglob$ be a mapping over
$(\Real^m)^\Nglob$. Henceforth we abuse the language by saying that a
member of $(\Real^m)^\Nglob$, say $\bsfU = (\bsfU_1,\ldots,\bsfU_\Nglob)$, is
in the set $\calB\subset \Real^m$ to actually mean that
$\bsfU_i \in \calB$ for all $i \in \intset{1}{\Nglob}$.

\begin{definition}[Invariant domain]\label{Def:invariant_domain} A
convex invariant set $\calB \subset \calA\subset \Real^m$ is said to be
an invariant domain for the mapping $\bR_h : (\Real^m)^\Nglob \to (\Real^m)^\Nglob$ if and only if for any
state $\bsfU$ in $\calB$, the state $\bR_h(\bsfU)$ is also in $\calB$.
\end{definition}

For scalar conservation equations the notions of invariant sets and
invariant domains are closely related to the notion of maximum principle. In the
case of nonlinear hyperbolic systems, the maximum principle property does not
apply and must be replaced by the notion of an invariant domain. To the
best of our knowledge, the definition of invariant sets for the
Riemann problem was introduced in \cite{Nishida_1968}, and the general
theory of positively invariant regions was developed in
\cite{Chueh_Conley_Smoller_1977}. The analysis and development of
numerical methods preserving invariant regions was considered in
\cite{Hoff_1979,Hoff_1985,Frid_2001}. The objective of this paper is
to generalize the invariant domain preserving method originally developed
in \cite{GuePo2016} and the (invariant domain preserving) convex limiting technique introduced in
\cite{GuerNazPopTom2017}.

\begin{remark}[Siff source terms] The assumption that there exists a
  uniform $\tau_0$ so that $\calB+\tau\bS(\calB)\subset \calB$ for all
  $\tau\in [0,\tau_0]$ is not reasonable for hyperbolic systems with stiff source terms since
  it imposes a very severe restriction on the time step. In this case
  other strategies must be adopted. We are going to restrict ourselves
  in the present paper to source terms that are moderately stiff in
  the sense of Definition~\ref{Def:invariant_set}, and we postpone the
  extension of the present work to systems with stiff source terms to a future
  publication.
\end{remark}

\subsection{Examples}   We briefly go over some examples of systems with
source terms and show that the proposed definition for invariant sets is
meaningful/useful.

\subsubsection{Euler + co-volume EOS} For the compressible Euler
equations with covolume of state the dependent variable is
$\bu=(\rho,\bbm,E)\tr$, where $\rho$ is the density, $\bbm$ is the
momentum, and $E$ is the total energy. The flux is
$\polf(\bu)=(\rho\bv, \bbm\otimes \bv + p \polI, \bv(E + p))\tr$ where
$\bv := \bbm/\rho$ and the pressure is given by the equation of state
$p ( 1 - b \rho ) = ( \gamma - 1 ) e \rho$. The constant $b\ge 0$ is
called the covolume and $\gamma > 1$ is the ratio of specific
heats. We have $\calA:=\{\bu\st 1\ge 1- b \rho \ge 0, e(\bu)\ge 0\}$ and
it is shown in \cite{Guermond_Popov_Fast_Riemann_2016} that
$\calB:=\{\bu\st 1\ge 1- b \rho \ge 0, e(\bu)\ge 0, \Phi(\bu)\ge \Phi_0\}$ is an invariant set for 
any $\Phi_0\in \Real$,
where $e(\bu):= E/\rho-\frac12 \bv^2$ is the specific internal energy, and 
$\Phi(\bu)$ is the specific physical entropy.
In this paper we call internal energy the quantity
$\varepsilon(\bu) := \rho e(\bu)$.

\subsubsection{Shallow water}
Saint-Venant's shallow water model describes the time and space
evolution of a body of water evolving in time under the action of
gravity assuming that the deformations of the free surface are small
compared to the water elevation and the bottom topography $z$ varies
slowly. The dependent variable is $\bu=(\waterh,\bq)\tr$, where
$\waterh$ is the water height and $\bq$ is the flow rate in the
direction parallel to the bottom. The flux is
$\polf(\bu)=(\bq,\bq {\otimes}\bv + \frac12 g \waterh^2 \polI)\tr$, where
$\bv:=\bq/\waterh$ and $g$ is the gravity constant. The source
including the influence of the topography and Manning's friction law is
$\bS(\bu) = (0, g \waterh\GRAD z - g n^2 \waterh^{-\gamma} \bq
\|\bv\|_{\ell^2})$,
where $\Manning$ is Manning's roughness coefficient, and $\gamma$ is
an experimental parameter often close to $\frac43$.

It is well-known that $\calA=\calB:=\{\bu \st \waterh\ge 0\}$ is an invariant
set for the system without source term. Let $\bu\in \calB$ and $\tau>0$,
then
$\bu + \tau \bS(\bu) = (\waterh, \bq + \tau (g \waterh\GRAD z - g n^2
\waterh^{-\gamma} \bq \|\bv\|_{\ell^2}))\tr$,
and it is clear that $\bu + \tau \bS(\bu)\in \calB$ for any $\tau\ge 0$
because $h\ge 0$ by definition. Hence $\calB$ is an invariant set
according to Definition~\ref{Def:invariant_set} with $\tau_0=\infty$.

\subsubsection{ZND model}
We now consider the Zel'dovich--von Neumann--D\"oring model for compressible
reacting flows. The dependent variable is
$\bu=(\rho_1,\rho_2,\bbm, E)\tr$, where $\rho_1$ is the density of
the burned gas (fuel), $\rho_2$ is the density of the unburned gas, $\bbm$ is
the momentum of the mixture, and $E$ is the total energy.  The flux is
$\polf(\bu)=(\rho_1\bv,\rho_2\bv, \bbm\otimes \bv + p \polI, \bv(E + p))\tr$
where $\bv := \bbm/(\rho_1+\rho_2)$ and the pressure is given by an
appropriate equation of state; for instance, for ideal polytropic
gases it is common to adopt the so called $\gamma$-law,
$p = (\gamma-1)(E-\frac12 \rho \bv^2 - q_0 \rho_2)$, where $q_0$ is
the specific energy of the unburned gas. Denoting by $T:= p/(\rho_1+\rho_2)$,
the source term is
$\bS(\bu) = (\kappa(T)\rho_2,-\kappa(T)\rho_2,\bzero, 0)\tr$, where
$\kappa(T) = \kappa_0 \text{e}^{-T_0/T}$, where $\kappa_0\ge 0$ is the
reaction rate constant and $T_0$ is the ignition temperature (up to
multiplication by the gas constant $R$).

Denoting $\rho:=\rho_1+\rho_2$ and setting
$e(\bu):= (E-\frac12 \rho \bv^2 - q_0 \rho_2)/\rho$, it can be shown
that $\calA=\calB:=\{\bu \st \rho_1\ge 0, \rho_2\ge 0, e(\bu)\ge 0\}$ is an
invariant set for the homogeneous system, \ie when $\bS\equiv \bzero$.
One can convince oneself that this is indeed true by realizing that
when $\bS\equiv \bzero$, upon denoting $E':= E - q_0 \rho_2$, the
dependent variable $(\rho,\bbm,E')$ solves the compressible Euler
equations, and it is well-known that
$\{\bu \st \rho \ge 0, E'-\frac12\rho \bv^2 \ge 0\}$ is an invariant
set. 

Now let us establish that  for any
$\bu\in \calB$ and any $\tau\le \tau_0:=\kappa_0^{-1}$, the quantity
$\bu + \tau \bS(\bu)$ is in $\calB$. Let $\bu\in \calB$ and let $\tau\ge 0$,
then
$\bu + \tau \bS(\bu) = (\rho_1+ \tau \kappa(T)\rho_2,\rho_2 -\tau
\kappa(T)\rho_2,\bbm, E)\tr$.
Since $T:= (\gamma-1) e(\bu)\ge 0$, $\rho_1\ge0$,
$\rho_2\ge0$, and $\tau\ge 0$, it is clear that $\rho_1+ \tau \kappa(T)\rho_2\ge 0$.
Moreover,
$\rho_2 -\tau \kappa(T)\rho_2 =\rho_2(1 -\tau \kappa(T)) \ge\rho_2(1
-\tau \kappa_0)$;
hence $\rho_2 -\tau \kappa(T)\rho_2\ge 0$ provided
$\tau\le \tau_0:=\kappa_0^{-1}$. Finally, observing that
$\rho:=\rho_1 +\tau \kappa(T)\rho_2 + \rho_2 -\tau \kappa(T)\rho_2 >0$,
we have
$\rho e(\bu + \tau \bS(\bu)) = E - \frac12 \rho \bv^2 - q_0\rho_2(1 -\tau
\kappa(T)) \ge E - \frac12 \rho \bv^2 - q_0\rho_2 = e(\bu) \ge 0$, thereby proving that 
$\bu + \tau \bS(\bu)\in \calB$.

\subsubsection{Euler equations with sources} In some astrophysical
applications one may want to solve the compressible Euler equations
with Coriolis effects, gravitation effects and some heat transfer
effects due to the emission and/or absorption of radiation. The
dependent variables and the flux are the same as those of Euler's
equations, but the source term is
$(0,-2 \bOmega\CROSS \bbm -\rho\GRAD \Phi,-\bbm\SCAL\GRAD \Phi+\rho
H)\tr$,
where $\bOmega$ is the angular velocity of the system, $\Phi$ some
given gravitation potential, and $\rho H$ is a term that aggregates
all the cooling and heating effects. One invariant domain for the
homogeneous system is
$\calA=\calB:=\{\bu\st \rho\ge 0, e(\bu)\ge 0\}$.  Let $\bu \in \calB$
and $\tau\ge 0$. Then
$\bu + \tau \bS(\bu) = (\rho, \bbm -2 \tau \bOmega\CROSS \bbm -\tau
\rho\GRAD\Phi, E-\tau \bbm\SCAL\GRAD \Phi+\tau \rho H)\tr$.
The density of the state $\bu + \tau \bS(\bu)$ is $\rho$, which is
nonnegative by definition. The specific internal energy of the state
$\bu + \tau \bS(\bu)$ is bounded from below as follows:
$e(\bu + \tau \bS(\bu)) \ge e(\bu) - \dt^2(4\bOmega^2 \bv^2 +(\GRAD
\Phi)^2) + \dt H$.
For instance, for a $\gamma$-law equation of state, we have
$e= c(\bu)^2/(\gamma(\gamma-1))$, where $c(\bu)$ is the speed of sound
and $e(\bu + \tau \bS(\bu)) \ge 0$, if
$\tau^2 \le \frac{c(\bu)^2}{2\gamma(\gamma-1)(4\bOmega^2 \bv^2 +(\GRAD
  \Phi)^2)}$
and $\tau H \ge -\frac{e(\bu)}{2} $. If $\GRAD \Phi =\bg$ is a
constant, then the first condition is satisfied if
$\tau^2 \le \frac{c(\bu)^2}{2\gamma(\gamma-1)(4\bOmega^2 M(\bu)^2
  c(\bu)^2 +\|\bg\|_{\ell^2}^2)}$
where $ M(\bu)$ is the local Mach number; assuming that one can
establish that $M(\bu)\le M_{\max}$ uniformly w.r.t. $\bu$, and
$\inf c(\bu) \ge c_{\min}>0$, which is required for hyperbolicity to
hold, then the first condition holds if
$\tau \le \frac{c_{\min}}{(2\gamma(\gamma-1)(4\bOmega^2 M_{\max}^2
  c_{\min}^2 +\|\bg\|_{\ell^2}^2))^{\frac12}}$.
One can also verify that many astrophysical models for the heat
transfer effect lead to existence of $\tau_0'>0$ such that
$\tau H \ge -\frac{e(\bu)}{2} $ for all $\dt\le \tau_0'$; the details
are left to the reader.

\section{Abstract low-order approximation}\label{sec:firstorder} In this section
we describe a generic invariant domain preserving technique for
approximating solutions to \eqref{def:hyperbolic_system}.  In order to
stay general we present the method without referring to any particular
discretization technique, we are going to use instead the graph
theoretic language to describe the method. The method is illustrated
with finite volumes, continuous elements, and discontinuous elements
in \S\ref{Sec:examples_discretization}.

\subsection{The low-order scheme} \label{sec:abstract_low_order} To
identify properly the time stepping technique, we denote by $t^n$ the
current time, $n\in \polN$, and we denote by $\dt$ the current time
step size; that is $t^{n+1}:=t^n+\dt$. We now address the approximation in
space by assuming that we have at hand some finite-dimensional vector
space $X_h$ with some basis $\{\varphi\}_{i\in \calV}$, where
$\varphi_i^n:\Dom\to \Real$, for all $i\in\calV$.  We introduce
$\bX_h^n:=(X_h)^m$ and denote the approximation of $\bu(\cdot,t^n)$ in
$\bX_h$ by $\bu_h^n:=\sum_{i\in \calV} \bsfU_i^n\varphi_i$, with
$\bsfU_i^n\in \calA\subset\Real^m$ for all $i\in \Real^m$. We do not need to know
for the time being what the basis functions
$\{\varphi_i\}_{i\in \calV}$ are, but we assume that this setting
allows us to construct an inviscid (very accurate) approximation of
$\bu(\cdot,t^{n+1})$ in $\bX_h$, denoted
$\bu_h\upGn:=\sum_{i\in \calV} \bsfU_i\upGnp\varphi_i$,  as follows:
\begin{align}
\label{def_Galerkin_scheme}
  \frac{m_i}{\dt}(\bsfU_i\upGnp-\bsfU_i^n)
  + \sum_{j\in \calI(i)} \polf(\bsfU_j^n)\bc_{ij} = m_i \bS(\bsfU_i^n),
\end{align}
for any $i\in\calV$, where the numbers $\{m_i\}_{i \in \vertind}$ are
assumed to be positive. Note here that we use the forward Euler time
stepping.  Higher-order time stepping schemes will be considered in
\S\ref{Sec:SSP_overview}.  For any $i\in\calV$, the set $\calI(i)$ is
a (small) subset of $\calV$, which we call stencil at $i$ or adjacency
list at $i$.  We assume that the following property holds:
$j\in\calI(i)$ iff $i\in\calI(j)$.  We assume also that the
$\Real^d$-valued matrix $\{\bc_{ij}\}_{i\in\calV,j\in\calI(i)}$ has
the following properties:
\begin{align}\label{cijprop}
\bc_{ij} = - \bc_{ji}\quad \text{and} \quad \sum_{j\in \calI(i)} \bc_{ij} = \bzero.
\end{align}
The quantities $m_i$, $\{\bc_{ij}\}_{j\in \calI(i)}$, and the set
$\calI(i)$ depend on the discretization that is chosen.  We are going
to be more specific in \S\ref{Sec:examples_discretization}. We think of
\eqref{def_Galerkin_scheme} as the ``centered'' consistent
approximation of \eqref{def:hyperbolic_system} that delivers optimal
accuracy (for the considered setting) for smooth solutions.

Notice that the above construction allows us to introduce an
undirected finite graph $(\calV,\calE)$, where for any pair
$(i,j)\in\calV\CROSS\calV$, we say that $(i,j)$ is an edge of the
graph, \ie $(i,j)\in\calE$, iff $i\in\calI(j)$ and $j\in\calI(i)$. We
say that $(\calV,\calE)$ is the connectivity graph of the
approximation.

Since \eqref{def_Galerkin_scheme} is ``centered'', it cannot handle
properly shocks and discontinuous data. To address this issue we
introduce some artificial dissipation. We do so by using the graph
Laplacian associated with the connectivity graph $(\calV,\calE)$.  We
assume that the graph viscosity  $\{d_{ij}\upLn\}_{(i,j)\in\calE}$ is scalar and has
the following properties:
\begin{align}\label{dijprop}
  d_{ij}\upLn = d_{ji}\upLn> 0, \quad \text{if} \quad i\ne j.
\end{align}
Although the diagonal value $d_{ii}\upLn$ is not needed, we adopt the convention
$d_{ii}\upLn:=-\sum_{j\in\calI(i)\backslash\{i\}} d_{ij}\upLn$.  This convention will
help us shorten some expressions later.  We are now in position to
define the first-order method on which the rest of the paper is built.
We call low-order update $\bsfU_i\upLnp$ the quantity computed as
follows:
\begin{align}
\label{def_dij_scheme}
  \frac{m_i}{\dt}(\bsfU_i\upLnp-\bsfU_i^n)
  + \sum_{j\in \calI(i)} \polf(\bsfU_j^n)\bc_{ij}
  - \sum_{j\in \calI(i){\setminus}\{i\}} d_{ij}\upLn (\bsfU^n_j-\bsfU^n_i) = m_i \bS(\bsfU_i^n),
\end{align}
for all $i\in\calV$.
Without further assumptions, the scheme has built-in conservation properties;
more specifically, the following holds true.
\begin{lemma}[Conservation]\label{RemLocGlobCons} 
  Assume that $\bS\equiv\bzero$, then the scheme
  \eqref{cijprop}--\eqref{def_dij_scheme} is conservative in the sense
  that the following identity holds for any $n\in \polN$:
\begin{align}\label{globmasscons}
  \sum_{i \in \vertind} m_i \bsfU_i\upLnp = \sum_{i \in \vertind} m_i\bsfU_i^n. 
\end{align}
\end{lemma}
\begin{proof} Using that $\sum_{j\in\calI(i)} \bc_{ij}=\bzero$, we
  rewrite~\eqref{def_dij_scheme} in the form
\begin{equation*}
\frac{m_i}{\dt}(\bsfU_i\upLnp-\bsfU_i^n)
+ \sum_{j\in \calI(i)} ( \polf(\bsfU_j^n) + \polf(\bsfU_i^n) )\bc_{ij}
- d_{ij}\upLn (\bsfU^n_j - \bsfU^n_i)  = \bzero.
\end{equation*}
Defining
$\bsfF_{ij}\upLn := ( \polf(\bsfU_j^n) + \polf(\bsfU_i^n) )\bc_{ij} 
- d_{ij}\upLn (\bsfU^n_j - \bsfU^n_i)$,
the above identity implies that
$\sum_{i \in \vertind} m_i \bsfU_i^{L,n+1} = \sum_{i \in \vertind}
m_i\bsfU_i^n + \sum_{i \in \vertind}\sum_{j\in \calI(i)}
\bsfF_{ij}\upLn$.
The assertion is a consequence of the skew-symmetry of $\bc_{ij}$
and the symmetry of $d_{ij}\upLn$, \ie
$\sum_{i \in \vertind}\sum_{j\in \calI(i)} \bsfF_{ij}\upLn=\bzero$.
\end{proof}

\begin{remark}[Consistency] Although the consistency question will be
  addressed later, let us say at this point that consistency is not an
  immediate consequence of \eqref{cijprop} and
  \eqref{dijprop}. Consistency will be achieved provided one can show
  that $\frac{m_i}{\dt}(\bsfU_i\upLnp-\bsfU_i^n)$ is an approximation
  of $\partial_t\bu$ (\ie a moment with a shape function),
  $\sum_{j\in \calI(i)} \polf(\bsfU_j^n)\bc_{ij}$ is an approximation
  of $\DIV \polf(\bu)$ (\ie a moment with a shape function), and
  $m_i \bS(\bsfU_i^n)$ is an approximation of $\bS(\bu)$ (\ie a moment
  with a shape function).  Note that if all the values
  $\{\bsfU_j\}_{j\in\calI(i)}$ are constant, the graph viscosity term
  $ \sum_{j\in \calI(i)} d_{ij}\upLn (\bsfU^n_j - \bsfU^n_i)$
  vanishes; which in some sense implies that
  \eqref{def_dij_scheme} is a first-order consistent
    perturbation of \eqref{def_Galerkin_scheme}. The scalars $m_i$
  and the vectors $\{\bc_{ij}\}_{j\in \calI(i)}$ are not uniquely
  defined and they may take different forms depending on the method of
  choice. In sections \S\ref{secFV}, \S\ref{secCG} and \S\ref{secDG}
  we will describe three methods based on finite volumes, continuous
  finite elements, and discontinuous finite elements, all of which can
  be written in the form \eqref{cijprop}-\eqref{def_dij_scheme}.
\end{remark}

\begin{remark}[Algebraic-Fluxes]\label{Rem:algeFluxes} For further reference it will
be useful to define the following quantity which we henceforth refer to as 
low-order algebraic flux:
\begin{align}\label{defFijabstract}
  \bsfF_{ij}\upLn := ( \polf(\bsfU_j^n) + \polf(\bsfU_i^n) )\bc_{ij}
  - d_{ij}\upLn (\bsfU^n_j - \bsfU^n_i).
\end{align}
Algebraic fluxes will be instrumental for the development of limiting
techniques in \S\ref{sec:ConvLimAbstract}. In particular, the scheme
\eqref{def_dij_scheme} is conveniently rewritten as
follows:
\begin{align}\label{abscons}
\frac{m_i}{\dt}(\bsfU_i\upLnp-\bsfU_i^n)
 + \sum_{j\in \calI(i)} \bsfF_{ij}\upLn = m_i \bS(\bsfU_i^n).
\end{align}
\end{remark}

\begin{remark}[Well-balancing]
  In general, systems with a source term have time-independent solutions,
  \ie fields solving $\DIV\polf(\bu)=\bS(\bu)$, and it is often a
  desirable feature of numerical schemes that they preserve these
  steady states. This lead to the notion of well-balancing introduced
  in \cite{bermudez_vazquez_94,greenberg_leroux_96}; we also refer to
  \cite[\S3]{Huang_Tai-ping_1986} for early ideas on well-balancing.
  Although, well-balancing is a very important notion, it will not be
  addressed in this paper.
\end{remark}

\subsection{Invariant domain preserving graph viscosity}\label{secInvVisco} 
Now we propose a definition of the graph viscosity that makes the
algorithm~\eqref{def_dij_scheme} invariant domain preserving.  Recall that the discretization
setting is still unspecified. Most of
the arguments presented in this subsection are generalizations of those
in \S3.2, \S4.1 and \S4.2 of \cite{GuePo2016}.

Since $\sum_{j\in \calI(i)} \polf(\bsfU_i^n)\bc_{ij} = \boldsymbol{0}$
(see property \eqref{cijprop}) we can rewrite the scheme
\eqref{def_dij_scheme} as follows:
\begin{equation*}
\frac{m_i}{\dt}(\bsfU_i\upLnp-\bsfU_i^n) +\!\!\sum_{j\in \calI(i){\setminus}\{i\}} 2d_{ij}\upLn \bsfU^n_i
+  ( \polf(\bsfU_j^n) - \polf(\bsfU_i^n) ) \bc_{ij}
- d_{ij}\upLn (\bsfU^n_j + \bsfU^n_i)  = m_i \bS(\bsfU_i^n).
\end{equation*}
Then, upon introducing the auxiliary states (recalling that
$d_{ij}\upLn>0$ by assumption),
\begin{align}
\label{def_barstates}
\overline{\bsfU}_{ij}^{n} :=
\frac{1}{2}(\bsfU^{n}_i + \bsfU^{n}_j)
-(\polf(\bsfU_j^n) - \polf(\bsfU_i^n))\frac{\bc_{ij}}{2 d_{ij}\upLn},
\end{align}
with the convention $\overline{\bsfU}_{ii}^{n} := \bsfU^{n}_i$,
the low-order scheme \eqref{def_dij_scheme} can be rewritten as
follows:
\begin{align}
\label{def_dij_scheme_convex}
  \bsfU_i\upLnp = 
\bigg(1 -\!\! \sum_{j \in \calI(i)\backslash\{i\}} \frac{2 \dt d_{ij}\upLn}{m_i} \bigg) 
\bsfU_{i}^{n}
  + \!\! \sum_{j \in \calI(i)\backslash\{i\}} \frac{2 \dt d_{ij}\upLn}{m_i}  
\overline{\bsfU}_{ij}^{n} + \dt \bS(\bsfU_i^n). 
\end{align}
A first key observation we make at this point about \eqref{def_dij_scheme_convex} is that
upon setting $\bn_{ij}:= \bc_{ij}/\|\bc_{ij}\|_{\ell^2}$, we realize
that $\overline\bsfU_{ij}^{n}$ is exactly of the form
$\overline\bu(t,\bn_{ij},\bsfU_i^n,\bsfU_j^n)$ as defined in
\eqref{elemetary_Riemann_pb} with the fake time
$t_{ij}=\|\bc_{ij}\|_{\ell^2}/2d_{ij}\upLn$. Then
Lemma~\ref{Lem:elemetary_Riemann_pb} motivates the following
definition for the graph viscosity coefficients $d_{ij}\upLn$:
\begin{equation}
\label{Def_of_dij} 
d_{ij}\upLn := \max(\lambda_{\max}(\bn_{ij},\bsfU_i^n,\bsfU_j^n)
\|\bc_{ij}\|_{\ell^2},
\lambda_{\max}(\bn_{ji},\bsfU_j^n,\bsfU_i^n) \|\bc_{ji}\|_{\ell^2}), 
\end{equation}
where recall that $\lambda_{\max}(\bn_{ij},\bsfU_i^n,\bsfU_j^n)$ is the maximum wave speed 
defined in \S\ref{Sec:Riemann_pb_average_max_wavespeed}.

\begin{lemma}[Invariance of the auxiliary states]\label{Lem:InvarBar} 
  Let $\calB\subset \calA$ be a convex invariant set for
  \eqref{def:hyperbolic_system} such that
  $\bsfU_i^n,\bsfU_j^n\in \calB$. The state
  $\overline{\bsfU}_{ij}^{n}$ defined in \eqref{def_barstates}, with
  $d_{ij}\upLn$ as defined in \eqref{Def_of_dij}, belongs to $\calB$.
\end{lemma}

\begin{proof} Let us set
  $t_{ij}:= \|\bc_{ij}\|_{\ell^2}/(2d_{ij}\upLn)$, then according to
  Lemma~\ref{Lem:elemetary_Riemann_pb}, we have
  $\overline\bsfU_{ij}^{n}:=\overline\bu(t_{ij},\bn_{ij},\bsfU_{i}^n,\bsfU_{j}^n)
  \in \calB$
  if
  $\lambda_{\max}(\bn_{ij},\bsfU_{i}^n,\bsfU_{j}^n) t_{ij}\le
  \frac12$.
  But the definition \eqref{Def_of_dij} implies that
  $d_{ij}\upLn \ge \lambda_{\max}(\bn_{ij},\bsfU_{i}^n,\bsfU_{j}^n)
  \|\bc_{ij}\|_{\ell^2}$,
  which is the CFL condition
  $t_{ij}\lambda_{\max}(\bn_{ij},\bu_L,\bu_R)\le \frac12$ for the
  conclusions of Lemma~\ref{Lem:elemetary_Riemann_pb} to hold. This
  proves that
  $\overline\bsfU_{ij}^{n}:=\overline\bu(t,\bn_{ij},\bsfU_{i}^n,\bsfU_{j}^n)
  \in \calB$ for all $j\in \calI(i)$ since $\calB$ is a convex invariant set.
\end{proof}

A second important observation about \eqref{def_dij_scheme_convex} is
that $\bsfU_i\upLnp - \dt \bS(\bsfU_i^n)$ is a convex combination of
$\bsfU_i^n$ and the states
$\{\overline\bsfU_{ij}^n\}_{j\in\calI(i){\setminus}\{i\}}$ provided
$\dt$ is small enough. This is the key to the following result.

\begin{theorem}[Local invariance] \label{Thm:UL_is_invariant} Let
  $n\ge 0$ and let $i \in \vertind$. Assume that $\dt$ is small enough
  so that $1+4\dt \frac{d_{ii}\upLn}{m_i}\ge 0$ and $2\dt \le \tau_0$.  Let
  $\calB\subset \calA$ be a convex invariant set for
  \eqref{def:hyperbolic_system} such that
  $\bsfU_j^n\in \calB$ for all $j\in \calI(i)$, then
  $\bsfU_i\upLnp \in \calB$.
\end{theorem}

\begin{proof}
Using the definition $d_{ii}\upLn:=\sum_{j \in \calI(i)\backslash\{i\}} -d_{ij}\upLn$,
we first notice that \eqref{def_dij_scheme_convex} can be rewritten as follows:
\begin{equation}
  \bsfU_i\upLnp = \frac12\Bigg( 
\bigg(1+4\dt \frac{d_{ii}\upLn}{m_i}\bigg) \bsfU_{i}^{n}
  + \!\! \sum_{j \in \calI(i)\backslash\{i\}}\!\!\! \frac{4 \dt d_{ij}\upLn}{m_i} 
\overline{\bsfU}_{ij}^{n}\Bigg) 
+ \frac12\big(\bsfU_{i}^{n} + 2\dt\bS(\bsfU_i^n)\big).\label{def_dij_scheme_convex_with_source}
\end{equation}
With obvious notation, let us rewrite the above equation as follows
$\bsfU_i\upLnp = \frac12 \bsfW_1 + \frac12 \bsfW_2$.  Owing to
the local CFL assumption $1+4\dt \frac{d_{ii}\upLn}{m_i}\ge 0$,
$\bsfW_1$ is a convex combination of $\bsfU_i^n$ and the collection of
states $\{\overline\bsfU_{ij}^{n}\}_{j\in \calI(i)}$. 
But we have established in  Lemma~\ref{Lem:InvarBar} that 
$\overline\bsfU_{ij}^{n}\in \calB$. Then, the
convexity of $\calB$ implies $\bsfW_1$ is in $\calB$. Since $\calB$ is an
invariant set according to Definition~\ref{Def:invariant_set} and
$\bsfU_i^{n}\in \calB$ by assumption, the condition $2\dt \le \tau_0$
implies that $\bsfW_2:= \bsfU_{i}^{n} + 2\dt\bS(\bsfU_i^n)$ is a
member of $\calB$. In conclusion, the convexity of $\calB$ implies that
$\bsfU_i\upLnp = \frac12 \bsfW_1 + \frac12 \bsfW_2$ is in $\calB$.
\end{proof}

\begin{corollary}[Global invariance]\label{Cor:global_invariance} Let
  $n\in \polN$. Assume that the global
  CFL condition
  $\min_{i \in \calV} \big(1+4\dt
  \frac{d_{ii}\upLn}{m_i}\big)\ge 0$
  holds and $2\dt \le \tau_0$. Let $\calB\subset \calA$ be a convex invariant set. 
  Assume that
  $\bsfU_i^n \in \calB$ for all $i\in\calV$, then
  $\bsfU_i\upLnp\in \calB$ for all  $i \in \calV$.
\end{corollary}

\begin{theorem}[Entropy inequality]\label{Thm:entrop_ineq} 
   Let $(\eta,\bq)$ be an entropy pair
  for \eqref{def:hyperbolic_system}.  Let $n\ge 0$ and $i\in\calV$.
  Assume also that the local CFL condition holds
  $1+2\dt\frac{d_{ii}\upLn}{m_i}\ge 0$ and $2\dt\le \tau_0$, then the following
  local entropy inequality holds true for any entropy pair  $(\eta,\bq)$ of the system
  \eqref{def:hyperbolic_system}:
\begin{multline}
\frac{m_i}{\dt} (\eta(\bsfU_i\upLnp) - \eta(\bsfU_i^n))   
+ \sum_{j\in\calI(i)} \bq(\bsfU_j^n)\bc_{ij} 
- d_{ij}\upLn (\eta(\bsfU_{j}^{n})-\eta(\bsfU_{i}^{n})) \\
\le m_i \bS(\bsfU_i^n)\SCAL  \GRAD \eta(\bsfU_i\upLnp) . 
\label{discrete_entropy_inequality}
\end{multline}
\end{theorem}
\begin{proof}
  Let $i \in \vertind$ and let $(\eta,\bq)$ be an entropy pair for the system
  \eqref{def:hyperbolic_system}. Then recalling
  \eqref{def_dij_scheme_convex}, the CFL condition and the convexity
  of $\eta$ imply that
\begin{align*}
  \eta(\bsfU_i\upLnp - \dt \bS(\bsfU_i^n))  \le 
  \Big(1 - \sum_{j\in \calI(i)\backslash\{i\}} \frac{2\dt d_{ij}\upLn}{m_i} \Big)\eta(\bsfU_i^n) 
  + \sum_{j\in \calI(i)\backslash\{i\}}\frac{2\dt d_{ij}\upLn}{m_i}\eta(\overline\bsfU_{ij}^{n}).   
\end{align*}
Lemma~\ref{Lem:elemetary_Riemann_pb} implies that 
$\eta(\overline\bsfU^{n}_{ij})
\le \tfrac{1}{2}(\eta(\bsfU_i^n)+\eta(\bsfU_j^n)) 
  - t_{ij}(\bq(\bsfU_j^n)\SCAL\bn_{ij} - \bq(\bsfU_i^n)\SCAL\bn_{ij})$,
with $t_{ij}=\|\bc_{ij}\|_{\ell^2}/2 d_{ij}\upLn$; hence,
\begin{multline*}
\frac{m_i}{\dt} (\eta(\bsfU_i\upLnp - \dt \bS(\bsfU_i^n)) - \eta(\bsfU_i^n))   \le 
 \sum_{j\in \calI(i)\backslash\{i\}} 2 d_{ij}\upLn(\eta(\overline\bsfU_{ij}^{n}) -\eta(\bsfU_i^n)) \\
\le  \sum_{j\in \calI(i)\backslash\{i\}} d_{ij}\upLn(\eta(\bsfU_{j}^{n}) - \eta(\bsfU_i^n))
- \|\bc_{ij}\|_{\ell^2}(\bq(\bsfU_j^n)\SCAL\bn_{ij} - \bq(\bsfU_i^n)\SCAL\bn_{ij}).
\end{multline*}
Moreover, the convexity of $\eta$ implies that 
\[
\eta(\bsfU_i\upLnp) - \dt \bS(\bsfU_i^n)\SCAL  \GRAD \eta(\bsfU_i\upLnp) \le
 \eta(\bsfU_i\upLnp - \dt \bS(\bsfU_i^n)). 
\]
The conclusion follows from the definitions of $\bn_{ij}$, $\bc_{ij}$
and $d_{ij}\upLn$.
\end{proof}

\begin{remark}[Terminology] In order to refer to the scheme
  \eqref{def_dij_scheme} with \eqref{Def_of_dij}, following
  \citep{GuerNazPopTom2017} we will use the acronym GMS-GV, standing
  for Guaranteed Maximum Speed Graph Viscosity.
\end{remark}

\begin{remark}[Symmetry] Since $\bc_{ij} = -\bc_{ji}$ we note that
  $\overline{\bsfU}_{ij}^{n} = \overline{\bsfU}_{ji}^{n}$ (see
  definition \eqref{def_barstates}) which in turn implies that
  $\lambda_{\max}(\bn_{ij},\bsfU_i^n,\bsfU_j^n)=
  \lambda_{\max}(\bn_{ji},\bsfU_j^n,\bsfU_i^n)$.
  In conclusion
  $\lambda_{\max}(\bn_{ij},\bsfU_i^n,\bsfU_j^n) \|\bc_{ij}\|_{\ell^2}=
  \lambda_{\max}(\bn_{ji},\bsfU_j^n,\bsfU_i^n) \|\bc_{ji}\|_{\ell^2}$.
Note that these properties may not hold at the boundary if 
nontrivial boundary conditions are applied.
\end{remark}

\begin{remark}[Positivity]
  It may happen that estimating a guaranteed upper bound
  $\lambda_{\max}(\bn,\bsfU_L,\bsfU_R)$ on the maximum wave speed in
  the Riemann problem is difficult. In this case one has to come up
  with some informed guess.  We now give a lower bound on
  $\lambda_{\max}(\bn,\bsfU_L,\bsfU_R)$ that guaranties positivity if
  it happens that some components of $\bsfU$, say $\sfU$, has to be
  positive (think of the density and the total energy in the Euler
  equations or the water height in the shallow water equations). 
  Let $\bef_{\sfU}:\calA\to \Real^d$ be the component of
  $\polf$ that corresponds to the component $\sfU$ of $\bsfU$. Assume
  that $\calB:=\{\bsfU\in\calA\st \sfU>0\}$ is an invariant set for
  \eqref{def:Riemann_problem}, assume also that the estimate on the
  maximum wave speed is such that the
  $\lambda_{\max}(\bn_{ij},\bsfU_i^n,\bsfU_j^n)\ge
  \max \big(\frac{\bef_{\sfU}(\bsfU_j^n) \SCAL \bn_{ij}}{\sfU_j^n},0\big)$,
  then under the same CFL condition as in
  Theorem~\ref{Thm:UL_is_invariant} and
  Corollary~\ref{Cor:global_invariance}, the set
  $\widetilde \calB:=\{\bsfU\in \Real^m\st \sfU>0\} \subseteq \calB$ is such that
  $\big(\bsfU_i^n \in \calB,  \ \forall i\in \calV\big)
  \Rightarrow \big(\bsfU_i\upLnp \in \widetilde\calB,  \ \forall i \in \calV\big)$.
  Let us finally illustrate the
  above result in one space dimension.  For instance, for finite
  volumes and for piecewise linear continuous elements in one space
  dimension, one has $\bc_{ij}=\frac 12 \bn_{ij}$ (see
  \S\ref{Sec:examples_discretization}). Then, for the density in the
  Euler equations, or for the water height in the Saint-Venant
  equations, the above estimate becomes
  $\lambda_{\max}(\bn_{ij},\bsfU_i^n,\bsfU_j^n)\ge
  \max\big(\frac12\bn_{ij}\SCAL\bsfV(\bsfU_j^n),0\big)$
  where $\bsfV(\bsfU)$ is the velocity.  One recognizes here the
  standard upwind estimate.
\end{remark}

\section{Examples of
  discretizations} \label{Sec:examples_discretization} In this section
we illustrate the GMS-GV scheme described in \S\ref{sec:firstorder} in
the following three space discretization settings: finite volumes,
continuous finite elements, and discontinuous elements.

\subsection{Finite Volumes}\label{secFV} 
We now illustrate the construction of the abstract low-order
scheme~\eqref{cijprop}--\eqref{def_dij_scheme} in the context of finite volumes
(FV).

\subsubsection{Technical preliminaries}

We unify our presentation by putting into a single framework the
so-called cell-centered and vertex-centered finite volume techniques,
see Figure~\ref{FigPatchFV}.  We refer the reader to
\cite{Barth2004,Eymard2000} for comprehensive reviews on the finite
volume techniques.  For any manifold $E\subset \Real^d$ of dimension
$l$ we denote by $|E|$ the $l$-Lebesgue measure of $E$.  We assume
that we have at hand a partition of the computational domain $\Dom$
into polygonal (polyhedral) cells $\{K_i\}_{i \in\calV}$. We
henceforth denote by $\calT_h$ this collection of cells. For any pair
of cells $K_i,K_j$ having a common interface, we denote by
$\Gamma_{ij}:=\partial K_i\cap\partial K_j$ the interface in question. The unit vector on
$\Gamma_{ij}$ pointing from $K_i$ to $K_j$ is denoted $\bn_{ij}$.

 \begin{figure}[h]
\centering
\includegraphics[scale=0.35]{\FIGS/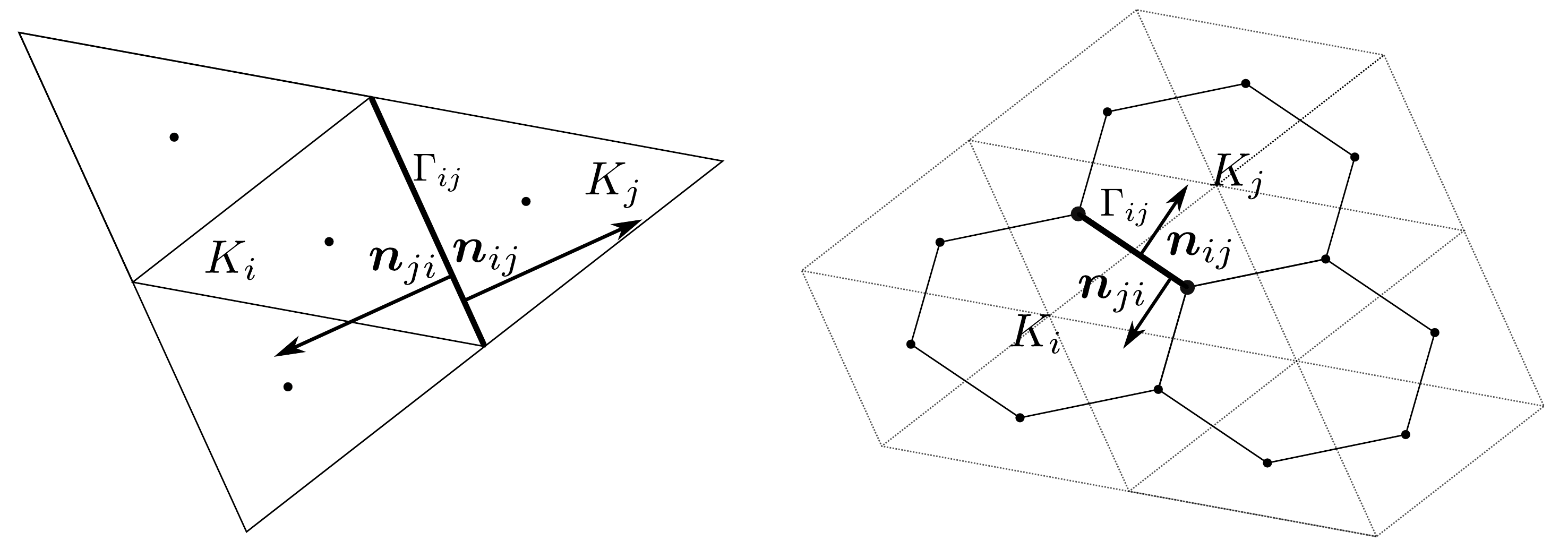}
\caption{Finite volume patch arising from a cell-centered
  discretization (left) and a vertex-centered discretization
  (right).}\label{FigPatchFV}
\end{figure}

\subsubsection{Definitions of $(\calV,\calE)$, $m_i$, and $\bc_{ij}$}
We define the connectivity graph $(\calV,\calE)$ by identifying the
vertices of this graph with the cells in $\calT_h$, and we say that a
pair of cells $K_i,K_j$ form an edge of the graph, \ie
$(i,j)\in\calE$, iff the cells $K_i$ and $K_j$ share an interface, \ie
$\partial K_i\cap\partial K_j$ is a $(d-1)$-manifold of positive
measure.  For any $i\in\calV$ we define the adjacency list $\calI(i)$ to be the list of all the 
cells in $\calT_h$ sharing an interface with $K_i$, \ie 
$\calI(i) := \{ j \in \vertind\st\ (i,j)\in \calE \}$,
see Figure~\ref{FigPatchFV}. Denoting by $\indic_{K_j}$ the indicator
function of the cell $K_j$, we set $X_h:=\vect\{ \indic_{K_j}\}_{j\in\calV}$ and then define the approximation space
$\bX_h:=(X_h)^m=\{\sum_{j\in\calV} \bsfV_j\indic_{K_j}\st
\bsfV_j\in \Real^m, \forall j\in \calV\}$.

Let
$\bu_{h}^n = \sum_{j\in\calV} \bsfU_j^n \indic_{K_j}\in
\bX_h$
be the approximation of $\bu$ at time $t^n$, then most first-order finite volume schemes are written as follows
\begin{align*}
  \frac{|K_i|}{\dt}(\bsfU_i\upLnp-\bsfU_i^n)
  + \sum_{j\in \calI(i)\backslash\{i\}} \bsfF_{ij}\upLn = |K_i| \bS(\bsfU_i^n),
\end{align*}
where $\bsfF_{ij}\upLn$ is usually the Lax-Friedrichs/Rusanov flux (integrated over
$\Gamma_{ij}$):
\begin{align}\label{FVflux}
\bsfF_{ij}\upLn := \frac{|\Gamma_{ij}|}{2}( \polf(\bsfU_j^n) + \polf(\bsfU_i^n) )\bn_{ij}
- \alpha_{ij}\upLn (\bsfU^n_j - \bsfU^n_i),
\end{align}
where $\alpha_{ij}\upLn$ is some wave speed.
Hence, we recover the generic expression~\eqref{def_dij_scheme} for
the finite volume framework by setting
\begin{align}
m_i &:= |K_i|, \quad
\bc_{ij} := \frac{|\Gamma_{ij}|}{2} \bn_{ij},\ \forall j\in \calI(i)\backslash\{i\}, \quad
\bc_{ii}:=\bzero, \quad
d_{ij}\upLn := \alpha_{ij}\upLn. 
\end{align}
The definition of $\bc_{ij}$ immediately implies that $\bc_{ij} = -\bc_{ij}$, and
the Stokes theorem implies that
$\sum_{j\in \calI(i)} \bc_{ij} = \tfrac{1}{2} \sum_{j\in
  \calI(i)\backslash\{i\}} \bn_{ij} |\Gamma_{ij}| = \tfrac{1}{2}
\int_{\partial K_i} \bn \, \diff s = \bzero$, which is the
conservation property stated in \eqref{cijprop}.
 Note that $\bsfF_{ij}\upLn = - \bsfF_{ji}\upLn$ since
$\bn_{ij} = -\bn_{ji}$. Let us mention in passing that while any
family of vectors of the form
$\bc_{ij} =\alpha \bn_{ij} |\Gamma_{ij}|$ satisfies the conservation
constraint \eqref{cijprop}, only the factor $\alpha = \frac{1}{2}$
leads to a consistent discretization of the divergence operator.

\subsection{Continuous finite elements}\label{secCG}
We describe in this section one possible implementation of the
abstract low-order scheme~\eqref{cijprop}--\eqref{def_dij_scheme} in
the context of continuous finite elements (cG). The set of the 
$d$-variate polynomials of degree at most $k\in\polN$ is denoted
$\polP_{k,d}$. The reader who is familiar with
\citep{GuePo2016,guermond_popov_second_order_2018,GuerNazPopTom2017}
is invited to move to \S\ref{secDG}.

\subsubsection{Technical preliminaries}\label{sec:CGpreliminaries} Let $\famTh$ be
a shape-regular sequence of matching meshes. To keep some level of
generality we assume that the elements in the mesh are generated from
a finite number of reference elements denoted
$\wK_1,\dots,\wK_\varpi$.  For example, the mesh $\calT_h$ could be
composed of a combination of triangles and parallelograms in dimension
two (we would have $\varpi=2$ in this case); it could also be composed
of a combination of tetrahedra, parallelepipeds, and triangular prisms
in dimension three (we would have $\varpi=3$ in this case). The
diffeomorphism mapping $\wK_r$ to an arbitrary element $K\in \calT_h$
is denoted $T_K : \wK_r \longrightarrow K$. We now introduce a set of
reference finite elements
$\{(\wK_r,\wP_r,\wSigma_r)\}_{1\le r\le \varpi}$ (the index
$r\in\intset {1}{\varpi}$ will be omitted in the rest of the paper to
simplify the notation), and we define the following scalar-valued and
vector-valued continuous finite element spaces:
\begin{align} 
\label{eq:Xh}
X_h &=\{ v\in \calC^0(\Dom;\Real)\st 
v_{|K}{\circ}T_K \in \wP,\ \forall K\in \calT_h\},\qquad 
\bX_h= [X_h]^m.
\end{align}  
The global shape functions are denoted by
$\{\varphi_i\}_{i \in \vertind}$ and we assume that they satisfy the partition of unity
property $\sum_{i \in \vertind}\varphi_i(\bx) =1$, for all
$\bx \in\Dom$.

\subsubsection{Definitions of $(\calV,\calE)$, $m_i$, and $\bc_{ij}$}
We define the connectivity graph $(\calV,\calE)$ by identifying the
shape functions $\{\varphi_i\}_{i\in\calV}$ with the vertices of the
graph. The edges are defined as follows: we say that two shape
functions (or two degrees of freedom) form an edge, \ie
$(i,j)\in\calE$, iff $\varphi_i\varphi_j \not\equiv 0$. For any
$i\in \calV$, the adjacency list $\calI(i)$ is defined by setting
$\calI(i):=\{j\in \calV\st (i,j)\in \calE\}$.

Let $\calM$ be the consistent mass matrix with entries
$\int_{\Dom} \varphi_i(\bx)\varphi_j(\bx)\diff \bx$, $i,j\in \calV$,
and let $\calM^L$ be the diagonal lumped mass matrix with entries
\begin{equation}
 \label{def_of_mi}
  m_i:= \int_{\Dom} \varphi_i(\bx)\diff \bx.
\end{equation}
The partition of unity property implies that
$m_i = \sum_{j\in \calI(i)} \int_\Dom \varphi_j(\bx)\varphi_i(\bx)\diff
\bx$,
\ie the entries of $\calM^L$ are obtained by summing the rows of
$\calM$. In the rest of the paper we assume that $m_i>0$, for all
$i \in \vertind$. This assumption is satisfied by many  families of finite elements.

Let $\bu_{h}^n = \sum_{j\in\calV} \bsfU_j^n \varphi_j \in \bX_h$ be
the approximation of $\bu$ at time $t^n$, where $\bX_h$ is the
continuous finite element space defined in \eqref{eq:Xh}.  We
  approximate $\polf(\bu_h^n)$ by
  $\sum_{j \in \vertind} \polf(\bsfU_j^n) \varphi_j$.  If $\wP$ is
  composed of Lagrange elements, then
  $\sum_{j \in \vertind} \polf(\bsfU_j^n) \varphi_j$ is the
  Lagrange interpolation of $\polf(\bu_h^n)$, and in this case the
  approximation is fully consistent with the polynomial degree of
  $\wP$; otherwise, the approximation is formally at least
  second-order accurate in space since it is exact if $\polf$ is
  linear. As a result, we have
\begin{align}\label{fluxinterp}
  \int_\Dom \DIV(\polf(\bu_h^n)) \varphi_i\diff \bx 
&\approx \sum_{j\in \calI(i)} \polf(\bsfU_j^n) \int_\Dom \varphi_i\GRAD \varphi_j\diff \bx
=\sum_{j\in \calI(i)} \polf(\bsfU_j^n)\bc_{ij},
\end{align} 
where the coefficients $\bc_{ij}\in \Real^d$ are defined by
\begin{equation}
\label{def_of_cij_CG}
\bc_{ij} = \int_\Dom \varphi_i \GRAD\varphi_j \diff \bx, \quad \forall j \in \calI(i).
\end{equation}
Here we observe that the partition of unity property and definition
\eqref{def_of_cij_CG} imply that
$\sum_{j\in \calI(i)} \bc_{ij} 
= \sum_{j\in \calI(i)} \int_\Dom \varphi_i \GRAD\varphi_j \diff \bx
= \int_\Dom \varphi_i \GRAD\big(\sum_{j\in \calI(i)}\varphi_j\big) \diff \bx
=\bzero.$
On the other hand, the skew-symmetry property $\bc_{ij} = -\bc_{ji}$
follows using integration by parts if $\Dom$ is the $d$-torus (which
is the case for periodic boundary conditions) or if either $\varphi_i$
or $\varphi_j$ vanish at the boundary of $D$ (which is the case when
we solve the Cauchy problem).

\subsection{Discontinuous finite elements}\label{secDG}
We finally describe in this section one possible implementation of the
abstract low-order scheme~\eqref{cijprop}--\eqref{def_dij_scheme} in the context of
discontinuous  finite elements (dG). This section builds on
top of the definitions and notation already introduced in
\S\ref{sec:CGpreliminaries}. 

\subsubsection{Technical preliminaries}\label{sec:DGpreliminaries} 
Here we clarify/expand on the specific details related to
discontinuous spaces. We define scalar-valued and vector-valued
discontinuous finite element spaces as follows:
\begin{equation}
\label{eq:XhDG}
X_h =\{ v\in L^1(\Dom;\Real)\st 
v_{|K}{\circ}T_K \in \wP,\ \forall K\in \calT_h\},\quad
\bX_h:=[X_h]^m.
\end{equation} 
We denote by $\{\varphi_i\}_{i\in\calV}$ the collection of global shape
functions generated from the reference shape functions,
\ie $X_h = \vect\{\varphi_i\}_{i\in\calV}$. Each shape function has
support on one cell only.  We denote by $\calI(K)$ the set of indices
of the shape functions with support in $K$. Similarly, letting
$\partial K$ to be the boundary of the cell $K$, we denote by
$\calI(\partial K)$ the set of indices of the shape functions with
non-vanishing trace on $\partial K$:
\begin{align}\label{def:setpartialK}
\calI(K) := \big\{ i \in \vertind \st \varphi_{i|K} \not \equiv 0 \big\},\qquad 
  \calI(\partial K) := \big\{ i \in \vertind \st \varphi_{i|\partial K} \not \equiv 0 \big\}.
\end{align}
Note that $\calI(\partial K)$ not only includes indices of shape
functions with support in $\calI(K)$ but this set also includes indices
of shape functions that do not have support in $K$ (see
Figure~\ref{DGpatch} for additional geometrical insight). More
precisely $\calI(\partial K)$ is the union of two disjoint sets
$\calI(\partial K\upi)$ and $\calI(\partial K\upe)$
defined as
\begin{align}
\calI(\partial K\upi) 
&:= \big\{ i \in \calI(K) \ \big| \ \varphi_{i|\partial K} \not \equiv 0 \big\}, 
\qquad  \calI(\partial K\upe) 
:= \calI(\partial K)\backslash \calI(\partial K\upi).
\end{align}
Finally, we assume that the finite element spaces
are always constructed so that the sets of shape functions $\{\varphi_j\}_{j \in \calI(K)}$ 
form a partition of unity over $K$ and 
the shape functions
$\{\varphi_j\}_{j \in \calI(\partial K\upi)}$,
$\{\varphi_j\}_{j \in \calI(\partial K\upe)}$ form partitions
of unity over $\partial K$, \ie
\begin{align}
\label{boundary_partition_unity}
 \sum_{j \in \calI(K)} \varphi_{j|K} = 1, \qquad 
  \sum_{j \in \calI(\partial K\upi)} \varphi_{j|\partial K} = 1,
\ \text{ and  } \sum_{j \in \calI(\partial K\upe)} \varphi_{j|\partial K} = 1.
\end{align}

\begin{figure}[h]
\centering
\includegraphics[scale=0.3]{\FIGS/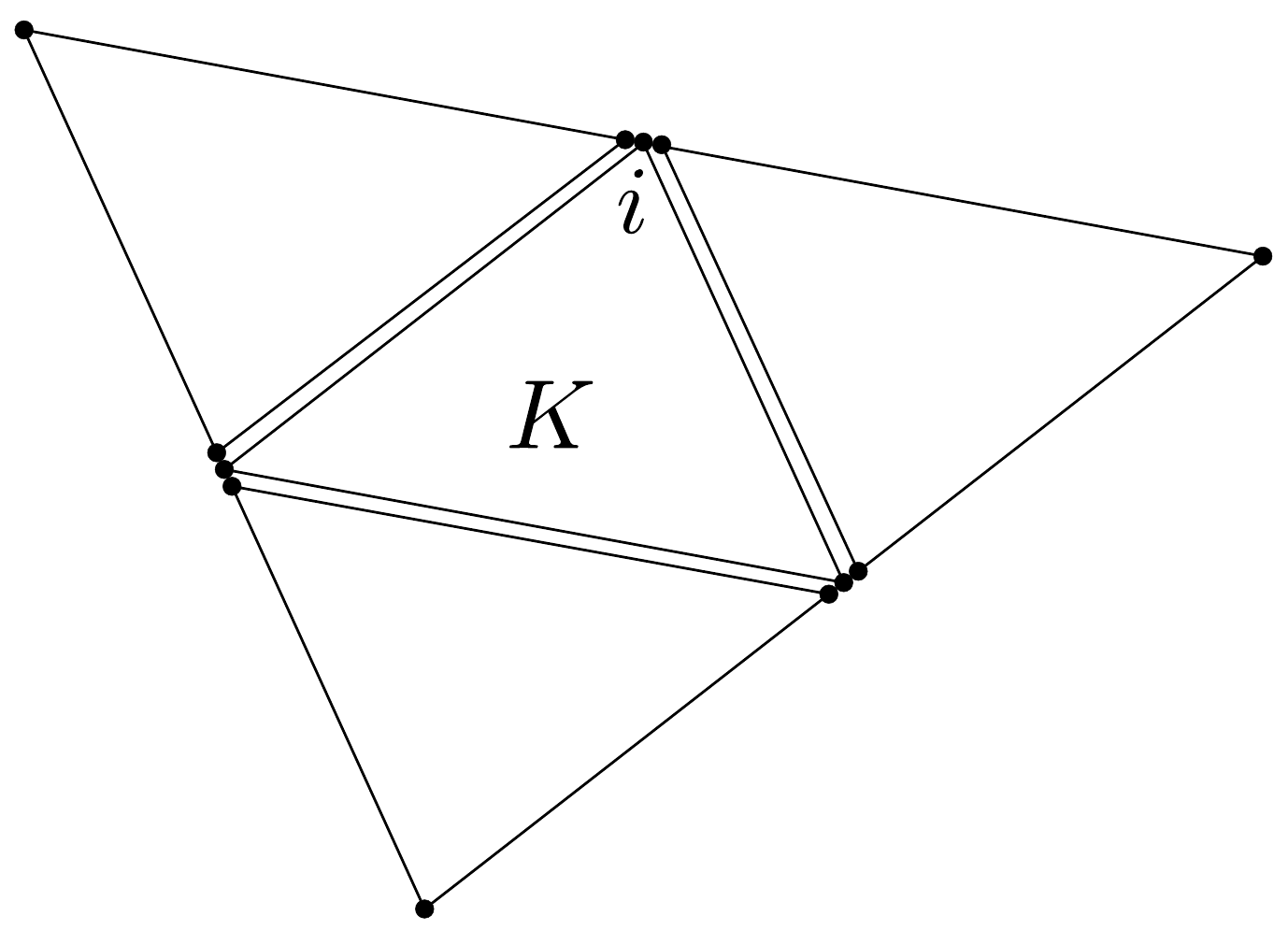}
\caption{Discontinuous $\polP_{1,2}$ finite element patch (exploded
  view). Each black dot represents a scalar shape function. In this
  picture $i \in \calI(K)$, $\textup{card}(\calI(i))= 7$,
  $\textup{card}(\calI(K)) = 3$, $\textup{card}(\calI(\partial K)) = 9$,
  $\textup{card}(\calI(\partial K^{\mathsf{i}}))= 3$,
  $\textup{card}(\calI(\partial K^{\mathsf{e}}))= 6$ and
  $\textup{card}(\calI(K)\backslash\calI(\partial K^{\mathsf{i}}) )=
  0$.}
\label{DGpatch}
\end{figure}

\subsubsection{Definitions of $(\calV,\calE)$, $m_i$, and $\bc_{ij}$}
We start by defining the undirected graph $(\calV,\calE)$. The
vertices are identified with the shape functions
$\{\varphi_i\}_{i\in\calV}$. Let $i\in\calV$ and let $K$ be the unique
cell containing the support of $\varphi_i$. For any $i,j\in\calV$, we
say that the pair $(i,j)$ is an edge of the connectivity graph, \ie
$(i,j)\in \calE$, iff either $j\in \calI(K)$ or $j\in \calI(\partial K\upe)$ and
$\varphi_i\varphi_j|_{\partial K}\not\equiv 0$.

The consistent mass matrix and the lumped mass matrix are defined as
in \S\ref{secCG}; in particular we set 
\begin{equation}
 \label{def_of_mi_dG}
  m_i:= \int_{\Dom} \varphi_i(\bx)\diff \bx.
\end{equation}

Let
$\bu_{h}^n= \sum_{j\in\calV} \bsfU_j^n \varphi_j \in \bX_h$
be the approximation of $\bu$ at time $t^n$, where $\bX_h$
is a discontinuous finite element space defined in \eqref{eq:XhDG}.
Let $K\in \calT_h$ and $i\in \calI(K)$.
The traditional heuristics for the derivation of dG schemes consists
of integrating by parts on each cell $K$ and introducing a numerical
flux $\widehat{\polf}$ on the boundary $\partial K$ as follows:
\begin{align}
\label{heuristics_flux}
\int_K \nabla\SCAL(\polf(\bu_h^n)) \varphi_i\diff \bx
\approx \int_K -\polf(\bu_h^n) \SCAL\GRAD\varphi_i\diff \bx 
+\int_{\partial K} \widehat{\polf}\bn_K \varphi_i \diff s.
\end{align}
Upon denoting by $\bu_h^{n,\mathsf{i}}$ the interior trace of $\bu_h^n$ on
$\partial K$ and $\bu_h^{n,\mathsf{e}}$ the exterior trace on
$\partial K$, it is common to define the numerical flux as follows:
\begin{equation}
\label{heuristics_flux_consistent_and_beta}
\widehat{\polf}\bn_K = \frac{1}{2}(\polf(\bu_h^{n,\mathsf{i}})+\polf(\bu_h^{n,\mathsf{e}}))\bn_K 
+ \alpha_{\partial K}^n (\bu_h^{n,\mathsf{i}}-\bu_h^{n,\mathsf{e}}),
\end{equation}
where $\alpha_{\partial K}^n>0$ is usually some ad-hoc wave speed.  The exact form of
$\alpha_{\partial K}^n$ is unimportant for the time being; the sole purpose of the term
$\alpha_{\partial K}^n (\bu_h^i-\bu_h^e)$ is to stabilize the algorithm. We are just going to
assume that this term introduces a first-order consistency error and that we are
perfectly allowed to introduce further modifications to the discrete
divergence operator \eqref{heuristics_flux} consistent with this
assumption.  Inserting \eqref{heuristics_flux_consistent_and_beta}
into \eqref{heuristics_flux} and integrating by parts, we obtain
\begin{multline}\label{contarg}
\int_K \nabla\SCAL(\polf(\bu_h)) \varphi_i\diff \bx \approx  
\int_K \DIV(\polf(\bu_h)) \varphi_i\diff \bx \\
+\int_{\partial K} \tfrac{1}{2}(\polf(\bu_h^\mathsf{e})-\polf(\bu_h^\mathsf{i}))\SCAL \bn_K \varphi_i \diff s 
+ \int_{\partial K} \alpha_{\partial K}^n (\bu_h^\mathsf{i} -\bu_h^\mathsf{e}) \varphi_i \diff s.
\end{multline}
We now consider an idea analogous to \eqref{fluxinterp} and we replace
$\polf(\bu_h)$ on the right-hand side of \eqref{contarg} by
$\sum_{j \in \vertind} \polf(\bsfU_j^n) \varphi_j$ (where
$\{\varphi_j\}_{j \in \vertind}$ are the shape functions of our
discontinuous finite element space) to get:
\begin{align}
\label{algedivh}
\begin{aligned}
&\int_K \nabla\SCAL(\polf(\bu_h)) \varphi_i\diff \bx \approx   
\sum_{j \in \calI(K)} \polf(\bsfU_j^n) \SCAL\bc_{ij}^K \\
& \ \ \ \ + \sum_{j \in \calI(\partial K^{\mathsf{e}})}  \polf(\bsfU_j^n) \SCAL \bc_{ij}^{\partial } 
- \sum_{j \in \calI(\partial K^{\mathsf{i}})} \polf(\bsfU_j^n) \SCAL \bc_{ij}^{\partial } 
+ \int_{\partial K} \alpha_{\partial K}^n (\bu_h^i -\bu_h^e) \varphi_i \diff s, 
\end{aligned}
\end{align}
with the notation
\begin{equation}
\label{def_of_cij_bulk_and_bnd}
\bc_{ij}^{K}:= \int_{K}\varphi_i \GRAD\varphi_j \diff \bx,\qquad
\bc_{ij}^{\partial } := \tfrac{1}{2} \int_{\partial K} \varphi_{j} \varphi_i \bn_K \diff s,
\end{equation}
The three summations in \eqref{algedivh} represent a consistent
discretization of the divergence operator. In order to condense these
three summations into a single one, and after noticing that $j$ can belong
to only one of three possible (disjoint) subsets:
$\calI(K)\backslash\calI(\partial K^{\mathsf{i}})$,
$\calI(\partial K^{\mathsf{i}})$ or $\calI(\partial K^{\mathsf{e}})$,
we define the vector $\bc_{ij}$ by setting:
\begin{equation}
\label{def:cijCG}
\bc_{ij} :=
 \begin{cases} 
\bc_{ij}^K                    &\text{if }j \in \calI(K)\backslash\calI(\partial K^{\mathsf{i}}), \\
(\bc_{ij}^K-\bc_{ij}^{\partial})&\text{if }j \in \calI(\partial K^{\mathsf{i}}), \\
\bc_{ij}^{\partial}            &\text{if }j \in \calI(\partial K^{\mathsf{e}}).
\end{cases}
\end{equation}
Therefore, \eqref{algedivh} can be rewritten as follows:
\begin{align}
\label{algedivh01}
\begin{aligned}
  \int_K \nabla\SCAL(\polf(\bu_h)) \varphi_i\diff \bx \approx
  \sum_{j\in \calI(i)} \polf(\bsfU_j^n)\SCAL \bc_{ij} + \int_{\partial
    K} \alpha_{\partial K}^n (\bu_h^i -\bu_h^e) \varphi_i \diff s.
\end{aligned}
\end{align}

\begin{lemma} The set of coefficients $\{\bc_{ij}\}_{j \in \calI(i)}$
  defined in \eqref{def:cijCG} satisfy the conservation properties
  \eqref{cijprop}.
\end{lemma}

\begin{proof} Let us start by proving the skew-symmetry property.
Notice that \eqref{def_of_cij_bulk_and_bnd} is equivalent to
\[
\bc_{ij} :=
 \begin{cases} 
\bc_{ij}^K-\bc_{ij}^{\partial}&\text{if }j \in \calI(K), \\
\bc_{ij}^{\partial}            &\text{if }j \in \calI(\partial K\upe).
\end{cases}
\]
Let $j\in \calI(K')$. Assume first that $K=K'$, then
$\bc_{ij} = \bc_{ij}^K-\bc_{ij}^{\partial}$.  An integration by parts
gives $\bc_{ij} = -\bc_{ji}^K + \bc_{ij}^{\partial}$, which implies
that $\bc_{ij} = - \bc_{ji}$ because $i\in \calI(K')$.  Assume now
that $K\ne K'$ but $i\in \calI(\partial {K'}\upe)$, then
$\bc_{ij} = \bc_{ij}^{\partial}$. But $\bn_{K} = -\bn_{K'}$, hence
$\bc_{ij}^{\partial} = - \bc_{ji}^{\partial}$, which means that
$\bc_{ij} = -\bc_{ji}$ because $i\in \calI(\partial {K'}\upe)$.
    
Let us now prove that $\sum_{j\in \calI(i)} \bc_{ij} = \bzero$.  Using
that
$\calI(K)\backslash\calI(\partial K\upi)$, $\calI(\partial K\upi)$,
$\calI(\partial K\upe)$ is a partition of $\calI(i)$ and definition \eqref{def:cijCG} we have that
\begin{align*}
\sum_{j\in \calI(i)} \bc_{ij} &= 
\sum_{j \in \calI(K)\backslash\calI(\partial K\upi)} \bc_{ij}
+ \sum_{j \in \calI(\partial K\upi)} \bc_{ij}
+ \sum_{j \in \calI(\partial K\upe)} \bc_{ij} \\
&= \sum_{j \in \calI(K)\backslash\calI(\partial K\upi)} \bc_{ij}^K
+ \sum_{j \in \calI(\partial K\upi)} (\bc_{ij}^K - \bc_{ij}^{\partial})
+ \sum_{j \in \calI(\partial K\upe)} \bc_{ij}^{\partial} \\
&= \sum_{j \in \calI(K)} \bc_{ij}^K
- \sum_{j \in \calI(\partial K\upi)} \bc_{ij}^{\partial}
+ \sum_{j \in \calI(\partial K\upe)} \bc_{ij}^{\partial}. 
\end{align*}
the partition of unity property on $K$ (see
\eqref{boundary_partition_unity}) implies that
$\sum_{j \in \calI(K)} \bc_{ij}^K = \boldsymbol{0}$. The partition of
unity property on $\partial K$ (see \eqref{boundary_partition_unity})
implies that
$\sum_{j \in \calI(\partial K\upi)} \bc_{ij}^{\partial} =
\int_{\partial K} \varphi_i \bn_K \diff s$
and
$\sum_{j \in \calI(\partial K\upe)} \bc_{ij}^{\partial} =
\int_{\partial K} \varphi_i \bn_K \diff s$;
hence, the last two summations cancel each other. This completes the proof.
\end{proof}

\subsection{Graph viscosity for dG}
It is important to notice at this stage, that the formulation of the
viscous fluxes
$\int_{\partial K} \alpha_{\partial K}^n (\bu_h^{n,\textup{i}}
-\bu_h^{n,\textup{e}}) \varphi_i \diff s$ in \eqref{algedivh01} is not
compatible with our pursuit of a purely algebraic formulation. Note that 
the dissipation in \eqref{algedivh01} is active only on $\partial K$ and
there is no dissipation in the bulk of $K$, at least when the
polynomial degree of the approximation is larger than or equal to
$1$.
More precisely, assume for the sake of simplicity that
$\alpha_{\partial K}^n$ is constant over $\partial K$. Let us asume
also that the shape functions are Lagrange-based and let
$\{\bx_i\}_{i\in\calV}$ be the Lagrange nodes associated with
$\{\varphi_i\}_{i\in\calV}$. Then using the quadrature generated by
the Lagrange nodes, one can legitimately approximate the integral
$\int_{\partial K} \alpha_{\partial K}^n (\bu_h^{n,\textup{i}}
-\bu_h^{n,\textup{e}} ) \varphi_i\diff s$ by
$m_i^\partial\alpha_{\partial K}^n (\bu_h^{n,\textup{i}} (\bx_i) -
\bu_h^{n,\textup{e}}(\bx_i))$, where
$m_i^\partial=\int_{\partial K} \varphi_i \diff s$.  This means that
if $i$ and $j$ are in $\calI(K)$ and $i\ne j$ (which we can assume
since the polynomial degree is at least $1$), then the ``stabilizing''
term
$\int_{\partial K} \alpha_{\partial K}^n (\bu_h^{n,\textup{i}}
-\bu_h^{n,\textup{e}} ) \varphi_i\diff s$ does not contain any term
proportional to
$\bu_h^{n,\textup{i}} (\bx_i) - \bu_h^{n,\textup{i}}(\bx_j)$. That is
to say, the traditional dG stabilization does not contain any
stabilizing mechanism between the degrees of freedom that are internal
to $K$. It is at this very point that we depart from the traditional
dG formulation: we replace
$\int_{\partial K} \alpha_{\partial K}^n(\bu_h^i -\bu_h^e) \varphi_i
\diff s$ by the graph Laplacian
$- \sum_{j\in \calI(i)} d_{ij}\upLn (\bsfU^n_j-\bsfU^n_i)$ which
accounts for any possible interactions inside $K$ and with the
exterior traces on $\partial K$.  Therefore, we finally replace
\eqref{algedivh01} by
\begin{align}
\label{algedivh_dg}
\begin{aligned}
  \int_K \nabla\SCAL(\polf(\bu_h)) \varphi_i\diff \bx \approx
  \sum_{j\in \calI(i)} \polf(\bsfU_j^n)\SCAL \bc_{ij} - \sum_{j\in
    \calI(i)} d_{ij}\upLn (\bsfU^n_j-\bsfU^n_i),
\end{aligned}
\end{align}
and, thus modified, the final dG scheme exactly matches the generic form of the abstract
scheme \eqref{def_dij_scheme}.

\section{Runge Kutta SSP time integration} \label{Sec:SSP_overview}
Increasing the time accuracy while keeping the invariant domain
property can be done by using so-called Strong Stability Preserving
(SSP) time discretization methods.  The key idea is to achieve
higher-order accuracy in time by making convex combinations of forward
Euler steps. More precisely each time step of a SSP method is
decomposed into substeps that are all forward Euler steps; the final
update is constructed as a convex combination of the intermediate
solutions. This section is meant to be a brief overview of SSP
methods; we refer the reader to
\citet{Ferracina_Spijker_2005,Higueras_2005,Gottlieb_Ketcheson_Shu_2009}
for more detailed reviews. The main result of this section is the
Shu-Osher Theorem~\ref{Th:Shu_Osher}. Our formulation of the
  result is slightly different from the original statement to
  emphasize that this result is only about convexity (\ie it does not
  involve any norm, seminorm, or convex functional). The reader familiar with
  this material is invited to move to
  \S\ref{Sec:High_order_methods}.

\subsection{SSPRK methods}
We are going to illustrate the SSP concept with explicit Runge Kutta
methods. Let us consider a finite-dimensional vector space $E$, a
subset $A\subset E$ and a (nonlinear) operator
$L:[0,T]\CROSS A\longrightarrow E$.  We are interested in
approximating in time the following problem
$\partial_t u +L(t,u)=0$ with appropriate initial condition.  We assume that this system of ordinary differential
equations makes sense
(for instance $L$ is continous w.r.t. $t$ and Lipschitz w.r.t. $u$).
We further assume that there exists a convex
subset $B\subset A$ and $\dt_{\max}>0$ such that
\begin{equation}
v+ \dt L(t,v)\subset B,\qquad \forall v\in B,\ \forall t\in [0,T],\  
\forall \dt \in [0,\dt_{\max}]. \label{Forward_Euler_assumption}
\end{equation}
Consider a general $s$ stages, explicit Runge--Kutta method
identified by its Butcher tableau composed of a matrix
$(a_{ij})_{\{1\le i,j\le s\}}\in \Real^{s\CROSS s}$ and a vector $(b_j)_{\{1\le j\le s\}}\in\Real^s$
\begin{equation}
\begin{array}{c|ccccc}
0                                               \\
c_2    & a_{21}                                  \\
c_3    & a_{31}  & a_{32}                         \\
\vdots & \vdots &        & \ddots &             \\
c_s    & a_{s1}  & a_{s2}  & \cdots & a_{s,s-1}     \\
\hline
       & b_1    & b_2    & \cdots & b_{s-1} & b_s
\end{array} \label{Butcher_tableau}
\end{equation}
where $c_i:=\sum_{j=1}^{i-1} a_{ij}$ for $i\in\intset{2}{s}$. Let us
also set $c_1:=0$.
Let us assume now that the above $s$ stages, explicit Runge Kutta
method has the following $(\alpha-\beta)$ representation: There are
real coefficients $\alpha_{ki}$, $\beta_{ki}$ with
$k\in\intset{0}{i-1}$ and $i\in\intset{1}{s}$ such that $u^{n+1}$
is obtained by first setting $w^{(0)} = u^n$, then computing
\begin{align} 
w^{(i)} &= \sum_{k=0}^{i-1} \alpha_{ik}w^{(k)} +
\beta_{ik}\dt L(t^n+\gamma_{k}\dt,w^{(k)}),\quad i\in\intset{1}{s},\label{Shu_Osher_Representation_kth}
\end{align}
and finally setting $u^{n+1} = w^{(s)}$, where
$\sum_{0\le k\le i-1} \alpha_{ik}=1$, $\gamma_k:=c_{k+1}$,
$\alpha_{ik}\ge 0$, and $\beta_{ik}\ge 0$, for all
$k\in\intset{0}{i-1}$ and all $i\in\intset{1}{s}$.  We further assume
that $\beta_{ik}=0$ if $\alpha_{ik}=0$, $k\in\intset{0}{i-1}$,
$i\in\intset{1}{s}$.  Not every $s$ stages, explicit Runge Kutta
method admits an $(\alpha-\beta)$ representation. Any Runge Kutta
method that admits an $(\alpha-\beta)$ representation as defined above
is said to be SSP for a reason that will be stated in
Theorem~\ref{Th:Shu_Osher}.

\begin{example}[Midpoint rule]
  The midpoint rule, defined by the Butcher tableau
\begin{equation}
\begin{array}{c|ccccc}
 0                                               \\ 
 \frac12    & \frac12                   \\[2pt] \hline
  & 0  & 1                       
\end{array} \label{mid_point_Butcher_tableau}
\end{equation}
does not have a legitimate $(\alpha-\beta)$ representation, since it
would require that $\beta_{20}+\alpha_{21}=0$, which in turn would
imply that either $\beta_{20}<0$ or $\alpha_{21}<0$.  
\end{example}

\begin{example}[SSPRK(2,2)] Heun's method, which is 
a second-order Runge--Kutta technique composed of two stages, is
SSP. It has the following $(\alpha-\beta)$ tableau and can be implemented as follows:
\[
\begin{array}{|cc|cc|c|c|}
\hline
 \alpha         &          &  \beta    &             &\gamma& c_{\textup{os}} \\ \hline\hline
\ST    1        &           &   1         &             &   0        &    \\
\ST  \frac12&\frac12&   0        & \frac12 &   1        & 1 \\[2pt] \hline
\end{array}\qquad\qquad
\begin{aligned}
w^{(1)}&= u^n\! + \dt L(t^n,u^n), \\
w^{(2)}&= w^{(1)}\! + \dt L(t^{n+1},w^{(1)}),\\
u^{n+1}&= \tfrac12 u^n\! + \tfrac12 w^{(2)}. 
\end{aligned}
\]
\par \vspace{-1.2\baselineskip}
\end{example}

\begin{example}[SSPRK(3,3), SSPRK(4,3)] The following Runge--Kutta methods, which are
third-order and composed of three substeps and four substeps, respectively, are SSP:
\[
\begin{array}{|ccc|ccc|c|c|}
\hline
                    &\alpha &          &        & \beta    &            &\gamma& c_{\textup{os}} \\ \hline\hline
\ST    1        &           &           & 1     &             &            &0           &    \\
\ST  \frac34&\frac14&           & 0     & \frac14&            &1           & 1 \\
\ST  \frac13&  0       &\frac23& 0     &0           &\frac23&\frac12  &    \\[2pt] \hline
\end{array}
\qquad\qquad
\begin{array}{|cccc|cccc|c|c|}
\hline
                   &\alpha&          &  &            & \beta  &           &            &\gamma& c_{\textup{os}} \\ \hline\hline
\ST    1       &          &           &  &\frac12&            &           &            &0           &    \\
\ST    0       &1        &           &  &0          &\frac12&           &            &\frac12 & 2 \\
\ST \frac23&0        &\frac13&  &0          &0         &\frac16&            &1          &    \\ 
\ST    0       &0       &0          &1&0          &0         &0          &\frac12&\frac12&    \\[2pt] \hline
\end{array}
\]
For instance the SSPRK$(3,3)$ method can be implemented as follows:
\begin{align*}
w^{(1)} &= u^n + \dt L(t^n,u^n), 
&&z^{(1)} = w^{(1)} + \dt L(t^n+\dt,w^{(1)}), \\
w^{(2)} &= \frac34 u^n + \frac14 z^{(1)},\quad
&&z^{(2)} = w^{(2)} + \dt L(t^n+\tfrac12\dt,w^{(2)}), \\
u^{n+1} &= \frac13 u^n + \frac23 z^{(2)}.\tag*{\qed}
\end{align*}
\end{example}

\subsection{The key result}
We
henceforth denote
\begin{equation}
c_\textup{os}:=\inf_{\{\alpha_{ik}\not=0,\ 1\le k+1\le i\le s\}}\alpha_{ik}\beta_{ik}^{-1}.
\end{equation}
The following theorem is the main result of this section.
\begin{theorem}[Shu-Osher] \label{Th:Shu_Osher} Assume that the Runge Kutta method
with the Butcher tableau \eqref{Butcher_tableau} is SSP. Let $B\subset A$ be convex.
Let $u^n\in B$ and assume that $\dt \le c_\textup{os} \dt_{\max}$, then
$u^{n+1}\in B$.
\end{theorem}
\begin{proof} 
  Let $n\ge 0$ and assume that $u^{n}\in B$. Let $i\in\intset{1}{s}$
  and assume that $w^{(k)}\in B$ for all $k\in\intset{0}{i-1}$. Note
  that this assumption is satisfied for $i=1$ since
  $w^{(0)} = u^n\in B$. Consider the $k$th term in
  \eqref{Shu_Osher_Representation_kth}, $0\le k\le i-1$.  If
  $\alpha_{ik}=0$ then $\beta_{ik}=0$ by construction, and there is
  nothing to sum. Assume now that $\alpha_{ik}>0$. Let us denote
  $r_{ik}:=\beta_{ik}/\alpha_{ik}$ and
  $z^{(i,k)}:=w^{(k)} + r_{ik}\dt L(t^n+\gamma_{k}\dt,w^{(k)})$, then
  the condition $\dt \le c_\textup{os} \dt_{\max}$ implies that
  $r_{ik}\dt \le (\beta_{ik}/\alpha_{ik}) c_\textup{os} \dt_{\max} \le
  \dt_{\max} $,
  which, owing to \eqref{Forward_Euler_assumption}, is sufficient to
  ascertain that $z^{(i,k)} \in B$ for all $k\in\intset{0}{i-1}$.
  Observing that $w^{(i)}=\sum_{k=1}^{i-1} \alpha_{ki} z^{(i,k)} $,
  the condition $\sum_{0\le k\le i-1} \alpha_{ik}=1$ together with
  $0\le \alpha_{ik}$, $0\le k\le i-1$, implies that $w^{(i)}$ is a
  convex combination of $z^{(i,0)}, \ldots, z^{(i,i-1)}$; hence
  $w^{(i)}$ is in $B$ since $B$ is convex. In conclusion
  $w^{(k)}\in B$ for all $k\in\intset{0}{i}$ and all
  $i\in\intset{1}{s}$, thereby proving that $u^{n+1}= w^{(s)}\in B$.
\end{proof}

\begin{remark}[Literature] Theorem~\ref{Th:Shu_Osher} has been
  established in a slightly different form in
  \citet[Prop.~2.1]{Shu_Osher1988} not explicitly invoking
  convexity. Although our proof is very similar to that in
  \citep{Shu_Osher1988}, the statement of Theorem~\ref{Th:Shu_Osher}
  is slightly different since it only involves convexity; no norm or
  seminorm (as in \cite[p.~92]{Gottlieb_Shu_Tadmor_2001}), or convex
  functional (as in \citep[Eq.~(1.3)]{Gottlieb_Ketcheson_Shu_2009}) is
  involved.  This variant of the theorem does not seem to be very
  well known.
\end{remark}

\begin{remark}[Structure of $B$]\label{Rem:Tensor_structure_of_B}
In the original paper \citep{Shu_Osher1988} and in
  \citep{Gottlieb_Shu_Tadmor_2001}, $E$ is a normed vector space
  equipped with some norm $\|\cdot\|_E$.  The
  assumption~\eqref{Forward_Euler_assumption} then consists of stating
  that $\polI+ \dt L(t,\cdot)$ maps any ball $B$ centered at $0$ into $B$ for any
  $s \in [0,\dt_{\max}]$ and any $t\in [0,T]$. In particular taking
  any $v\in E$ and defining $B$ to be the ball of radius $\|v\|_B$
  centered at $0$, the assumption~\eqref{Forward_Euler_assumption}
  amounts to saying that $\|v+ \dt L(t,v)\|_B \le \|v\|_B$, which is
  Eq.~(1.3) in \citep{Gottlieb_Shu_Tadmor_2001}. The norm that is used
  in \citep{Shu_Osher1988} is the total variation. In the present
  paper the assumption~\eqref{Forward_Euler_assumption} is more
  general. We are going to use it with the following structure: we are
  going to assume that there are two positive integers
  $\Nglob,m\in\polN{\setminus}\{0\}$ such that
  $E=(\Real^m)^\Nglob$. Here $\Real^m$ is called the phase space. Then
  we assume that there is convex subset of the phase space
  $\calB\subset \Real^m$ such that the
  assumption~\eqref{Forward_Euler_assumption} holds with
  $B:=(\calB)^\Nglob$. All the convex arguments invoked in the rest of
  the paper extends to SSP RK techniques with this particular
  structure.
\end{remark}

\section{High-order method} \label{Sec:High_order_methods} The
algorithm that we are going to develop in \S\ref{Sec:convex_limiting}
relies on the construction of the low-order invariant domain
preserving solution $\bsfU_i\upLnp$ described in
\S\ref{sec:abstract_low_order}-\S\ref{secInvVisco} and a high-order
solution $\bsfU_i\upHnp$ that possibly wanders outside the invariant
domain. We are then going to limit the high-order solution by pushing
it back into the invariant domain in the direction of the low-order
solution. This limiting technique, which we call \emph{convex
  limiting}, will be explained in \S\ref{Sec:convex_limiting}. The
purpose of the present section is to present various ways to construct
$\bsfU_i\upHnp$.

\subsection{Achieving high-order consistency}
In this section we describe in broad terms how high-order consistency can be achieved.
\subsubsection{Discretization-independent setting}
Independently of the space discretization that is used, we henceforth assume that
the high-order update $\bsfU_i\upHnp$ is computed as follows:
\begin{equation}
  \frac{m_i}{\dt}(\bsfU_i\upHnp -  \bsfU_i^{n})  + \sum_{j\in\calI(i)} \bsfF_{ij}\upHn=m_i \bS(\bsfU_i^n), 
  \label{abstract_high_order}
\end{equation}
where the high-order flux $\bsfF_{ij}\upHn$ is assumed to be
skew-symmetric; \ie $\bsfF_{ij}\upHn=-\bsfF_{ji}\upHn$ for all
$i\in\calV$, $j\in \calI(i)$ (under appropriate boundary conditions).
The skew-symmetry implies that the high-order update is conservative;
\ie
$\sum_{i\in\calV} m_i \bsfU_i\upHnp = \sum_{i\in\calV} m_i \bsfU_i^n$
if $\bS\equiv \bzero$.  The expression \eqref{abstract_high_order} is
the only information regarding the high-order update that will be
necessary for the convex limiting technique to be presented in
section~\S\ref{Sec:convex_limiting}.


There are many different techniques to compute high-order consistent
fluxes $\bsfF_{ij}\upHn$ which depend on the space discretization of
choice. For the sake of completeness, we list some of those in
\S\ref{sec:algefluxesFV}, \S\ref{sec:algefluxesCG}, and
\S\ref{sec:algefluxesDG}.  None of this material is essential to
understand the convex limiting technique explained in
\S\ref{Sec:convex_limiting}.


\subsubsection{High-order algebraic fluxes: Finite Volumes}\label{sec:algefluxesFV} In
the context of finite volume schemes, high-order algebraic fluxes
$\bsfF_{ij}\upHn$ are obtained as integrals of high-order numerical fluxes
over the interfaces between volumes, \ie
$\bsfF_{ij}\upHn := \int_{\Gamma_{ij}} \widehat{\polf}\bn_{ij} \diff s$ where
$\widehat{\polf}\bn_{ij}$ is some numerical flux. 
For instance, a widely popular choice of algebraic flux consists of setting:
\begin{align}\label{FijHighFV}
  \bsfF_{ij}\upHn := 
  \int_{\Gamma_{ij}} \left(\tfrac{1}{2} (\polf(\bu_h^{\high,n,\mathsf{i}}) 
  + \polf(\bu_h^{\high,n,\mathsf{e}}) ) \SCAL \bn_{ij} 
  + d_{ij}\upLn (\bu_h^{\high,n,\mathsf{i}} - \bu_h^{\high,n,\mathsf{e}}) \right)\diff s, 
\end{align}
where the superscripts $\mathsf{e}$ and $\mathsf{i}$ denote the
exterior and interior traces respectively, and
$\bu_h\upHn$ is a discontinuous piecewise
polynomial reconstruction (of degree at most $k$) recovered from the
piecewise constant solution
$\bu_h^{n} = \sum_{j\in\calV} \bsfU_j^{n} \mathbb{I}_{K_j}$
satisfying the conservation constraint
$\frac{1}{|K_i|}\int_{K_i} (\bu_h\upHn -\bu_h^{n})\diff \bx = 0$.
More precisely
$\bu_h^{\high,n,\mathsf{i}}(\bx) = \lim_{K_i\ni \by \to \bx}
\bu_h\upHn (\by)$ and
$\bu_h^{\high,n,\mathsf{e}}(\bx) = \lim_{K_j\ni \by \to \bx}
\bu_h\upHn (\by)$.  In practice, \eqref{FijHighFV} has to be
computed using quadrature on the faces of the element.  The choices of
numerical flux $\widehat{\polf} \bn_{ij}$ and reconstruction
$\bu_h\upHn$ that could be used in \eqref{FijHighFV} are not
unique. There is a massive body of literature on this topic and it is
well beyond the scope of the current paper to elaborate further in
this direction; we refer the reader to
\cite{Barth2004,Kroner1997,Morton2007} for additional background. For
the purpose of the present paper, we are only going to assume that
\eqref{abstract_high_order} holds with skew-symmetric algebraic fluxes
$\bsfF_{ij}\upHn$.

\subsubsection{High-order algebraic flux: Continuous Finite Elements}\label{sec:algefluxesCG} 
We now turn our attention to continuous finite elements. In
  this case high-order consistency can be achieved by using a degenerate
  graph viscosity $d_{ij}\upHn$  such that $d_{ij}\upHn \ll d_{ij}\upLn$ in smooth
regions while $d_{ij}\upHn \approx d_{ij}\upLn$ near shocks. Of course 
$d_{ij}\upHn$ must also satisfy the conservation constraints
\begin{align}
\label{dijHprop}
d_{ij}\upHn = d_{ji}\upHn\ge 0 \quad \text{if $i\not=j$,}  
\quad \text{and} \quad \sum_{j\in \calI(i)} d_{ij}\upHn =0.
\end{align} 
The algebraic flux looks as the one defined in
\eqref{defFijabstract} for the low-order method;
the only difference here is that we use the high-order
viscosities $\{d_{ij}\upHn\}_{j \in \calI(i)}$:
\begin{align}
\label{defFijCGlumped}
\bsfF_{ij}\upHn := ( \polf(\bsfU_j^n) + \polf(\bsfU_i^n) )\bc_{ij}
- d_{ij}\upHn (\bsfU^n_j - \bsfU^n_i).
\end{align}

Higher-order accuracy in space can also be obtained by using the
consistent mass matrix instead of the lumped mass matrix for the
discretization of the time derivative. By reducing dispersive errors,
this technique is known to yield superconvergence at the grid points;
see \cite{FLD:FLD719,Guermond_pasquetti_2013}. In this case the
high-order update is computed by solving the following mass matrix
problem:
\begin{align}
\label{CGconsHigh}
\sum_{j\in \calI(i)} \!\!\frac{m_{ij}}{\dt}(\bsfU_j\upHnp\!-\bsfU_j^n)
+ ( \polf(\bsfU_j^n) + \polf(\bsfU_i^n) )  \bc_{ij}
- d_{ij}\upHn (\bsfU^n_j - \bsfU^n_i) = m_i \bS(\bsfU_i^n).
\end{align}
Noticing that $m_{ij} = \delta_{ij} m_i + m_{ij} - \delta_{ij} m_i$,
we can rewrite \eqref{CGconsHigh} as
\begin{align}
\label{CGconsHigh01}
\begin{aligned}
\frac{m_i}{\dt} (\bsfU_i\upHnp-\bsfU_i^n) 
&+ \sum_{j\in \calI(i)} \frac{(m_{ij} - \delta_{ij} m_i)}{\dt}(\bsfU_j\upHnp-\bsfU_j^n) \\
&+ ( \polf(\bsfU_j^n) + \polf(\bsfU_i^n) )\bc_{ij}
- d_{ij}\upHn (\bsfU^n_j - \bsfU^n_i) = m_i \bS(\bsfU_i^n).
\end{aligned}
\end{align}
Since $\sum_{j\in \calI(i)} (m_{ij} - \delta_{ij} m_i) = 0$, we add
$- \sum_{j\in \calI(i)} \frac{(m_{ij} - \delta_{ij} m_i)}{\dt}
(\bsfU_i\upHnp-\bsfU_i^n) = 0$ to the identity \eqref{CGconsHigh01} to get
\begin{align}
\label{CGconsHigh02}
\begin{aligned}
  \frac{m_i}{\dt} (\bsfU_i\upHnp-\bsfU_i^n)
  &+ \sum_{j\in \calI(i)\backslash\{i\}} \frac{(m_{ij} 
- \delta_{ij} m_i)}{\dt}(\bsfU_j\upHnp-\bsfU_j^n  - \bsfU_i\upHnp + \bsfU_i^n ) \\
&+ ( \polf(\bsfU_j^n) + \polf(\bsfU_i^n) ) \bc_{ij}
- d_{ij}\upHn (\bsfU^n_j - \bsfU^n_i) = m_i \bS(\bsfU_i^n). 
\end{aligned}
\end{align}
Then \eqref{abstract_high_order} holds with the following definition
for the high-order algebraic flux:
\begin{align}\label{defFijCGconsistent}
\bsfF_{ij}\upHn :=&{} \frac{m_{ij}}{\dt}(\bsfU_j\upHnp-\bsfU_j^n  
- \bsfU_i\upHnp + \bsfU_i^n )  \\
 & + (\polf(\bsfU_j^n) + \polf(\bsfU_i^n) ) \bc_{ij} -
d_{ij}\upHn (\bsfU^n_j - \bsfU^n_i). \nonumber
\end{align}
In the context of finite difference methods, a scheme with the above structure 
is said to be linearly implicit as the numerical fluxes depend linearly 
on the state $\bsfU_j\upHnp$.

We finally mention a third approach which has antidispersive
properties that are similar to \eqref{CGconsHigh} but does not
require solving a mass matrix problem a each time step. This method
consists of approximating the inverse of $\calM$ by
$(\calM^L)^{-1}(\calI + (\calM^L-\calM)(\calM^L)^{-1})$, where $\calI$
is the identity matrix. We refer the reader to
\cite[\S3.3]{Guermond_Nazarov_Popov_Yong_2014} for the details.

\subsubsection{High-order algebraic flux: Discontinuous Finite Elements}
\label{sec:algefluxesDG} Just like for continuous finite
  elements, high-order consistency is space is obtained for
  discontinuous finite elements by replacing the low-order graph
  viscosity $d_{ij}\upLn$ by a high-order graph viscosity
  $d_{ij}\upHn$ satisfying the symmetry and positivity properties
  stated in~\eqref{dijHprop}. The corresponding flux in
  \eqref{abstract_high_order} is
\begin{align} \label{defFijdGlumped}
  \bsfF_{ij}\upHn := 
  (\polf(\bsfU_j^n) + \polf(\bsfU_i^n) ) \bc_{ij}
  - d_{ij}\upHn (\bsfU^n_j - \bsfU^n_i),\qquad  \forall \calI(K)\cup \calI(K\upe).
\end{align}

Like for continuous elements, superconvergence can be obtained by 
using the consistent mass matrix. A high-order discontinuous finite
element scheme using the consistent mass matrix can be written as
follows:
\begin{multline}
  \sum_{j\in \calI(K)}
  \frac{m_{ij}}{\dt}(\bsfU_j\upHnp-\bsfU_j^n)
  + \sum_{j\in \calI(i)} (\polf(\bsfU_j^n) + \polf(\bsfU_j^n)) \bc_{ij} \\
- d_{ij}\upHn (\bsfU^n_j - \bsfU^n_i) = m_i\bS(\bsfU_i^n).
\end{multline}
Notice that the mass matrix only involves the dofs in $\calI(K)$.  As
in the continuous case, noting that
$m_{ij} = \delta_{ij} m_i + m_{ij} - \delta_{ij} m_i$, using the
partition of unity properties, and proceeding as in
\eqref{CGconsHigh01}-\eqref{CGconsHigh02}), we obtain the following
definition for the high-order flux $\bsfF_{ij}\upHn$ that is used in
\eqref{abstract_high_order}:
\begin{align} \label{defFijdGconsistent}
\bsfF_{ij}\upHn := 
\begin{cases}
\dfrac{m_{ij}}{\dt}(\bsfU_j\upHnp-\bsfU_j^n  - \bsfU_i\upHnp + \bsfU_i^n ) & \\[2pt]
\qquad + (\polf(\bsfU_j^n) + \polf(\bsfU_i^n) )  \bc_{ij}
- d_{ij}\upHn (\bsfU^n_j - \bsfU^n_i) &\ \ \text{if }j \in \calI(K), \\[4pt]
(\polf(\bsfU_j^n) + \polf(\bsfU_i^n) ) \bc_{ij}
- d_{ij}\upHn (\bsfU^n_j - \bsfU^n_i) &\ \ \text{if }j \in \calI(K\upe).
\end{cases}
\end{align}

\subsection{Smoothness-based graph viscosity}\label{sec:SmoothVisc}

The objective of this section is to present a method where the
high-order graph viscosity in \eqref{defFijCGlumped},
  \eqref{defFijCGconsistent}, \eqref{defFijdGlumped}, or
  \eqref{defFijdGconsistent} is obtained by estimating the smoothness
of some functional (\eg an entropy) of the current solution.

\subsubsection{Principles of the method} \label{Sec:Principles_smoothness_based_viscosity}
Let $\bu_h^n = \sum_{i\in\calV} \bsfU_i^n \varphi_i$ be the current
approximation and let $g:\calA \to \Real$ be some functional (examples
will be given below). We define the
smoothness indicator associated to $g$ as follows:
\begin{equation}
  \alpha_i^n:= \frac{\left|\sum_{j\in\calI(i)} 
\beta_{ij}(g(\bsfU_j^n)-g(\bsfU_i^n))\right|}{\max(\sum_{j\in\calI(i)} 
|\beta_{ij}||g(\bsfU_j^n)-g(\bsfU_i^n)|,\epsilon_i)},
  \label{def_of_alpha_in}
\end{equation}
with $\epsilon_i = \epsilon \max_{j\in \calI(i)} |g(\bsfU_j^n)|$, where
  $\epsilon$ is very small number. This term avoids degeneracy when
  $g(\bsfU_j^n)$ is constant for all $j\in\calI(i)$; see
  Remark~\ref{Rem:epsilon_smoothness}.  The real numbers $\beta_{ij}$
are selected to make the method linearity-preserving (see
\cite{Berger:2005:ASL} for a review on linearity-preserving limiters
in the finite volume literature). The reader is referred to
Remark~\ref{Rem:linearity_preserving} for the details.  Notice that
$\alpha_i^n\in [0,1]$ for all $i\in\calV$ and $\alpha_i^n=1$ if
$g(\bsfU_i)$ is a local extremum. This property will play an important
role in the proof of Theorem~\ref{Thm:max_principle_smoothness} which
is the main result of \S\ref{sec:SmoothVisc}.

We now define the high-order graph viscosity by setting 
\begin{equation}
d_{ij}\upHn := d_{ij}\upLn \max(\psi(\alpha_i^n),\psi(\alpha_j^n)), \label{dij_smoothness_indicator}
\end{equation}
where $\psi\in\text{Lip}([0,1];[01,])$ is any Lipschitz function from
$[0,1]$ to $[0,1]$ such that $\psi(1)=1$. One typical example is
$\psi(\alpha)=\big(\frac{\alpha-\alpha_0}{1-\alpha_0}\big)^q$ with
$q\ge 2$ and $\alpha_0\in [0,1)$. For instance one can take
$\alpha_0=\frac12$ and $q=4$. One need to be careful though not to
take $\alpha_0$ too close to $1$ and $q$ not too large since we will
see in Theorem~\ref{Thm:max_principle_smoothness} below that the
Lipschitz constant of $\psi$ plays a important role in the properties
of the method.

\begin{remark}[Choices for $\epsilon$] \label{Rem:epsilon_smoothness}
  Using double precision arithmetic, the regularization in
  \eqref{def_of_alpha_in} can be done with
  $\epsilon=10^{-\frac{16}{2}}$. We have also observed that
    using $\epsilon = (m_i/|D|)^{\frac{3}{d}}$ maintains the
    second-order accuracy properties of the method in any $L^q$-norm,
    $q\in[1,\infty]$.
\end{remark}

\begin{remark}[Linearity-preserving
  $\beta_{ij}$] \label{Rem:linearity_preserving} To be linearity-preserving with
  continuous finite elements one should obtain $\alpha_i^n=0$ if
  $g(\bu_h^n)$ is linear on the support of the shape function
  $\varphi_i$. One simple choice for continuous finite elements
  consists of setting
  $\beta_{ij} =\int_\Dom \GRAD \varphi_i \SCAL\GRAD\varphi_j \diff x$
  (for the time being we do not require $\beta_{ij}>0$ in
  \eqref{def_of_alpha_in}).  For discontinuous elements, one could
  take
  $\beta_{ij} = \int_{K}\GRAD \varphi_j\SCAL\GRAD\varphi_i\diff x -
  \int_{\partial K} \frac12 \GRAD \varphi_j \SCAL \bn_K \varphi_i
  \diff x$,
  where $K$ is the unique cell such that $i\in \calI(K)$ and $\bn_K$
  is the unit normal vector on $\partial K$ pointing outward $K$.  For
  finite volumes, one should get $\alpha_i^n=0$ if a linear
  reconstruction fits all the data
  $\{g(\bsfU_j^n)\}_{j\in\calI(i)}$. For instance, one can use the
  mean-value coordinates; see \cite[Eq.~5.1]{Floater_2015} for the
  details.  Let us finally remark that although using $\beta_{ij}=1$
  is not a priori linearity preserving, we have numerically verified
  that this choice works reasonably well on quasi-uniform meshes.
\end{remark}

If the coefficients $\beta_{ij}$ are defined so the
linearity-preserving property holds, then the numerator of
\eqref{def_of_alpha_in} behaves like
$h^2\|D^2 g(\bu(\bxi,t^n))\|_{\ell^2(\Real^{d\CROSS d})}$ at some
point $\bxi$, whereas the denominator behaves like
$h\|\GRAD g(\bzeta)\|_{\ell^2(\Real^d)}$ at some point
$\bzeta$. Therefore, we have
$ \alpha_i^n \approx h \|D^2 g(\bxi)\|_{\ell^2(\Real^{d\CROSS
    d})}/\|\GRAD g(\bzeta)\|_{\ell^2(\Real^d)}$,
that is to say $\alpha_i^n$ is of order $h$ in the regions where $g$ is
smooth and does not have a local extremum. This argument shows that
$d_{ij}\upHn$ is one order smaller than $d_{ij}\upLn$ (in terms of
mesh size). Hence it is reasonable to expect that the method using
$d_{ij}\upHn$ is formally second-order accurate in space.

\begin{example}[Choosing $g(\bsfU)$]
  In the context of the shallow water equations one can use the water
  height as smoothness indicator.  For the compressible Euler
  equations one can use the density.  We are going to prove
  stability properties for these two choices in
  Theorem~\ref{Thm:max_principle_smoothness}, (see also
  Example~\ref{Ex:max_principle_smoothness}). In general it is a good
  idea to choose $g(\bsfU)$ to be entropy associated
  with~\eqref{def:hyperbolic_system} (with or without the source
  term). We refer the reader to \cite{GuerNazPopTom2017}, where a full
  set of tests is reported for the compressible Euler equations with the
  $\gamma$-law. The computations therein are done with 
  $g(\bsfU) = \frac{\rho}{\gamma-1} \log(e(\bsfU)\rho^{1-\gamma})$,
  where $e(\bsfU)$ is the specific internal energy
\end{example}

\subsubsection{Stability} \label{Sec:bounds_smoothness_based_viscosity} 
We now establish some invariant domain preserving properties
associated with the smoothness-based
graph viscosity~\eqref{def_of_alpha_in} when the coefficients $\beta_{ij}$ are positive.  
We further specialize the setting
by assuming that $g:\calA\to \Real$ is a projection onto one of the
scalar components of $\bsfU$. Without loss of generality we set
$g(\bsfU)=\sfU_1$ with the convention
$\bsfU:=(\sfU_1,\ldots,\sfU_m)\tr$.  From now on, we drop the index
$_1$ to simplify the notation; that is, we set $g(\bsfU)=\sfU$.  We
denote by $S:\calA \to \Real$ the corresponding scalar component of the source
$\bS$.  One important assumption in this section is that $S\equiv 0$,
\ie the scalar component of the source acting on $\sfU$ is zero.

We have seen in Theorem~\ref{Thm:UL_is_invariant} that the auxiliary
states $\overline{\bsfU}_{ij}^{n}$ defined in \eqref{def_barstates} 
play an important role in the stability analysis.
These states are such that if $\bsfU_i^n , \bsfU_j^n \in \calB$,
where $\calB\subset \calA$ is some
convex invariant set, then $\overline{\bsfU}_{ij}^{n}\in \calB$,
provided that $1 +  \frac{2 \dt d_{ii}\upLn}{m_i}\geq 0$,
and the low-order graph viscosity
$d_{ij}\upLn$ is defined as in \eqref{Def_of_dij}. We denote by
$\overline{\sfU}_{ij}^{n}$ the scalar component of
$\overline{\bsfU}_{ij}^{n}$ that is of interest to us. Then we set
\begin{equation}
\sfU_i^{\umax,n} := \max_{j\in \calI(i)} \overline{\sfU}_{ij}^n,\qquad
\sfU_i^{\umin,n} := \min_{j\in \calI(i)} \overline{\sfU}_{ij}^n. \label{def_of_min_max_of_bar_states}
\end{equation}
 We set
$\calI(i^+):=\{j\in\calI(i)\st \sfU_i^n< \sfU_j^n\}$ and
$\calI(i^-):=\{j\in\calI(i)\st \sfU_j^n<\sfU_i^n\}$.
To simplify the notation we set
\begin{equation}
\gamma_i^n := -\frac{2\dt d_{ii}\upLn}{m_i},\qquad
\gamma_i^{+,n} := \frac{2\dt}{m_i}\sum_{j\in\calI(i^+)} d_{ij}\upLn,\qquad 
\gamma_i^{-,n} := \frac{2\dt}{m_i}\sum_{j\in\calI(i^-)} d_{ij}\upLn.
\label{def_of_gammam_gammap}
\end{equation} 
The following key ``gap lemma'' will be invoked later. 
\begin{lemma}[Gap estimates] \label{Lem:lower_upper_bounds} Let
  $n\ge 0$, and $i\in\calV$. 
We define  the gap parameter
\begin{equation}
\theta_i^n := \frac{\sfU^n_i-\sfU^{\umin,n}_i}{\sfU_i^{\umax,n}-\sfU^{\umin,n}_i},
\ \text{if $\sfU_i^{\umax,n}-\sfU^{\umin,n}_i\ne 0$}; \qquad \qquad 
\theta_i^n := \frac12, \ \text{otherwise}. 
\label{def_f_theta_in}
\end{equation}
Assume that $\gamma_i^n <1$. Let $\bsfU_i^{n+1}$ be the high-order update
  given by \eqref{abstract_high_order} using either the high-order cG
  flux~\eqref{defFijCGlumped} or the high-order dG
  flux~\eqref{defFijdGlumped} with any graph viscosity $\{d_{ij}\upHn\}_{j\in\calI(i)\backslash\{i\}}$ defined 
by $d_{ij}\upHn := d_{ij}\upLn \max(\psi_i^n,\psi_j^n)$ with $\psi_i^n,\psi_j^n\in [0,1]$. Then,
\begin{align}
 \sfU_i^{n+1} 
 & \!\le \sfU_i^{\umax,n}\! - (\sfU_i^{\umax,n}\!-\sfU_i^n) \left((1-\theta_i^n)(1-\gamma_i^n) 
  - \theta_i^n(1-\psi_i^n)\tfrac12 \gamma_i^{-,n}\right)\!, \label{Eq:Lem:lower_upper_bounds:up}\\
\sfU_i^{n+1} 
& \!\ge \sfU_i^{\umin,n}\! + (\sfU_i^{\umax,n}\!-\sfU_i^n) \left(\theta_i^n(1-\gamma_i^n) 
  -(1- \theta_i^n)(1-\psi_i^n)\tfrac12 \gamma_i^{+,n}\right)\!. \label{Eq:Lem:lower_upper_bounds:down}
\end{align}
\end{lemma}
\begin{proof} There is nothing to prove if
  $\sfU_i^{\umax,n} -\sfU_i^{\umin,n}=0$. Let us now assume that
  $\sfU_i^{\umax,n} -\sfU_i^{\umin,n}\not =0$. Subtracting
  \eqref{def_dij_scheme} from \eqref{abstract_high_order} we obtain
\[
 \bsfU_i\upHnp=\bsfU_i\upLnp 
  + \frac{\dt}{m_i}\sum_{j\in \calI(i)} (d_{ij}\upHn  - d_{ij}\upLn)(\bsfU^n_j -\bsfU^n_i).
\]
Let us focus on the scalar component $\sfU_i^n$. Recalling the auxiliary
states $\overline\sfU_{ij}^{n}$ defined in \eqref{def_barstates} and
  recalling that we have assumed $S\equiv 0$,
  the identity \eqref{def_dij_scheme_convex_with_source} gives
  $\sfU_i\upLnp = (1 -\gamma_i)\sfU_i^n + \sum_{j\in
    \calI(i){\setminus}\{i\}}\frac{2\dt
    d_{ij}\upLn}{m_i}\overline\sfU_{ij}^{n}$.
  Then setting
  $\sfU_i^{*,n} := \frac{1}{\gamma_i^n}\sum_{j\in
    \calI(i){\setminus}\{i\}}\frac{2\dt
    d_{ij}\upLn}{m_i}\overline\sfU_{ij}^{n}$,
  we have
  $\sfU_i\upLnp = (1-\gamma_i^n)\sfU_i^n + \gamma_i^n \sfU_i^{*,n}$,
  and this in turn implies that
\begin{align*}
 \sfU_i\upHnp & = (1-\gamma_i^n)\sfU_i^n + \gamma_i^n \sfU_i^{*,n}
  + \frac{\dt}{m_i}\sum_{j\in \calI(i){\setminus}\{i\}} (d_{ij}\upHn  -
  d_{ij}\upLn)(\sfU^n_j -\sfU^n_i).
\end{align*}
\textup{(ii)} Using that
$\sfU_{i}^{*,n}\in
\text{conv}\{\overline\sfU_{ij}^{n}\}_{j\in\calI(i){\setminus}\{i\}}$,
we have $\sfU_{i}^{*,n} \le \sfU_i^{\umax,n}$, and we infer that
\begin{align*}
\sfU_i\upHnp& \le \sfU_i^{\umax,n} + (\sfU_i^n -\sfU_i^{\umax,n}) (1-\gamma_i^n) 
  + \frac{\dt}{m_i}\sum_{j\in \calI(i){\setminus}\{i\}} (d_{ij}\upHn  -
  d_{ij}\upLn)(\sfU^n_j -\sfU^n_i).
\end{align*}
Then using that $d_{ij}\upHn \le d_{ij}\upLn$, since $\max(\psi_i^n,\psi_j^n)\le 1$, the above
inequality gives
\begin{align*}
 \sfU_i\upHnp & \le \sfU_i^{\umax,n} + (\sfU_i^n -\sfU_i^{\umax,n}) (1-\gamma_i^n) 
  + \frac{\dt}{m_i}\sum_{j\in \calI(i^-)} (d_{ij}\upLn  -
  d_{ij}\upHn)(\sfU^n_i -\sfU^n_j) \\
&\le \sfU_i^{\umax,n} + (\sfU_i^n -\sfU_i^{\umax,n}) (1-\gamma_i^n) 
  + \frac{\dt}{m_i}\sum_{j\in \calI(i^-)} (d_{ij}\upLn  -
  d_{ij}\upHn)(\sfU^n_i -\sfU^{\umin,n}_i).
\end{align*}
Now using that $\sfU_i^{\umax,n}-\sfU^{\umin,n}_i\ne 0$ and that 
$\sfU^n_i $ is in the convex hull
of $\sfU_i^{\umax,n}$ and $\sfU^{\umin,n}_i$, we have
$\sfU^n_i = \theta_i^n \sfU_i^{\umax,n} + (1-\theta_i^n) \sfU^{\umin,n}_i$
where $\theta_i^n\in[0,1]$ has been defined in
\eqref{def_f_theta_in}. Hence,
$\sfU^n_i - \sfU^{\umin,n}_i= -\theta_i^n(\sfU^{\umin,n}_i-\sfU_i^{\umax,n})$ and
$\sfU^n_i-\sfU^{\umax,n}_i =(1-\theta_i^n)(\sfU^{\umin,n}_i-\sfU_i^{\umax,n})$.
With these definitions, the above inequality is rewritten as follows:
\begin{align*}
 \sfU_i\upHnp 
&\le \sfU_i^{\umax,n} + (\sfU_i^{\umin,n} -\sfU_i^{\umax,n}) \bigg((1-\theta_i^n)(1-\gamma_i^n) 
  - \theta_i^n\frac{\dt}{m_i}\sum_{j\in \calI(i^-)} (d_{ij}\upLn  -
  d_{ij}\upHn)\bigg).
\end{align*}
\textup{(iii)}
Using that $d_{ij}\upHn\ge d_{ij}\upLn \psi_i^n$ and $\psi_i^n\ge 0$,
we infer that $-d_{ij}\upHn\le -d_{ij}\upLn \psi_i^n$, which in turn
implies the following inequalities:
\begin{align*}
 \sfU_i\upHnp
&\le \sfU_i^{\umax,n} + (\sfU_i^{\umin,n} -\sfU_i^{\umax,n}) \bigg((1-\theta_i^n)(1-\gamma_i^n) 
  - \theta_i^n(1- \psi_i^n)\frac{\dt}{m_i}\sum_{j\in \calI(i^-)}
  d_{ij}\upLn  \bigg)\\
& \le \sfU_i^{\umax,n} + (\sfU_i^{\umin,n} -\sfU_i^{\umax,n}) \left((1-\theta_i^n)(1-\gamma_i^n) 
  - \theta_i^n(1-\psi_i^n)\tfrac12 \gamma_i^{-,n}\right).
\end{align*}
\textup{(iv)}
The other estimate is obtained similarly. More precisely, 
using that  $\sfU_{ij}^{*,n} \ge \sfU_i^{\umin,n}$, we infer that
\begin{align*}
\sfU_i\upHnp& \ge \sfU_i^{\umin,n} + (\sfU_i^{\umax,n} -\sfU_i^{\umin,n}) (1-\gamma_i^n) 
  + \frac{\dt}{m_i}\sum_{j\in \calI(i^+){\setminus}\{i\}} (d_{ij}\upHn  -
  d_{ij}\upLn)(\sfU^{\umax,n}_i -\sfU^n_i) \\
 & \ge \sfU_i^{\umin,n} + (\sfU_i^{\umax,n}-\sfU_i^{\umin,n}) \left( \theta_i^n(1-\gamma_i^n) 
  - (1-\psi_i^n)(1-\theta_i^n) \tfrac12 \gamma_i^{+,n} \right),
\end{align*}
which completes the proof.
\end{proof}

We now formulate the main result of this section.

\begin{theorem}\label{Thm:max_principle_smoothness} 
  Let $\psi\in \textup{Lip}([0,1];[0,1])$ be 
  such that $\psi(1)=1$ and with Lipschitz constant $k_\psi$.
Consider the scheme \eqref{abstract_high_order} using either
  the high-order cG flux~\eqref{defFijCGlumped} or the high-order dG
  flux~\eqref{defFijdGlumped} with the graph viscosity defined in
  \eqref{dij_smoothness_indicator}.  Assume that $g(\bsfU)=\sfU$ in
  \eqref{def_of_alpha_in}. Assume that all the coefficients
  $\beta_{ij}$ in \eqref{def_of_alpha_in} are positive and there
  exists $\varpi^\sharp\in(0,\infty)$ uniform with respect to the mesh
  sequence $\famTh$, such that
  $\max_{i\in\calV}(\max_{j\in\calI(i)}
  \beta_{ij}/\min_{j\in\calI(i)} \beta_{ij})\le \varpi^\sharp$. Let $i\in\calV$ and $n\ge 0$.
  Then, under the local CFL condition
  $\gamma_i^n \le \frac{1}{1+ k_\psi c_\sharp}$, where
  $c_\sharp = \varpi^\sharp \max_{i\in\calV}\textup{card}(\calI(i))$
  (this number is uniformly bounded with respect to the mesh
  sequence), the scheme is locally invariant domain preserving for the scalar component $\sfU$: \ie
  $\sfU_i\upHnp\in [\sfU^{\umin,n}_i,\sfU_i^{\umax,n}]$.
\end{theorem}

\begin{proof}
Note first that if $\sfU_i^{\umax,n}=\sfU^{\umin,n}_i$, then
$\sfU_i\upHnp = \sfU_i^{n}\in [\sfU^{\umin,n}_i,\sfU_i^{\umax,n}]$
irrespective of the value of $d_{ij}\upHn$, which proves the
statement. Let us assume now that $\sfU_i^{\umax,n}\ne \sfU^{\umin,n}_i$. If
$\theta_i^n=\frac{\sfU^n_i-\sfU^{\umin,n}_i}{\sfU_i^{\umax,n}-\sfU^{\umin,n}_i}
\in\{0,1\}$, then
either $\sfU^n_i=\sfU^{\umin,n}_i$ or $\sfU^n_i=\sfU^{\umax,n}_i$. In this
case, $\alpha_i^n=1$ and $\psi(\alpha_i^n) =1$; as a result,
$d_{ij}\upHn =d_{ij}\upLn\max(1,\psi(\alpha_j)) = d_{ij}\upLn$ for
all $j\in \calI(i)$, which implies that
$\sfU_i\upHnp = \sfU_i\upLnp\in [\sfU^{\umin,n}_i,\sfU_i^{\umax,n}]$. Finally,
let us assume that $0<\theta_i^n<1$.  Observing that
$||y|-|x||= \max(-|x|+|y|,|x|-|y|)$, we infer that
$-||y|-|x||\le |y|-|x|$ for all $x,y\in\Real$. This inequality in
turn implies that
\begin{align*}
1-\alpha_i^n 
&= 1- \frac{\left|\sum_{j\in\calI(i^+)}
  \beta_{ij}|\sfU_j^n-\sfU_i^n| - \!\sum_{j\in\calI(i^-)}
  \beta_{ij}|\sfU_j^n-\sfU_i^n|\right|}{\sum_{j\in\calI(i)} \beta_{ij}|\sfU_j^n-\sfU_i^n|}
  \\
&\le \frac{\sum_{j\in\calI(i)} \beta_{ij}|\sfU_j^n-\sfU_i^n| + \sum_{j\in\calI(i^+)}
  \beta_{ij}|\sfU_j^n-\sfU_i^n| - \sum_{j\in\calI(i^-)}
  \beta_{ij}|\sfU_j^n-\sfU_i^n|}{\sum_{j\in\calI(i)}\beta_{ij}|\sfU_j^n-\sfU_i^n|}\\
&\le 2 \frac{\sum_{j\in\calI(i^+)}
\beta_{ij} (\sfU_j^n-\sfU_i^n)}{\sum_{j\in\calI(i)} \beta_{ij}|\sfU_j^n-\sfU_i^n|} 
\le 2
  \frac{\sum_{j\in\calI(i^+)}\beta_{ij}(\sfU_j^{\umax,n}-\sfU_i^n)}{\min_{j\in\calI(i)} \beta_{ij}
(|\sfU_i^{\umax,n}-\sfU_i^n|+|\sfU_i^{\umin,n}-\sfU_i^n|)} \\
&\le 2 \frac{
  \sfU_i^{\umax,n}-\sfU_i^n}{\sfU_i^{\umax,n}-\sfU_i^{\umin,n}}
\frac{\max_{j\in\calI(i)} \beta_{ij}}{\min_{j\in\calI(i)} \beta_{ij}}
\text{card}(\calI(i^+)) \le 2 c_\sharp (1-\theta_i^n),
\end{align*}
where $c_\sharp =
\varpi^\sharp\max_{i\in\calV}\text{card}(\calI(i)) $ is a number
uniformly bounded with respect to the mesh sequence.  Likewise
we have
\[
1-\alpha_i^n \le 2c_\sharp \theta_i^n.
\]
Let $k_\psi$ be the Lipschitz constant of $\psi$. Then $1-\psi(\alpha_i^n)=
\psi(1)-\psi(\alpha_i^n) \le k_\psi (1-\alpha_i^n)$. This in turn implies
that
\begin{align*}
(1-\theta_i^n)(1-\gamma_i^n) 
  - \theta_i^n(1- \psi(\alpha_i^n))\tfrac12 \gamma_i^{-,n} &\ge 
(1-\theta_i^n)(1-\gamma_i^n) 
  -  k_\psi c_\sharp \theta_i^n(1- \theta_i^n)\gamma_i^{n} \\
&\ge (1-\theta_i^n)(1-(1+  k_\psi c_\sharp  \theta_i^n)\gamma_i^{n})\ge 0,
\end{align*}
provided $\gamma_i^{n} \le \frac{1}{1+ k_\psi c_\sharp}$.
Similarly, provided again that $\gamma_i^{n} \le \frac{1}{1+ k_\psi c_\sharp}$, 
we have
\begin{align*}
\theta_i^n(1-\gamma_i^n) 
  - (1-\theta_i^n)(1- \psi(\alpha_i^n))\tfrac12 \gamma_i^{+,n} &\ge 
\theta_i^n(1-\gamma_i^n) 
  -  k_\psi c_\sharp \theta_i^n(1- \theta_i^n)\gamma_i^{n} \\
&\ge \theta_i^n(1-(1+  k_\psi c_\sharp  (1-\theta_i^n))\gamma_i^{n})\ge 0,
\end{align*}
The conclusion follows from
Lemma~\ref{Lem:lower_upper_bounds}.
\end{proof}

\begin{example}[Shallow water/Euler equations] \label{Ex:max_principle_smoothness} 
  The above technique can be used to solve the Saint-Venant equations.
  In this case one can use the water height as smoothness indicator.
  This technique can also be used to solve the compressible Euler
  equations.  In this case one can use the density as smoothness
  indicator. Let us denote by $\sfU$ the scalar component that is
  chosen for the smoothness indicator. Then the scheme
  \eqref{abstract_high_order} using the high-order
  flux~\eqref{defFijCGlumped} or \eqref{defFijdGlumped} with the graph
  viscosity defined in \eqref{dij_smoothness_indicator} with
  $g(\bsfU)=\sfU$ satisfies the local maximum/minimum principle
  $\sfU_i\upHnp\in [\sfU^{\umin,n}_i,\sfU_i^{\umax,n}]$ for all
  $i\in\calV$ under the appropriate CFL condition. This means in
  particular that the water height (or the density) stays positive.
\end{example}

\begin{remark}[Literature]
  The origins of the smoothness-based viscosity can be found in
  \eg\cite[Eq. (12)]{Jameson_aiaa_1981}, see also the second formula in
  the right column of page~1490 in \cite{Jameson_aiaa_2015}.  A
  version of Theorem~\ref{Thm:max_principle_smoothness} for scalar
  conservation equations is proved
  in~\cite{guermond_popov_second_order_2018}. To the best of our
  knowledge, it seems that Theorem~\ref{Thm:max_principle_smoothness}
  as stated here for hyperbolic systems and generic discretizations is
  original.  The technique presented here shows similarities with that
  proposed in \citet[Thm.~4.1]{Burman_2007} and
  \citet[Eq. (2.4)-(2.5)]{Barrenechea_2016}. The quantity
  $(\alpha_i^n)^p$, $p\ge 2$, is used in \citep{Burman_2007} to construct
  a nonlinear viscosity that yields the maximum principle and
  convergence to the entropy solution for Burgers' equation in one
  dimension. It is used in \citep{Barrenechea_2016} for solving linear
  scalar advection--diffusion equations.
\end{remark}

\subsection{Greedy graph viscosity} \label{Sec:Greedy viscosity} We continue
with a technique entirely based on the observations made in
Lemma~\ref{Lem:lower_upper_bounds}, irrespective of any smoothness
considerations. 
As in \S\ref{Sec:bounds_smoothness_based_viscosity}, 
we specialize the setting by assuming that there is one scalar component of $\bsfU$,
say $\sfU$, for which the source term is zero, \ie  $S\equiv 0$.

Let $i\in\calV$ and $n\ge 0$. Let $\theta_{n}^n$,
$\gamma_i^{-,n}$, and $\gamma_i^{+,n}$ be the quantities defined in
\eqref{def_of_gammam_gammap}-\eqref{def_f_theta_in} for all
$i\in\calV$. 
We recall that Lemma~\ref{Lem:lower_upper_bounds} is quite general 
and just requires that $S(\bsfU)\equiv 0$ and $\psi_i^n,\psi_j^n\in [0,1]$. 
Let us set 
\begin{equation}
  \psi_i^n := \max\bigg(1-2(1-\gamma_i^n)\min\bigg(\frac{1}{\gamma_i^{-,n}}\frac{1-\theta_i^n}{\theta_i^n},
      \frac{1}{\gamma_i^{+,n}}\frac{\theta_i^n}{(1-\theta_i^n)}\bigg),0\bigg), \label{def_of_phin_greedy}
\end{equation}
if $\theta_i^n\not\in \{0,1\}$ and $\psi_i^n = 1$ otherwise. Then we set 
\begin{equation}
d_{ij}\upHn := d_{ij}^n\max(\psi_i^n,\psi_j^n), \qquad \forall i\in \calV, \ \forall j\in\calI(i){\setminus}\{i\}.
\label{dij_greedy}.
\end{equation}
We now formulate the main result of this section.

\begin{theorem}[Greedy graph viscosity]  \label{Thm:Greedy_viscosity}
  Consider the scheme \eqref{abstract_high_order} using either the
  high-order cG flux~\eqref{defFijCGlumped} or the high-order dG
  flux~\eqref{defFijdGlumped} with the graph viscosity defined in
  \eqref{dij_greedy} using the definitions
  \eqref{def_of_gammam_gammap}-\eqref{def_f_theta_in} with
  $\sfU^{\umin,n}_i$, $\sfU_i^{\umax,n}$ defined in
  \eqref{def_of_min_max_of_bar_states}.  Assume that
  $\gamma_i^n\le 1$, then the scheme is locally invariant domain
  preserving for the scalar component $\sfU$: \ie
  $\sfU_i\upHnp\in [\sfU^{\umin,n}_i,\sfU_i^{\umax,n}]$.
\end{theorem}
\begin{proof}
  Note first that if $\sfU_i^{\umax,n}=\sfU^{\umin,n}_i$, then
  $\sfU_i^{n+1} = \sfU_i^{n}\in [\sfU^{\umin,n}_i,\sfU_i^{\umax,n}]$
  irrespective of the value of $d_{ij}^n$, which proves the statement.
  If $\theta_i^n\in\{0,1\}$, then $\psi_i^n =1$ implies that
  $d_{ij}^n =d_{ij}^{\textup{L},n}\max(1,\psi_j^n) = d_{ij}^{\textup{L},n}$ for all
  $j\in\calI(i){\setminus}\{i\}$, which again implies that
  $\sfU_i^{n+1} = \sfU_i^{n}\in [\sfU^{\umin,n}_i,\sfU_i^{\umax,n}]$. Finally,
  let us assume that $0<\theta_i^n<1$.  The definition of $\psi_i^n$
  in \eqref{def_of_phin_greedy} implies that
  $\psi_i^n \ge
  1-2\frac{1-\gamma_i^n}{\gamma_i^{-,n}}\frac{1-\theta_i^n}{\theta_i^n}$,
  which in turn gives
  $\theta_i^n(\psi_i^n-1) \frac12 \gamma_i^{-,n} +
  (1-\gamma_i^n)(1-\theta_i^n)\ge 0$.  This is the condition in
  Lemma~\ref{Lem:lower_upper_bounds} that shows that
  $\sfU_i^{n+1}\le \sfU_i^{\umax,n}$, see
  \eqref{Eq:Lem:lower_upper_bounds:up}. Similarly, we have
  $\psi_i^n \ge
  1-2\frac{1-\gamma_i^n}{\gamma_i^{+,n}}\frac{\theta_i^n}{1-\theta_i^n}$,
  which gives
  $(\psi_i^n-1)(1-\theta_i^n)\frac12\gamma_i^{+,n} +
  (1-\gamma_i^n)\theta_i^n \ge 0$.  This is the condition in
  Lemma~\ref{Lem:lower_upper_bounds} that shows that
  $\sfU_i^{\umin,n} \le \sfU_i^{n+1}$, see
  \eqref{Eq:Lem:lower_upper_bounds:down}.
\end{proof}
\begin{remark}[Small CFL number] 
  Note in \eqref{def_of_phin_greedy} that the quantity $\psi_i^n$ is
  almost equal to $1$ when $\sfU_i^n$ is not a local extremum and the
  local CFL number $\gamma_i^n$ is small. This shows that the method
  becomes greedier as the CFL number decreases; thereby the name of
  the method.
\end{remark}

\begin{remark}[Min-Max] 
The greedy graph viscosity based on~\eqref{def_of_phin_greedy} 
explicitly involves the bounds $\sfU_{i}^{\umin,n}$ and $\sfU_{i}^{\umax,n}$, 
whereas the smoothness-based graph viscosity using 
\eqref{def_of_alpha_in} does not.
\end{remark}

\subsection{Commutator-based graph viscosity}\label{sec:CommBased}
The objective of this section is to construct the high-order graph
viscosity so that the method is entropy consistent and close to be
invariant domain preserving. In other words, we do not want to rely on
the (yet to be explained) limiting process to enforce entropy
consistency. For instance one naive choice consists of using
$d_{ij}\upHn=0$, which gives the maximum accuracy for smooth
solutions, but as shown in Lemma~4.6 in
\cite{guermond_popov_second_order_2018} one can construct simple
counterexamples with Burgers' equation such that the resulting method
is maximum principle preserving, after limiting, but does not converge
to the entropy solution. A better option consists of estimating an
entropy residual/commutator as suggested in
\citep[\S5.1]{guermond_popov_second_order_2018},
\citep[\S3.4]{GuerNazPopTom2017},
\citep[\S6.1]{Guermond_Quezada_Popov_Kees_Farthing_2018}.

The key idea consists of measuring the smoothness of an entropy
  by measuring how well the chain rule is satisfied by the
  discretization at hand.  Given an entropy pair
  $(\eta(\bv),\bF(\bv))$ for \eqref{def:hyperbolic_system} we set
  $\eta_i^{\max,n}:=\max_{j\in\mathcal{I}(i)}\eta(\bsfU_j^n)$,
  $\eta_i^{\min,n}:=\min_{j\in\mathcal{I}(i)}\eta(\bsfU_j^n)$,
  $\epsilon_i = \epsilon\max_{j\in\mathcal{I}(i)}|\eta(\bsfU_j^n)|$
  and
  $\Delta\eta_i^n = \max(\frac12(\eta_i^{\max,n}-\eta_i^{\min,n}),
  \epsilon_i)$,
  then the so-called entropy viscosity, or commutator-based graph
  viscosity, is defined by setting
\begin{align}
N_i^n &:= \sum_{j\in\calI(i)} ( \bF(\bsfU_j^n) - (\eta'(\bsfU_i^n))\tr
\polf(\bsfU_j^n))\SCAL\bc_{ij}, \\
d_{ij}\upHn &:= \min(d_{ij}\upLn, \max(\frac{|N_i^n|}{\Delta\eta_i^n},\frac{|N_j^n|}{\Delta\eta_j^n})).
\label{EV}
\end{align}
The normalization in \eqref{EV} and the choice of entropy are not
unique; we refer the reader to \citep{GuerNazPopTom2017} where
relative entropies are used.

\section{Convex Limiting} \label{Sec:convex_limiting}
In this section we develop a general limiting framework to preserve
convex invariant sets and (more generally) quasiconcave constraints. This
work is aligned with the ideas presented in
\cite{Khobalatte_Perthame_1994}, \cite{Perthame_Youchun_1994,
  Perthame_Shu_1996} in the context of finite volume methods. 
We also refer the
reader to \cite{Zhang_Shu_2010,Zhang_Shu_2012,Jiang_Liu_2017} for
recent/related developments in the context of dG methods. The ideas
presented in this section are slightly more general as they naturally
extend beyond the Finite Volume/dG methods. The approach that we
propose is related to flux-limiting techniques like the
flux-corrected transport method by
\cite{Boris_books_JCP_1973,Zalesak_1979}.

\subsection{Quasiconcavity}
We have seen in \S\ref{sec:firstorder} that the low-order solution
$\bsfU_i\upLnp$ satisfies some ``convex bounds'' and, in
principle, we would like the high-order solution to satisfy these
``convex bounds'' as well. But, before proceeding any further, we need
to define clearly what we mean by convex bounds. We also need to
give a precise statement about the bounds that are naturally satisfied by
the first-order method. These are the two objectives of the present section and the next one \S\ref{sec:bounds}.

 In general, the convex bounds mentioned above can be described in terms of
upper  contour sets of quasiconcave functions and lower contour sets of
quasiconvex functions. For the sake of completeness we recall the
definitions of quasiconcavity and quasiconvexity.

\begin{definition}[Quasiconcavity]\label{def:quasiconv} Given a convex set
  $\calB \subset \Real^m$, we say that a function
  $\Psi:\calB \to \Real$ is quasiconcave if the set
  $L_\chi(\Psi) := \{\bsfU\in \calB \st \Psi(\bsfU) \ge \chi \}$
  is convex for any $\chi\in \Real$. The sets $L_\chi(\Psi)$ are called upper contour
  sets. 
\end{definition} 
We are going to make use of the following equivalent definition.
\begin{lemma}[Quasiconcavity] \label{lem:quasiconv_bis} Let
 $\calB \subset \Real^m$ be convex set. A function
  $\Psi:\calB \to \Real$ is quasiconcave iff for every finite set
  $\calS\subset\polN$, every corresponding set of convex coefficients
  $\{\lambda_j\}_{j\in\calS}$ (\ie $\sum_{j\in\calS} \lambda_j = 1$
  and $\lambda_j \geq 0$ for all $j \in \calS$), and every corresponding
   collection of vectors $\{\bsfU_j\}_{j\in\calS}$ in $\calB$, the following holds true:
\begin{align}\label{quasiconcineq}
\Psi\Big(\sum_{j \in \calS} \lambda_j \bsfU_j\Big) \geq \min_{j \in \calS} \Psi(\bsfU_j). 
\end{align}
\end{lemma} 

\begin{definition}[Quasiconvexity]\label{def:quasiconv_ter}
  A function $\Psi:\calB \to \Real$ is quasiconvex if $-\Psi$ is
  quasiconcave.
\end{definition} 

Note that Jensen's inequality implies that
concave/convex functions are quasiconcave/quasiconvex
(respectively). The reader is referred to \cite{Mordecai1988} for
further properties of quasiconcave/convex functions.  
We now give a result that is useful to prove that a function is quasiconcave.

\begin{lemma} \label{Lem:quasiconcavity} Let $\calB\subset \Real^m$ be
  a convex set.  Let $R:\calB\to (0,\Real)$ be a positive function.
  Let $\Psi:\calB\to \Real$ and assume that the product $R \Psi$ is
  concave. Then $\Psi$ is quasiconcave if one of the following two
  assumptions is satisfied:
(i) $R$ is affine or (ii) $R$ is convex and $\Psi$ is nonnegative.
\end{lemma}
\begin{proof}
  Let $\{\lambda_j\}_{j\in \calS}$ be a set of convex coefficients.
  Let $\{\bsfU_j^n\}_{j\in \calS}$ be members of $\calB$.  Let us set
  $\chi:=\min_{j\in\calS}\Psi(\bsfU_j)$.  Let
  $\Phi(\bsfU):= R(\bsfU) (\Psi(\bsfU) - \chi)$.  Notice that if $R$
  is affine, or if $R$ is convex and $\Psi$ is nonnegative,
  then $-\chi R(\bsfU)$ is concave.  As a result, $\Phi$ is concave
  since $R(\bsfU) \Psi(\bsfU)$ and $-\chi R(\bsfU)$ are both concave
  and the sum of two concave functions is concave (this may not be the
  case for the sum of quasiconcave functions).  Notice also that
  $\min_{j\in\calS} \Phi(\bsfU_j)\ge 0$ because $R\ge 0$ and
  $\min_{j\in\calS} \Psi(\bsfU_j)-\chi\ge 0$.  Hence
\begin{align*}
\Phi\Big(\sum_{j\in\calS} \lambda_j \bsfU_j\Big) = R\Big(\sum_{j\in\calS} \lambda_j \bsfU_j\Big)
\bigg(\Psi\Big(\sum_{j\in\calS} \lambda_j \bsfU_j\Big)-\chi\bigg)
&\ge 
\sum_{j\in\calS}\lambda_j  \Phi(\bsfU_j) \ge 0.
\end{align*}
This in turn implies that
$\Psi(\sum_{j\in\calS} \lambda_j \bsfU_j)\ge
\chi=\min_{j\in\calS}\Psi(\bsfU_j)$, which proves the assertion owing to Lemma~\ref{lem:quasiconv_bis}.
\end{proof}

\begin{example}[Entropy]\label{Example:density}
  Let $\eta:\calB\to \Real$ be any entropy for
  \eqref{def:hyperbolic_system} (recall that entropies are convex by
  definition), then $\Psi(\bsfU) = -\eta(\bsfU)$ is quasiconcave.
\end{example}

\begin{example}[Specific Entropy] \label{Example:specific_entropy}
  Let $\eta:\calB\to \Real$ be any entropy for
  \eqref{def:hyperbolic_system}. Let $R:\calB \to (0,\infty)$ be a
  positive linear function, then Lemma~\ref{Lem:quasiconcavity}
  implies that $\Psi(\bsfU) = -\eta(\bsfU)/R(\bsfU)$ is quasiconcave. One
  can think of this function as a specific entropy in the case of the
  shallow water equations ($R(\bsfU)$ is the water height), or the
  case of the Euler equations ($R(\bsfU)$ is the density),
\end{example}

Let us now give examples of quasiconcave functionals in the context of
the compressible Euler equations with an arbitrary equation of state.
The conserved variables in this case are $\bsfU:=(\rho,\bbm,E)\tr$.

\begin{example}[Density]
  We set $\calB := \mathbb{R}^{d+2}$, $\Psi(\bsfU) := \rho$.  The
  functional $\Psi:\calB\to \Real$ is linear, hence it is
  quasiconcave.  Note the following functional $\Psi(\bsfU) = -\rho$
  is also quasiconcave.
\end{example}

\begin{example}[Total energy]
  We set $\calB := \mathbb{R}^{d+2}$, $\Psi(\bsfU) := E$.  The
  functional $\Psi:\calB\to \Real$ is linear, hence it is
  quasiconcave.  Note the following functional $\Psi(\bsfU) = -E$ is
  also quasiconcave.
\end{example}
 
\begin{example}[Internal energy] \label{Ex:internal_energy} We set
  $\calB:=\{\bsfU =(\rho,\bbm,E)^\trans\in\Real^m \st \rho>0\}$ and introduce
  the internal energy
  $\varepsilon(\bsfU) := E - \frac{\bbm^2}{2\rho}$.  A direct
  computation shows that the functional $\varepsilon:\calB\to \Real$
  has a negative semi-definite Hessian for every equation of state,
  thereby proving that $\varepsilon$ is concave, hence quasiconcave.
\end{example}

Let us now illustrate the use of Lemma~\ref{Lem:quasiconcavity} with $R(\bsfU)=\rho$.

\begin{example}[Specific internal energy] Let
  $\calB:=\{\bsfU =(\rho,\bbm,E)^\trans\in \Real^m \st \rho>0\}$, and introduce
  the specific internal energy
  $e(\bsfU) := \frac{\varepsilon(\bsfU)}{\rho}= \frac{E}{\rho} -
  \frac{\bbm^2}{2\rho^2}$.
  Clearly $R(\bsfU):=\rho$ is convex; moreover,
  $\Phi(\bsfU):=R(\bsfU) e(\bsfU) =E - \frac{\bbm^2}{2\rho} =
  \varepsilon(\bsfU)$
  is the internal energy, which we know is a concave function for any
  equation of state. Hence we conclude from
  Lemma~\ref{Lem:quasiconcavity} that the specific internal energy is
  quasiconcave for any equation of state. Notice in passing that this
  argument proves that the set
  $\{\bsfU:=(\rho,\bbm,E)^\trans \st \rho\ge \rho_0 , e(\bsfU)\ge
  e_0\}$ is convex for any $\rho_0, e_0\in (0,\infty)$.
\end{example}

\begin{example}[Generalized specific entropies] We set
  $\calB:=\{\bsfU\in \Real^m \st \rho>0, e(\bsfU)>0\}$. Let
  $\eta:\calB \to \Real$ be a generalized entropy as defined in
  \cite[Eq.~(2.10a)]{Harten_1983},
  \cite[Thm.~2.1]{Harten_Lax_Levermore_Morojoff_1998}. Then using
  Lemma~\ref{Lem:quasiconcavity} with $R(\bsfU)=\rho$ and
  $\Psi(\bsfU)= \eta(\bsfU)/R(\bsfU)$, we conclude that the specific
  entropy $s(\bsfU):=\rho^{-1}\eta(\bsfU)$ is quasiconcave.  Note in
  passing that we have proved that the set
  $\{\bsfU:=(\rho,\bbm,E)^\trans \st \rho>\rho_0, e(\bsfU)>\rho_0,
  s(\bsfU)>s_0\}$
  is convex for any $\rho_0,e_0>0$ and any $s_0\in \Real$. We
    refer the reader to Theorem~8.2.2 from \cite{Serre_2000_II} for
    other properties of this set.
\end{example}

\begin{example}[Kinetic energy] \label{Example:kinetic_energy}
  We set $\calB:=\{\bsfU=(\rho,\bbm,E)^\trans\in \Real^m \st \rho>0\}$. Let
  $\Psi(\bsfU)=-\frac12 \rho^{-1} \bbm^2$ be the (negative) kinetic
  energy. It is clear that $\Phi(\bsfU)=-\frac12 \bbm^2$ is concave,
  then using Lemma~\ref{Lem:quasiconcavity} with $R(\bsfU)=\rho$, we
  conclude that the (negative) kinetic energy is quasiconcave.
\end{example}

We finish with a result that is useful to transform quasiconcave
functionals.
\begin{lemma}
\label{Lem:increasing_quasiconvave}
Let $\Psi: \calB \to \Real$ be a quasiconcave function. Let
$L:\Real\to \Real$ be a nondecreasing function, then $L\circ\Psi$ is
quasiconcave.
\end{lemma}
\begin{proof}
  Let us use the characterization \eqref{quasiconcineq}. Since $L$ is
  nondecreasing, we have
  $L\circ\Psi(\sum_{j\in\calS} \lambda_j\bsfU_j) \ge L
  (\min_{j\in\calS}\Psi(\bsfU_j))$. Let $k\in\calS$ be such that
  $\Psi(\bsfU_k) := \min_{j\in\calS} \Psi(\bsfU_j)$. Then, for any
  $j\in\calS$, we have $\Psi(\bsfU_k) \le \Psi(\bsfU_j)$, which
  implies that $L\circ \Psi(\bsfU_k) \le L\circ\Psi(\bsfU_j)$.  Hence
  $L (\min_{j\in\calS}\Psi(\bsfU_j)) = L(\Psi(\bsfU_k)) =
  \min_{j\in\calS} L(\Psi(\bsfU_j))$. In conclusion
  $L\circ\psi(\sum_{j\in\calS} \lambda_j\bsfU_j) \ge \min_{j\in \calS}
  L\circ\Psi(\bsfU_j)$, which proves the assertion.
\end{proof}

\begin{example}[Specific
  entropy] \label{Ex:change_of_variable_specific_entropy} Let us
  illustrate the use of Lemma~\ref{Lem:increasing_quasiconvave} with
  the compressible Euler equations, and, to simplify the argument, let 
  assume that the equation of state is the $\gamma$-law.  Consider the
  physical specific entropy
  $\Psi(\bsfU) =
  \frac{1}{\gamma-1}\log(\varepsilon(\bsfU)\rho^{-\gamma})$, where
  $\varepsilon(\bsfU)$ is the internal energy. This function is
  quasiconcave owing to Lemma~\ref{Lem:quasiconcavity} with
  $R(\bsfU)=\rho$, since $\rho \Psi(\bsfU)$ is known to be
  concave. Then using Lemma~\ref{Lem:increasing_quasiconvave} we
  conclude that $\tilde\Psi(\bsfU) = \varepsilon(\bsfU)\rho^{-\gamma}$
  is quasiconcave.
\end{example}

\subsection{Bounds} \label{sec:bounds} In this section we define the
bounds that we are going to use to limit the high-order solution.
The following result will play a key role in the rest of the paper,
since it tells us precisely what are the ``convex bounds'' that the
low-order solution produced by the GMS-GV scheme satisfies.

\begin{lemma}[Natural bounds on the GMS-GV scheme] 
  \label{Lem:NaturalBounds} Let $\calB\subset \calA \subset \Real^m$
  be a convex set and $\Psi: \calB \to \Real$ be a quasiconcave
  functional.  Let $n\ge 0$, $i\in\calV$, and assume that
  $1+4\dt \frac{d_{ii}\upLn}{m_i}\ge 0$ and $2\dt \le \tau_0$. Assume that 
$\bsfU_j^n\in \calB$ for all $j\in \calI(i)$. Let
  $\{\overline{\bsfU}_{ij}^{n}\}_{j\in \calI(i)}$ be the auxiliary
  states defined in \eqref{def_barstates}.  Consider the following
  quantity:
\begin{align}\label{def:convexBounds}
\Psi_i^{\min} :=\min(\Psi(\bsfU_i^n + 2\dt \bS(\bsfU_i^n)), 
\min_{j \in \calI(i)} \Psi(\overline{\bsfU}_{ij}^{n})).
\end{align}
Then, the first-order update $\bsfU_i\upLnp$ computed with the GMS-GV scheme 
(see \eqref{def_dij_scheme} plus \eqref{Def_of_dij}) is in $\calB$ and
satisfies the following inequality:
\begin{align}
\Psi(\bsfU_i\upLnp) \geq \Psi_i^{\min}.
\end{align}
\end{lemma}

\begin{proof} Using the assumptions,
  $1+4\dt \frac{d_{ii}\upLn}{m_i}\ge 0$ and $2\dt \le \tau_0$, we
  first observe that \eqref{def_dij_scheme_convex_with_source} shows
  that $\bsfU_i\upLnp$ is a convex combination of the states
  $\bsfU_i^n + 2 \dt \bS(\bsfU_i^n)$ and
  $\{\overline\bsfU_{ij}^n\}_{j\in\calI(i){\setminus\{i\}}}$ which are
  all in $\calB$; hence $\bsfU_i\upLnp$ is in $\calB$. Then the
  conclusion follows readily by using the quasiconcavity property
  \eqref{quasiconcineq}.
\end{proof}

\begin{remark}[Quasiconcavity vs. quasiconvexity] 
  Since any quasiconvex function can be transformed into a
  quasiconvave function by a sign change, the above lemma gives
  $\Psi(\bsfU_i\upLnp) \leq \Psi_i^{\max}:= \max(\Psi(\bsfU_i^n +
  2\dt \bS(\bsfU_i^n)), \max_{j \in \calI(i)}
  \Psi(\overline{\bsfU}_{ij}^{n}))$ for any quasiconvex function
  $\Psi: \calB \subset \calA \to \Real$.  Therefore, in order to alleviate
  the language, we will henceforth refrain from mentioning quasiconvexity and
  will formulate every ``convex bounds'' in terms quasiconcave
  functionals only. 
\end{remark}

\begin{remark}[Invariant set vs. local bound]
  Notice that
  Lemma~\ref{Lem:NaturalBounds} contains two statements that are of
  different nature.  The first one is an invariant domain property:
  $(\bsfU_j^n\in \calB,\ \forall j\in \calI(i))\Rightarrow
  (\bsfU_i\upLnp\in\calB)$.
  Since $\calB$ does not depend on
  $i\in \calV$, this local assertion can be reformulated into a global
  statement
  $(\bsfU_i^n\in \calB,\ \forall i\in \calV)\Rightarrow
  (\bsfU_i\upLnp\in\calB, \ \forall i\in \calV)$.
  The second statement $\Psi(\bsfU_i\upLnp) \geq \Psi_i^{\min}$ is a
  local bound that can be viewed as a local ``generalized minimum
  principle.'' This bound cannot be made uniform; it is local in time
  and space, since $\Psi_i^{\min}$ depends on $i$ and $n$.
\end{remark}

\begin{remark}[Relaxation]
  The reader must be aware that in general the bound $\Psi_i^{\min}$
  defined in~\eqref{def:convexBounds} must be slightly relaxed in order to go
  beyond second-order accuracy in space in the $L^1$-norm. We refer the reader
  to \S\ref{Sec:relaxing_the_bounds} for implementation details on
  relaxation techniques.
\end{remark}

\subsection{Abstract Framework}\label{sec:ConvLimAbstract} 
In the sections \S\ref{sec:algefluxesFV}, \S\ref{sec:algefluxesCG} and
\S\ref{sec:algefluxesDG} we have seen that most high-order methods can
be written in the algebraic form
\begin{align}
\label{absconsH}
\frac{m_i}{\dt}(\bsfU_i\upHnp-\bsfU_i^n)
 + \sum_{j\in \calI(i)} \bsfF_{ij}\upHn = m_i\bS(\bsfU_i^n),
\end{align}
with $\bsfF_{ij}\upHn \in \Real^m$ satisfying the skew-symmetry
constraint $\bsfF_{ij}\upHn = -\bsfF_{ij}\upHn$ for all
$j\in \calI(i)$ (whether we use the consistent mass matrix for the
discretization for the time derivative or not), where the superscript
$^\high$ denotes high-order. Subtracting \eqref{abscons} from
\eqref{absconsH} and reorganizing we get
$m_i \bsfU_i\upHnp = m_i \bsfU_i\upLnp + \sum_{j \in
  \calI(i)\backslash\{i\}} \dt (\bsfF_{ij}\upLn - \bsfF_{ij}\upHn)$.
This expression can be rewritten into the following important identity:
\begin{align}\label{PreFCT}
m_i \bsfU_i\upHnp = m_i \bsfU_i\upLnp 
+ \sum_{j \in \calI(i)\backslash\{i\}} \bsfA_{ij}^n,
\end{align}
where
$\bsfA_{ij}^n := \dt (\bsfF_{ij}\upLn - \bsfF_{ij}\upHn) \in
\Real^m$.
The \emph{convex limiting} technique to be explained in the next
section relies heavily on~\eqref{PreFCT}.  Note that (by construction)
we have that $\bsfA_{ij}^n = -\bsfA_{ji}^n$, which means that
$\sum_{i\in\calV} m_i \bsfU_i\upHnp = \sum_{i\in\calV} m_i
\bsfU_i\upLnp$;
that is to say, the high-order and the low-order solution have the
same mass whether the source term $\bS$ is present or not.

\subsection{Convex limiting}
Without loss of generality, we consider a family of quasiconcave
functionals $\{\Psi_i \}_{i \in \calV}$, $\Psi_i: \calB \to \Real$
where $\calB \subset \Real^m$ is a convex set and
$\Psi_i(\bsfU_i\upLnp)\ge 0$ for each $i \in \calV$. Or goal is to
modify the high-order update so that the modified high-order update
satisfies the same quasiconcave constraints as the low-order solution
and has the same mass as the high-order update.

Taking inspiration from the flux-corrected transport methodology, we introduce symmetric
limiting parameters $\limiter_{ij} = \limiter_{ji} \in [0,1]$,
$i,j \in\calV$, and we define the limited solution $\bsfU_i^{n+1}$ as
follows:
\begin{align}\label{RegularFCT}
m_i \bsfU_i^{n+1} := m_i \bsfU_i\upLnp 
+ \sum_{j \in \calI(i)\backslash\{i\}} \limiter_{ij} \bsfA_{ij}^n.
\end{align}
Notice that $\bsfU_i^{n+1} = \bsfU_i\upLnp$ if
$\limiter_{ij} = 0$ for all $j \in\calI(i)\backslash\{i\}$ and
$\bsfU_i^{n+1} = \bsfU_i\upHnp$ if $\limiter_{ij} = 1$ for all
$j \in\calI(i) \backslash\{i\}$; hence, $\Psi_i(\bsfU_i^{n+1})\ge 0$ when
$\limiter_{ij} = 0$.  Our goal is to find a set of coefficients
$\ell_{ij}$ as close to $1$ as possible so that
$\Psi_i(\bsfU_i^{n+1})\ge 0$.  
\begin{lemma}[Conservation] The limiting process is conservative for any choice 
of coefficients $\ell_{ij}$ if $\ell_{ij} = \ell_{ji}$ for any $j\in \calI(i){\setminus}\{i\}$.  
\end{lemma}

\begin{proof}
  the skew-symmetry of $\bsfA_{ij}^n$ together with the symmetry of
  the limiter $\limiter_{ij}$ implies that
  $\sum_{i\in\calV}\sum_{j \in \calI(i)\backslash\{i\}} \limiter_{ij} \bsfA_{ij}^n
  =\bzero$;
  therefore
  $\sum_{i \in \vertind}m_i \bsfU_i^{n+1} = \sum_{i \in
    \vertind}m_i\bsfU_i\upLnp$.
\end{proof}

The expression \eqref{RegularFCT} goes back to the flux-corrected transport framework
pioneered by \cite{Boris_books_JCP_1973,Zalesak_1979}. The reader can
further explore some current developments for flux-corrected transport methods in the books
\cite{KuzminLoehnerTurek2004, KuzminLoehnerTurek2012}.  At this point
we depart from the existing flux-corrected transport literature and follow \citep{GuerNazPopTom2017} instead.
We rewrite \eqref{RegularFCT} as follows:
\begin{align}
\label{convFCT}
\bsfU_i^{n+1} = \sum_{j \in \calI(i)\backslash\{i\} } \lambda_j (\bsfU_i\upLnp +\limiter_{ij}\bsfP_{ij}^n),
\qquad \text{with} \qquad \bsfP_{ij}^n := \frac{1}{m_i \lambda_j} \bsfA_{ij}^n,
\end{align}
where $\{\lambda_j\}_{j \in \calI(i)\backslash\{i\}}$ is any set of
strictly positive convex coefficients (see Remark
\ref{Rem:ConvexCoeff}), \ie
$\sum_{j \in \calI(i)\backslash\{i\}} \lambda_j = 1$, $\lambda_j > 0$
for all $j \in \calI(i)\backslash\{i\}$. The following two lemmas should convince the
reader that it is possible to estimate $\ell_{ij}$ efficiently by
doing one-dimensional line-searches only.

\begin{lemma} \label{Lem:quasiconcave} Let
  $\Psi_i(\bu):\calB \rightarrow \Real$ be a quasiconcave
  function. Assume that the limiting parameters
  $\limiter_{ij} \in [0,1]$ are such that
  $\Psi_i (\bsfU_i\upLnp + \limiter_{ij} \bsfP_{ij}^n) \geq 0$, for all
  $j \in \calI(i)\backslash\{i\}$, then the following inequality holds
  true:
\begin{align*}
\Psi_i \bigg(\sum_{j \in \calI(i)\setminus\{i\}} 
\lambda_j (\bsfU_i\upLnp + \limiter_{ij} \bsfP_{ij}^n) \bigg) \geq 0 .
\end{align*}
\end{lemma}

\begin{proof} 
  Let $L_0(\Psi_i):=\{\bsfU \in\calB \st \Psi_i(\bsfU)\ge 0 \}$. By definition
  all the limited states
  $\bsfU_i\upLnp + \limiter_{ij} \bsfP_{ij}^n$ are in $L_0(\Psi_i)$ for all
  $j \in \calI(i)\backslash\{i\}$.  Since $\Psi_i$ is quasiconcave,
  the upper contour set $L_0(\Psi_i)$ is convex. Hence, the convex
  combination
  $\sum_{j \in \calI(i)\setminus\{i\}} \lambda_j (\bsfU_i\upLnp +
  \limiter_{ij} \bsfP_{ij}^n)$
  is in $L_0(\Psi_i)$, \ie
  $\Psi_i \big(\sum_{j \in \calI(i)\setminus\{i\}} \lambda_j
  (\bsfU_i\upLnp + \limiter_{ij} \bsfP_{ij}^n) \big) \geq 0$,
  which concludes the proof.
\end{proof}

\begin{theorem} \label{Thm:compute_lij} For every $i \in \vertind$ and
  $j\in \calI(i)$, let $\limiter^i_j$ be defined by
\begin{equation}\label{Eq:Thm:compute_lij}
\limiter^i_j=
\begin{cases}
1 &\text{if } \Psi_i(\bsfU_i\upLnp + \bsfP_{ij}^n)\ge 0, \\
\max\{\limiter \in [0,1] \st \Psi_i(\bsfU_i\upLnp +
  \limiter \bsfP_{ij}^n)\ge 0\} & \text{otherwise}.
\end{cases}
\end{equation}
The following two
statements hold true: \textup{(i)}
$\Psi_i(\bsfU_i\upLnp + \limiter \bsfP_{ij}^n)\ge 0$ for every
$\limiter \in [0,\limiter^i_j]$; \textup{(ii)} Setting
$\limiter_{ij} = \min(\limiter^i_j, \limiter^j_i)$, we have
$\Psi_i(\bsfU_i\upLnp + \ell_{ij}\bsfP_{ij}^n)\ge 0$ and
$\limiter_{ij}=\limiter_{ji}$.
\end{theorem}
\begin{proof} \textup{(i)} First, if
  $\Psi_i(\bsfU_i\upLnp + \bsfP_{ij}^n)\ge 0$ we observe that
  $\Psi_i(\bsfU_i\upLnp + \limiter \bsfP_{ij}^n)\ge 0$ for any
  $\limiter\in [0,1]$ because $\bsfU_i\upLnp\in L_0(\Psi_i)$,
  $\bsfU_i\upLnp + \bsfP_{ij}^n\in L_0(\Psi_i)$ and $L_0(\Psi_i)$
  is convex. Second, if $\Psi_i(\bsfU_i\upLnp + \bsfP_{ij}^n)< 0$,
  we observe that quasiconcavity implies that $\limiter^i_j$ is
  uniquely defined since the segment $\{\bsfU_i\upLnp + \ell\bsfP_{ij}^n\st \ell\in [0,1]\}$ can cross
the level set $\partial\{\Psi_i(\bsfU)\ge 0\}$ only once; moreover, for any $\limiter\in [0,\limiter^i_j]$
  we have $\Psi_i(\bsfU_i\upLnp + \limiter \bsfP_{ij}^n)\ge 0$
  because $\bsfU_i\upLnp\in L_0(\Psi_i)$,
  $\bsfU_i\upLnp + \limiter^i_j \bsfP_{ij}^n\in L_0(\Psi_i)$ and
  $L_0(\Psi_i)$ is convex. \textup{(ii)} Since
  $\limiter_{ij} = \min(\limiter^i_j, \limiter^j_i)\le \limiter^i_j$,
  the above construction implies that
  $\Psi_i(\bsfU_i\upLnp + \ell_{ij}\bsfP_{ij}^n)\ge 0$. Note
  finally that
  $\limiter_{ij} = \min(\limiter^i_j, \limiter^j_i) = \limiter_{ji}$.
\end{proof}

\begin{remark}[Choice of convex coefficients]\label{Rem:ConvexCoeff}
There are infinitely many possible choices for the strictly positive
convex coefficients $\{\lambda_j\}_{j \in \calI(i){\setminus}\{i\}}$
in \eqref{convFCT}.  Note that it is even possible to choose a
different set $\{\lambda_j\}_{j \in \calI(i){\setminus}\{i\}}$ for
each $i \in \calV$ without affecting the results presented in this
paper. We have not made any theoretical attempt to exploit these
additional degrees of freedom in order to optimize the convex
limiting technique. All the computations reported
in \cite{GuerNazPopTom2017} have been done with the simplest choice
$\lambda_j := \frac{1}{\text{card}(\calI(i))-1}$ for all
$j\in\calI(i){\setminus}\{i\}$ for all $i \in \calV$. Other
choices have been explored computationally but none turned out to be more
efficient than the others.  It might be interesting though to
explore this question further; for instance, other choices of
convex coefficients could help preserve some symmetries.
\end{remark}

\begin{remark}[Multiple limiting]\label{rem:multlim}
  In general we have to consider families of quasiconcave functionals
  $\{\{\Psi_i\}_{i \in \calV}\}_{l\in \calL}$,
  $\Psi_i^l: \calB^l \to \Real$, where $\calB^l \subset \Real^m$ is
  the convex admissible set of the functional $\Psi_i^l$.  The list
  $\calL$ describes the nature of the functionals; this list could
  encompass any of the functionals shown in
  Examples~\ref{Example:density} to \ref{Example:kinetic_energy}.  The
  list $\calL$ is sometimes ordered in the sense that
  $\calB^{l'}\subset \calB^l$ if $l'\ge l$. Let us illustrate this
  concept with the compressible Euler equations. Usually one starts
  with $B^1=\Real^m$ to enforce a local minimum principle on the
  density (which implies positivity of the density).
We can also take $\calB^2=\Real$ to enforce a local maximum principle on the density 
by using $\Psi(\bsfU)=-\rho$. Then we can consider
  $B^3=\{\bsfU\in\calB^1\st \rho>0\}$ to enforce a local minimum
  principle on the (specific) internal energy (which implies
  positivity of the (specific) internal energy). We finally set
  $B^4=\{\bsfU\in \calB^2 \st e(\bsfU)>0\}$ to enforce a local minimum
  principle on the specific entropy. 
\end{remark}

The following result is the main conclusion of the paper.
\begin{theorem} \label{Thm:limiting_and_invariant_domain}
  Let $\{\Psi^l: \calB^l \to \Real\}_{l\in \calL}$, be a family of
  quasiconcave functionals, where the sets $\calB^l\subset \Real^m$
  are convex for all $l\in\calL$.  Let
  $\calB:\{\bsfU\in \Real^m \st \Psi^l(\bsfU)\ge 0, \ \forall
  l\in\calL\}$.  Let $n\ge 0$. Assume that
  $\min_{i\in\calV}(1+4\frac{d_{ii}\upLn}{m_i})\ge 0$ and
  $\dt\le 2 \dt_0$. Consider the quasiconcave functionals
  $\{\Psi_i^l\}_{i\in\calV,l\in\calL}$ defined by
  $\Psi_i^l(\bsfU)=\Psi^l(\bsfU)-\Psi_i^{l,\min}$ with
  $\Psi_i^{l,\min}$ defined in~\eqref{def:convexBounds}. Let
  $\limiter_{j}^{i,l}$ be the limiter computed by
  using~\eqref{Eq:Thm:compute_lij} for any $i\in\calV$, $j\in\calI(i){\setminus}\{i\}$, $l\in\calL$. Let
  $\limiter_{ij}=\min(\min_{l\in\calL} \limiter_{j}^{i,l},
  \min_{l\in\calL} \limiter_{i}^{j,l})$. Let $\bsfU_i^{n+1}$ be defined
  in~\eqref{convFCT}. Assume that $\calB$ is an invariant set for
  \eqref{def:hyperbolic_system}, then $\calB$ is an invariant domain, \ie
  $(\bsfU_i^n\in\calB,\ \forall i\in \calV)\Rightarrow (\bsfU_i^{n+1}\in\calB,\
  \forall u\in\calV)$.
\end{theorem}

\begin{proof} 
  Notice first that $\calB$ is convex since it is the intersection of
  convex sets
  $\calB=\bigcap_{l\in\calL}\{\bsfU\in\Real^m\st \psi^l(\bsfU)\ge
  0\}$.
  Since $\calB$ is a convex invariant set for
  \eqref{def:hyperbolic_system}, the CFL assumption together with
  Theorem~\ref{Thm:UL_is_invariant} implies that
  $\bsfU_i\upLnp\in \calB$ for all $i\in\calV$. Then
  Theorem~\ref{Thm:compute_lij} can be applied because
  $\Psi_i^l(\bsfU\upLnp)\ge 0$. This theorem then implies that
  $\Psi^l(\bsfU_i^{n+1})\ge \Psi_i^{l,\min}$ for all $\l\in \calL$.
  Moreover, $\bsfU_i^n+2\dt\bS(\bsfU_i^n) \in\calB$ and
  $\overline\bsfU_{ij}^n\in\calB$, then owing to the CFL assumption
  and definition \eqref{def:convexBounds}, this implies that
  $\Psi_i^{l,\min} \ge 0$. In conclusion $\Psi^l(\bsfU_i^{n+1})\ge 0$
  for all $\l\in \calL$, which implies that $\bsfU_i^{n+1}\in \calB$.
\end{proof}

\begin{remark}[SSP extension]
Owing to remark~\ref{Rem:Tensor_structure_of_B}, Theorem~\ref{Thm:limiting_and_invariant_domain}
extends to any SSP RK time stepping provided the limiting is done at the end of each 
elementary forward Euler substep.
\end{remark}

\subsection{Implementation details}

The objective of this section to give further details on the convex
limiting technique introduced above in order to help the reader to
implement it.

\subsubsection{Pseudocode of the limiting algorithm}
Given a set of quasi-convex functionals $\{\Psi_i\}_{i \in \vertind}$,
$\Psi_i:\calB\rightarrow \Real$, such that
$\Psi_i(\bsfU_i\upLnp) \geq 0$ with convex set $\calB$,
Algorithm~\ref{PseudoCodeAbstract} enforces the quasi-concave
constraints $\Psi_i(\bsfU_i^{n+1}) \geq 0$ for each $i \in \vertind$.
This pseudocode attempts to reflect as accurately as possible the way
convex limiting is coded in practice. Basically, convex limiting is
done in two loops over the set of the global degrees of freedom
$\vertind$: the first loop (lines \ref{beg1} to \ref{end1}) computes
the matrix $\limiter^i_j$ in general non-symmetric form; the second
loop (lines \ref{beg2} to \ref{end2}) computes the final symmetric
limiter $\limiter_{ij}$. Lemma \ref{Lem:quasiconcave} explains why the
limiters $\limiter^i_j$ estimated in the first loop are large enough
to enforce the constraint $\Psi_i(\bsfU_i^{n+1}) \geq 0$ for each
$i \in \vertind$.  Theorem \ref{Thm:compute_lij} explains why the
symmetrization (shrinkage) of the limiters done in the second loop
still produces limiters compatible with these constraints.  We have
found that initializing $\limiter_i$ with the
lines~\ref{beg_better_init}--\ref{end_better_init} instead of setting
$\limiter_i=1$ reduces the number of times the line-search in
line~\ref{line_l_i_j} is executed.  

 \begin{algorithm}
\caption{Convex Limiting}
\label{PseudoCodeAbstract}
\begin{algorithmic}[1]
\For{$i \in \vertind$} \label{beg1}
\If{$\Psi_i(\bsfU_i\upHnp )\ge 0$} \label{beg_better_init}
\State $\limiter_i := 1$ 
\Else
\State $\limiter_i := \max\{\limiter \in [0,1] \st 
\Psi_i(\bsfU_i\upLnp + \limiter  (\bsfU_i\upHnp - \bsfU_i\upLnp)  )\ge 0\}$
\label{line_l_i_j_init}
\EndIf \label{end_better_init}
\For{$j \in \calI(i)\backslash\{i\} $} 
\If{$\Psi_i(\bsfU_i\upLnp + \limiter_i\bsfP_{ij}^n)\ge 0$}
\State $\limiter^i_j := \limiter_i$ 
\Else
\State $ \limiter^i_j := \max\{\limiter \in [0,\limiter_i] \st \Psi_i(\bsfU_i\upLnp  
+ \limiter\,  \bsfP_{ij}^n)\ge 0\}$ \label{line_l_i_j}
\EndIf
\EndFor
\EndFor \label{end1}
\For{$i \in \vertind$} \label{beg2}
\For{$j \in \calI(i)\backslash\{i\} $}
\State $\limiter_{ij} := \min\{\limiter^i_j,\limiter^j_i\}$ 
\EndFor
\EndFor \label{end2}
\end{algorithmic}
\end{algorithm}



\subsubsection{Transforming $\Psi_i(\bsfU)\ge 0$ into a quadratic constraint}\label{sec:polratio} 
As mentioned in the previous subsection, the line-search invoked in
line~\ref{line_l_i_j_init} and line~\ref{line_l_i_j} of
Algorithm~\ref{PseudoCodeAbstract} could be computationally
expensive. However, it happens sometimes that the constraint of
interest $\Psi_i(\bsfU)\ge 0$ can be transformed into
$\tilde\Psi_i(\bsfU)\ge 0$ where $\tilde\Psi_i$ is a quadratic
function, not necessarily quasi-concave. In this case it is possible
to design a very efficient algorithm for the line-search.

\begin{example}[Internal energy] \label{Ex:Rhoe} To illustrate the
  above statment, let us consider the compressible Euler equations
  with some arbitrary equation of state.  Let us set
  $\calB=\{\bsfU:=(\rho,\bbm,E)^\trans \st \rho>0\}$,
  $\varepsilon(\bsfU) := E - \frac{|\bbm|_{\ell^2}^2}{2\rho}$
  (internal energy), and
  $\Psi_i(\bsfU) := \varepsilon(\bsfU) - \varepsilon_i^{\min}$.  We
  have seen in Example~\ref{Ex:internal_energy} that
  $\Psi_i:\calB \to \Real$ is quasiconcave (actually
  $\Psi_i:\calB \to \Real$ is concave). It is clear that one has
  $\Psi_i(\bsfU)\ge 0$ iff
  $\tilde{\Psi}_i(\bsfU) := \rho\varepsilon(\bsfU) - \rho
  \varepsilon_i^{\min} \ge 0$
  for all $\bsfU\in \calB$. Notice that
  $\rho\varepsilon(\bsfU) = E \rho - \frac{1}{2}\bbm^2$ and
  $\rho \varepsilon_i^{\min}$ are quadratic polynomials of the
  conserved variables; hence, $\tilde{\Psi}_i(\bsfU)$ is quadratic
  (but a simple computation shows also that $\tilde{\Psi}_i$ is not
  quasiconcave). In conclusion, instead of doing the line-search 
  with $\Psi_i(\bsfU) := \varepsilon(\bsfU) - \varepsilon_i^{\min}$,
  one can do the line-search with the quadratic functional
  $\tilde{\Psi}_i(\bsfU)= \rho\varepsilon(\bsfU) - \rho
  \varepsilon_i^{\min}$.
\end{example}

We now state an abstract result that formalizes the above observation.

\begin{lemma}\label{lem:parablinesearch} Let
  $\Psi:\calB\subset \Real^m \to \Real$.  Let $\bsfU\upL\in \calB$ and
  assume that $\Psi(\bsfU\upL) \geq 0$.  Let $\tilde\Psi:\calB\to \Real$,
  let $\bsfP\in\Real^m$, and assume that there is
  $\limiter^{\max}\in [0,1]$ such that
  $\Psi(\bsfU\upL + \limiter \bsfP)\ge 0$ iff
  $\tilde\Psi(\bsfU\upL + \limiter \bsfP)\ge 0$ for all
  $\limiter\in[0,\limiter^{\max}]$.  Assume that $\tilde\Psi$ is
  quadratic and let
  $a :=\frac{1}{2}\bsfP^\trans \mathrm{D}^2\tilde\Psi\bsfP$,
  $b := \mathrm{D}\tilde\Psi(\bsfU\upL) \!\cdot \!\bsfP$ and $c := \tilde\Psi(\bsfU\upL)$.
  Let $\ell^{\min}$ be the smallest positive root of the equation
  $a \ell^2 + b \ell + c=0$, with the convention that
  $\ell^{\min} :=1 $ if the equation has no positive root. Let
  $\limiter_j^{i} := \min(\ell^{\min},\ell^{\max})$, then
  $\Psi(\bsfU\upL + \limiter \bsfP)\ge 0$ for all
  $\limiter\in [0, \limiter_j^{i}]$.
\end{lemma}
\begin{proof}
  Let us first observe that
  $\tilde\Psi(\bsfU\upL + \limiter \bsfP) = a\ell^2 + b\ell +
  c\ell=:g(\ell)$ for all $\ell\in [0,\ell^{\max}]$; hence,
  $\Psi(\bsfU\upL + \limiter \bsfP)\ge 0$ iff $g(\ell)\ge 0$ for all
  $\ell\in [0,\ell^{\max}]$.  If there is no positive root to the
  equation $a \ell^2 + b \ell + c =0$, then the sign of $g(\ell)$ over
  $[0,\infty)$ is constant. The assumption $g(0)=c:=\Psi(\bsfU\upL)\ge 0$, implies that
  $g(\ell)\ge 0$ for all $\ell\in [0,\infty)$. That is,
  $\Psi(\bsfU\upL + \limiter \bsfP)\ge 0$ for all
  $\ell\in [0,\ell^{\max}]$, and in particular this is true for all
  $\limiter\in [0, \limiter_j^{i}]$ since in this case
  $\limiter_j^{i} := \min(\ell^{\min},\ell^{\max})\le \ell^{\max}$.
  Otherwise, if there is at least one positive root to the equation
  $g(\ell)=0$, then denoting by $\ell^{\min}$ the smallest positive
  root, we have $g(\ell) \ge 0$ for all $\ell\in[0, \ell^{\min}]$ (if
  not, there would exist $\ell_1\in (0,\ell^{\min})$ s.t.
  $g(\ell_1)< 0$ and the intermediate value theorem would imply the
  existence a root $\ell^*\in(0,\ell_1)$ which contradicts that
  $\ell^{\min}$ is the smallest positive root).  This argument implies
  again that $\Psi(\bsfU\upL + \limiter \bsfP )\ge 0$ for all
  $\limiter\in[0,\limiter_j^{i}]$.
\end{proof}

\begin{example}[Kinetic energy] \label{Ex:Kinetic_energy}
  Coming back to the compressible Euler equations or the shallow water
  equations, the above technique can be applied to enforce the local
  maximum principle on the kinetic energy $\Psi_i(\bsfU)\ge 0$, with
  $\Psi_i(\bsfU) = \Psi(\bsfU)-\Psi_i^{\min}$ and
  $\Psi(\bsfU) =-\frac12 \rho^{-1} \bbm^2$ with
  $\calB=\{\bsfU:=(\rho,\bbm,E)\tr \st \rho>0\}$. (Notice that
  because of the sign convention $\Psi_i^{\min}$ is the \emph{maximum}
  of the kinetic energy over the states
  $\{\overline\bsfU_{ij}^n\}_{j\in\calI(i)}$ and the state
  $\bsfU_i^n+2\dt \bS(\bsfU_i^n)$. Hence the constraint
  $\Psi_i(\bsfU)\ge 0$ amounts to enforcing a local maximum principle
  on the kinetic energy.) We have shown in
  Example~\ref{Example:kinetic_energy} that $\Psi_i$ is quasiconcave.
  In this case Lemma~\ref{lem:parablinesearch} can be applied with the
  functional $\tilde\Psi_i(\bsfU) =\rho \Psi_i(\bsfU) =-\frac12\bbm^2 - \rho\Psi_i^{\min}$
  which is clearly quadratic. 
Note that  $\tilde\Psi_i(\bsfU\upL+\limiter \bsfP)\ge 0$
iff $\Psi_i(\bsfU\upL+\limiter \bsfP)\ge 0$ provided $\rho(\bsfU\upL+\limiter \bsfP)\ge 0$.
Hence before applying Lemma~\ref{lem:parablinesearch}, one must 
compute the limiter $\limiter^{\max}$, which depends on 
$\bsfU\upL$ and $\bsfP$, such that $\rho(\bsfU\upL+\limiter \bsfP)\ge 0$ for all
$\limiter\in [0,\limiter^{\max}]$.
 This technique has been introduced in
  \citep[\S6.4]{Guermond_Quezada_Popov_Kees_Farthing_2018} in the
  context of the shallow water equations.
\end{example}

\begin{remark}[Parameter $\ell^{\max}$] 
  The purpose of the parameter $\ell^{\max}$ appearing in the
  statement of Lemma~\ref{lem:parablinesearch} is to ascertain that
  stating that $\Psi(\bsfU + \limiter \bsfP)\ge 0$ is equivalent to
  stating that $\tilde\Psi(\bsfU + \limiter \bsfP)\ge 0$ for all
  $\limiter\in[0,\limiter^{\max}]$. The limiter $\limiter^{\max}$ depends on 
$\bsfU\upL$ and $\bsfP$ and  must be computed before
applying Lemma~\ref{lem:parablinesearch}; see Example~\ref{Ex:Kinetic_energy}.
\end{remark}

\subsubsection{Transforming $\Psi_i(\bsfU)\ge 0$ into a concave constraint}
It is sometimes possible to transform a quasiconcave constraint into a
concave constraint. This type of transformation is useful, since
designing efficient and robust line-search procedures for general
quasiconcave functionals is not a trivial task, whereas it is always
possible to use the Newton-secant algorithm presented in
\S\ref{Sec:Newton_secant} for concave functionals.

For instance, let $\Psi:\calB\to \Real$ be a
quasiconcave function, then referring to
Lemma~\ref{Lem:quasiconcavity}, it is sometimes possible to find
$R:\calB\to (0,\infty)$, positive and convex, such that $R\Psi$ is
concave. This is indeed the case for any ``specific'' entropy as
described in Example~\ref{Example:specific_entropy}. 
The following lemma formalizes this observation.

\begin{lemma} \label{Lem:quasiconcave_to_concave} Let
  $\calB\subset \Real^m$ be a convex set. Let $\Psi:\calB \to \Real$
  and $R:\calB\to (0,\infty)$.  Assume that
  $\Phi:=R\Psi : \calB \to \Real$ is concave.  Let
  $\bsfU\upL\in \calB$ and assume that $\Psi(\bsfU\upL)\ge 0$.  Let
  $\bsfP\in \Real^m$ and let $\ell_{\max}\in [0,1]$ be such that
  $\bsfU\upL+\ell\bsfP\in\calB$ for all $\ell\in
  [0,\ell^{\max}]$.
  $\Psi^{\min}\in \Real$.  Assume that either (i) $R$ is affine or
  (ii) $\Psi^{\min}\ge 0$ and $R$ is convex.  Then the following
  statements hold true:
\begin{enumerate}[(i)] 
\item $\Psi(\bsfU\upL+\ell\bsfP) - \Psi^{\min} \ge 0$ iff
  $\Phi(\bsfU\upL+\ell\bsfP) -\Psi^{\min} R(\bsfU\upL+\ell\bsfP)\ge 0$ for all
  $\ell\in [0,\ell^{\max}]$;
\item the map
  $[0,\ell^{\max}]\ni \ell\mapsto \Phi(\bsfU\upL+\ell\bsfP) -\Psi^{\min}
  R(\bsfU\upL+\ell\bsfP)\in \Real$ is concave.
\end{enumerate}
\end{lemma}
\begin{proof}
  (i) Since $\bsfU\upL+\ell\bsfP\in \calB$ for all
  $\ell\in [0,\ell^{\max}]$, we infer that $R(\bsfU\upL+\ell\bsfP)>0$ for
  all $\ell\in [0,\ell^{\max}]$. Hence, the first assertion is a
  consequence of the assumption $R(\bsfU\upL+\ell\bsfP)>0$ for all
  $\ell\in [0,\ell^{\max}]$. (ii) Observe that
  $-\Psi^{\min}R:\calB \to \Real$ is concave if $R:\calB \to \Real$ is
  affine. Observe also that that $-\Psi^{\min}R:\calB \to \Real$ is
  concave if $R:\calB \to \Real$ is convex and
  $\Psi^{\min}\ge 0$. Hence the second
  assertion is just a consequence of the concavity of $\Phi: \calB \to \Real$.
\end{proof}

\begin{example}[Specific
  entropy] \label{Ex:specific_entropy_becomes_concave} Let us
  illustrate the use of Lemma~\ref{Lem:quasiconcave_to_concave} with
  the compressible Euler equations.  Assume to simplify the argument
  that the equation of state is the $\gamma$-law.  Consider the
  physical specific entropy
  $\Psi(\bsfU) =
  \frac{1}{\gamma-1}\log(\varepsilon(\bsfU)\rho^{-\gamma})$
  and the quasiconcave constraint $\Psi(\bsfU) - \Psi_i^{\min}\ge 0$.
  Line-searches for this quasiconcave functional may be delicate
  (lines \ref{line_l_i_j_init} and \ref{line_l_i_j} in Algorithm
  \ref{PseudoCodeAbstract}), not only because it is not strictly
  concave, but also because of the presence of the logarithm.  We have
  seen in Example~\ref{Ex:change_of_variable_specific_entropy} that
  this constraint can be transformed into another quasiconcave
  constraint $\tilde\Psi(\bsfU)-\tilde\Psi_i^{\min}\ge 0$ with
  $\tilde\Psi(\bsfU):=\varepsilon(\bsfU)\rho^{-\gamma} =
  \exp((\gamma-1)\Psi(\bsfU))$.
  Let us assume that the solution at the previous time step
  $\bsfU^{n}$ is such that $\tilde\Psi_i^{\min}\ge 0$ for all
  $i \in \calV$, which is reasonable since it requires the internal
  energy and the density to be nonnegative at $t^n$.  Then using
  $R(\bsfU)= \rho^{\gamma}$, which is convex over
  $\calB=\{\bsfU\st \rho>0\}$, using that
  $R(\bsfU)\tilde\Psi(\bsfU)=\varepsilon(\bsfU)$ is concave, and
  $\tilde\Psi_i^{\min}\ge 0$, and invoking
  Lemma~\ref{Lem:quasiconcave_to_concave}, we finally transform
  (again) the above quasiconcave constraint into the concave
  constraint
  $\varepsilon(\bsfU) - \rho^\gamma \tilde\Psi_i^{\min}\ge 0$.  Notice
  in passing that, for the $\gamma$-law, enforcing positivity of the
  density and the above local minimum principle on the
  specific entropy
  ($\varepsilon(\bsfU) - \rho^\gamma \tilde\Psi_i^{\min}\ge 0$)
  guarantees positivity of the internal energy.
\end{example}

The parameter $\ell^{\max}$ appearing in the statement of
Lemma~\ref{Lem:quasiconcave_to_concave} arises naturally when one
performs convex limiting for more than one functional. More precisely,
before applying \eqref{Lem:quasiconcave_to_concave} one must sure that
$\bsfU\upL+\ell\bsfP\in\calB$ for all $\ell\in [0,\ell^{\max}]$ by
convex limiting so that $R(\bsfU\upL+\ell\bsfP)>0$.  For instance, in
the setting of Example~\ref{Ex:specific_entropy_becomes_concave}, the
parameter $\ell^{\max}$ is the limiter that must be computed to
ascertain that the density of the state $\bsfU\upL+\ell \bsfP$ is
positive over the interval $[0,\ell^{\max}]$.

\subsubsection{Line-search: The Newton-secant solver} \label{Sec:Newton_secant}

Unless the function $g(\ell) := \Psi_i(\bsfU\upLnp + \ell \bsfP_{ij}^n)$ has a
special structure (say, linear or quadratic), the line-searches
invoked at lines \ref{line_l_i_j_init} and \ref{line_l_i_j} in
Algorithm~\ref{PseudoCodeAbstract} require the use of an iterative
procedure.  Without claiming originality, we now show how the
line-searches can be done by using the Newton-secant algorithm to guarantee
that $\Psi_i(\bsfU\upL + \limiter_j^i \bsfP_{ij}^n) \ge 0$ independently of the
tolerance that is given to the algorithm to estimate $\limiter_j^i$.
 
Let us assume that $g(\ell) \in \mathcal{C}^2([0,1];\Real)$ is
strictly concave and $g(0) > 0$. Let us set $\limiter_l^0=0$. Let us
assume also that there exists $\limiter_r^0\in(0,1]$ such that
$g(\limiter_r^0)<0$.  Hence there exists a unique number
$\limiter^*\in(\limiter_l^0,\limiter_r^0)$ such
$g(\limiter_l^0)>g(\limiter^*)=0>g(\limiter_r^0)$. Our goal is now to
estimate iteratively $\limiter^*$ from below, up to some fixed
tolerance.  Notice that in this particular setting Newton's algorithm
converges from above; that is, Newton's algorithm will always return an
approximate value of $\limiter^*$ that is larger than $\limiter^*$, (unless $g$ is
quadratic). The following lemma describes an iterative process $(\limiter_l^k,\limiter_r^k)\to
(\limiter_l^{k+1},\limiter_r^{k+1})$, $k\ge 0$, such that 
\[
\limiter_l^0<\ldots<\limiter_l^k< \limiter_l^{k+1}<\ldots\le \limiter^* \le
\ldots < \limiter_r^{k+1}<\limiter_r^{k}<\ldots< \limiter_r^{0}
\]


\begin{lemma}[One iteration update]\label{NSorder} 
  Let $\limiter_l^k<\limiter_r^k$. Let
  $g\in C^2([\limiter_l^k,\limiter_r^k];\Real)$.  Assume that
  $g''(\limiter)<0$ for all
  $ \limiter\in [\limiter_l^k,\limiter_r^k]$.  Assume that
  $g(\ell_l^{k}) > 0$ and $g(\ell_r^{k}) < 0$.
\begin{enumerate}[(i)]
\item Let $s_l^k := \frac{g(\ell^{k,r}) - g(\ell^{k,l})}{\ell^{k,r} - \ell^{k,l}}$ and
$s_r^k:= g'(\ell^{k,r})$. Then $s_l^k<0$ and $s_r^k<0$. 
\item Let $\limiter_l^{k+1}$
and  $\limiter_r^{k+1}$ be defined by 
\begin{align*}
\ell_l^{k+1} := \ell_l^{k} - \frac{g(\ell_l^{k})}{s_l^{k}},\qquad
\ell_r^{k+1} := \ell_r^{k} - \frac{g(\ell_r^{k})}{s_r^{k}}.
\end{align*}
Then $\ell^{k}_l<\ell_l^{k+1} < \ell^{*} < \ell_r^{k+1}< \limiter_r^{k}$.
\end{enumerate}
\end{lemma}

\begin{proof} The inequalities $\ell^{k}_l<\ell_l^{k+1} < \ell^{*}$
are standard properties of the secant algorithm. The inequalities 
$\ell^{*} < \ell_r^{k+1}< \limiter_r^{k}$ are standard properties of Newton's algorithm.
The details are left to the reader
\end{proof}

\begin{algorithm}
\caption{Newton-Secant solver}
\label{PseudoNewtSec}
\begin{algorithmic}[1]
\Require $k = 0$, $k_{\max} \geq 1$, $\ell_{l} < \ell_{r}$, 
$g(\ell_{l}) > 0$,  $g(\ell_{r}) < 0$, $\textup{tol} > 0$
\While{$k \leq k_{\max}$ \textbf{and} $\ell_{r} - \ell_{l} > \textup{tol}$} \label{checkstop} 
  \State $k := k + 1$
  \State $\ell_l^{\textup{aux}} := \ell_{l}$
	\If{$g(\ell_{l}) > g(\ell_{r})$} \label{wellpoSec} \label{begSec}
		\State $s_l := \frac{g(\ell_{r}) - g(\ell_{l})}{\ell_{r} - \ell_{l}}$ \Comment{Condition $\ell_{r} - \ell_{l} > 0$ checked in line \ref{checkstop}}
		\State $\ell_{l} := \ell_{l} - \frac{g(\ell_{l})}{s_l}$
	\Else
		\State \textbf{break}
	\EndIf \label{EndSec}
	\If{$\ell_{l} > \ell_{r}$ \textbf{or} $g(\ell_{l}) < 0$ } \label{BegSan1} \Comment{Assumes $g(\ell_{r}) < 0$}
	\State $\ell_{l} := \ell^{l,\textup{aux}}$
	\State \textbf{break}
	\EndIf \label{EndSan1}
	\If{$g'(\ell_{r}) < $} \label{BegNewt}
		\State $\ell_{r} := \ell_{r} - \frac{g(\ell_{r})}{g'(\ell_{r})}$
	\Else
	  \State \textbf{break}
	\EndIf \label{EndNewt}
	\If{$g(\ell_{r}) > 0$} \label{BegSan2} 
	\Comment{Condition $\ell_{r} - \ell_{l} > 0$ will be checked in line \ref{checkstop}}
		\State \textbf{break}
	\EndIf \label{EndSan2}
\EndWhile
\State \Return $\ell_i^j := \ell_{l}$ \label{finalLine}
\end{algorithmic}
\end{algorithm}

In Algorithm~\ref{PseudoNewtSec}, line~\ref{checkstop} checks the
stopping criteria. The ``break'' statements (or ``exit'' statements,
depending on the programming language) force the code out of the
while loop, redirecting the control to Line \ref{finalLine}.
One may reach break statements due to roundoff errors.  Lines
\ref{begSec}--\ref{EndSec} is the secant update (approximation from
the left), while Lines \ref{BegNewt}--\ref{EndNewt} define the Newton
update (approximation from the right). Lines
\ref{BegSan1}-\ref{EndSan1} and \ref{BegSan2}--\ref{EndSan2} are
sanity checks. The  Newton-secant update preserves the order
$\ell^{k}_l<\ell_l^{k+1} < \ell^{*} < \ell_r^{k+1}< \limiter_r^{k}$
(see lemma \ref{NSorder}), however some crossover may occur after some
iterations because of round-off errors (due to the nature of floating-point
arithmetic). Notice that the output of interest is the one
produced by the secant update (see line \ref{finalLine}), since the
output produced by Newton's method violates the inequality that we
want to satisfy.


\begin{remark}[Deficiencies of Newton's method]\label{Rem:NewProb} 
If we assume that $g(\ell)$ is strictly concave over $[0,1]$, which
is the case of interest here, one can construct counterexamples
illustrating that Newton's method can either not converge or produce
an output that violates the bound that we want to enforce. For
instance, if the initial guess $\ell^0 \in [0,1]$ for Newton's
method is such that $\ell^0 > \ell^*$ (\ie $g(\ell^0) < 0$), then
Newton's method produces a sequence $\{\ell^k\}_{k \in \polN}$
satisfying $\ell^* < \ell^k$ for all $k\in \polN$. This implies that
$g(\ell^k) < 0$ for all $k\in \polN$, which is incompatible with the
constraint that we want to satisfy.  On the other hand, if $g$
reaches a maximum at $l_c\in (0,\limiter^*)$ and the initial guess
is such that $\ell^0 \in (0,\limiter_c)$, then the sequence
$\{\ell^k\}_{k \in \polN}$ wanders outside the internal
$[0,1]$. Assuming that $g(\ell)$ is well defined outside $[0,1]$,
the sequence $\{\ell^k\}_{k \in \polN}$ may converge to a negative
solution.
\end{remark}

\begin{remark}[Actual performance] 
The convergence rate of Algorithm~\ref{PseudoNewtSec} is at least
$1.618$ because it combines the second-order Newton method with the
$\frac{\sqrt{5}+1}{2}$-order secant method.  In practice, we have
verified that Algorithm~\ref{PseudoNewtSec} rarely ever requires more
than three iterations to reach tolerances such as
$\textup{tol} = 10^{-10}$ (see \cite{GuerNazPopTom2017}). Most
frequently one exits the loop after reaching machine accuracy error.
\end{remark}

\subsection{Relaxing the bounds} \label{Sec:relaxing_the_bounds} In
general the quantity $\Psi_i^{\min}$ defined in
\eqref{def:convexBounds} is accurate enough to make the limited
high-order solution second-order in the $L^1$-norm in space. But it is
too tight to make the method higher-order or even second-order in the
$L^\infty$-norm in the presence of smooth extrema.  The situation is
even worse when using the specific physical entropy to limit the
high-order solution.  For instance, it is observed in
\cite[\S3.3]{Khobalatte_Perthame_1994} that strictly enforcing the
minimum principle on the specific (physical) entropy for the
compressible Euler equations degrades the converge rate to
first-order; it is said therein that ``It seems impossible to perform
second-order reconstruction satisfying the conservativity requirements
$\ldots$ and the maximum principle on $\varepsilon(\bu)$''. We confirm this observation. 
To recover full accuracy in  the
$L^\infty$-norm for smooth solutions, 
one must relax the bound $\Psi_i^{\min}$. 

To avoid repeating ourselves, we refer the reader to \cite[\S4.7]{GuerNazPopTom2017} where
we explain how the bound $\Psi_i^{\min}$ should be relaxed. In a
nutshell, one proceeds as follows: For each $i\in\calV$, we set
\[
\Delta^2 \Psi_i^{\min} = \frac{1}{\sum_{j\in\calI(i){\setminus}\{i\}} \beta_{ij}}\sum_{j\in\calI(i){\setminus}\{i\}}
\beta_{ij} (\Psi_j^{\min} - \Psi_i^{\min}),
\]
where the coefficients $\beta_{ij}$ are meant to make the computation
linearity-preserving (see Remark~\ref{Rem:linearity_preserving}). Then 
we compute the average 
\[
\overline{\Delta^2 \Psi_i^{\min}} := \frac{1}{2\text{card}(\calI(i))}
\sum_{i\ne j\in \calI(i)} (\frac12\Delta^2 \Psi_i^{\min} +\frac12\Delta^2\Psi_j^{\min}),
\]
and finally relax $\Psi^{\min}_i$ by setting 
\[
 \overline{\Psi^{\min}_i} 
= \max((1-\text{sign}(\Psi^{\min}_i)r_i) \Psi^{\min}_i, \Psi^{\min}_i - |\overline{\Delta^2 \Psi_i^n}|), 
\]
where $r_i = (\frac{m_i}{|\Dom|})^{\frac{1.5}{d}}$. Notice that
$r_i\in (0,1)$. The somewhat ad hoc threshold
$(1-\text{sign}(\Psi^{\min}_i)r_h)$ is never active when the mesh size
is fine enough. This term is just meant to be a safeguard on coarse
meshes. For instance, for the compressible Euler equations, when
$\Psi(\bsfU)$ is either the density (or the internal energy), this
threshold guarantees positivity of the density (or the internal
energy) because in this case $(1-\text{sign}(\Psi^{\min}_i)r_i\ge 0$.
The exponent $1.5$ is somewhat ad hoc; in principle one could take
$r_i = (\frac{m_i}{|\Dom|})^{\frac{\delta}{d}}$ with $\delta<2$.

\bibliographystyle{abbrvnat}

\bibliography{ref}

\end{document}